\documentclass[12pt,reqno]{amsart}
\usepackage{soul}
\usepackage{multicol,amsmath,graphics,cancel,soul,color,bbm,amsfonts,dsfont,tikz,mathrsfs}
\usepackage[left=1in, right=1in, top=1.1in,bottom=1.1in]{geometry}
\usetikzlibrary{matrix,patterns,positioning}
\numberwithin{equation}{section}
\newcommand{\blue}{\textcolor[rgb]{0.00,0.00,0.00}}
\newcommand{\green}{\textcolor[rgb]{0.00,0.00,0.00}}
\newcommand{\Fc}{\mathcal{F}}

\newcommand{\Hc}{\mathcal{H}}

\newcommand{\D}{\mathbb{D}}
\newcommand{\E}{\mathbb{E}}

\newcommand{\N}{\mathbb{N}}
\newcommand{\Pb}{\mathbb{P}}

\newcommand{\R}{\mathbb{R}}

\newcommand{\Z}{\mathbb{Z}}
\newcommand{\Hg}{\mathfrak{H}}

\newcommand{\Sf}{\mathscr{S}}

\newcommand{\lettrequivabien}{\lambda}
\newcommand{\processsymb}{;}

\newcommand{\Norm}[1]{\left\lVert#1\right\rVert}
\newcommand{\Abs}[1]{\left|#1\right|}
\newcommand{\Ip}[1]{\left\langle #1 \right\rangle}

\newcommand{\Indi}[1]{\mathbbm{1}_{#1}}

\newtheorem{thm}{Theorem}[section]

\newtheorem{cor}[thm]{Corollary}

\newcounter{dummy} \numberwithin{dummy}{section}

\newtheorem{Theorem}[dummy]{Theorem}

\newtheorem{Lemma}[dummy]{Lemma}

\newtheorem{Remark}[dummy]{Remark}

\def\1{{L}\hskip -0.21truecm 1}

\begin{document}

\title[Fractional local time approximation]{Approximation of fractional local times: \\ { zero energy and \blue{derivatives}}}
\address{Arturo Jaramillo: Mathematics Research Unit, Universit\'e du Luxembourg, Maison du Nombre 6, Avenue de la Fonte, L-4364 Esch-sur-Alzette, Luxembourg.}
\email{arturo.jaramillogil@uni.lu}
\address{Ivan Nourdin: Mathematics Research Unit, Universit\'e du Luxembourg, Maison du Nombre 6, Avenue de la Fonte, L-4364 Esch-sur-Alzette, Luxembourg.}
\email{ivan.nourdin@uni.lu}
\address{Giovanni Peccati: Mathematics Research Unit, Universit\'e du Luxembourg, Maison du Nombre 6, Avenue de la Fonte, L-4364 Esch-sur-Alzette, Luxembourg.}
\email{giovanni.peccati@uni.lu}

\begin{abstract}

We consider empirical processes associated with high-frequency observations of a fractional Brownian motion (fBm) $X$ with Hurst parameter $H\in (0,1)$, and derive conditions under which these processes verify a (possibly uniform) law of large numbers, as well as a second order (possibly uniform) limit theorem. We devote specific emphasis to the `zero energy' case, corresponding to a kernel whose integral on the real line equals zero. Our asymptotic results are associated with explicit rates of convergence, and are expressed either in terms of the local time of $X$ or of its \blue{derivatives}: in particular, the full force of our finding applies to the `rough range' $0< H < 1/3$, on which the previous literature has been mostly silent. The {\color{black}use of the derivatives} of local times for studying the fluctuations of high-frequency observations of a fBm is new, and is the main technological breakthrough of the present paper. Our results are based on the use of Malliavin calculus and Fourier analysis, and extend and complete several findings in the literature, e.g. by  Jeganathan (2004, 2006, 2008) and Podolskij and Rosenbaum (2018). 

\end{abstract}

\author{Arturo Jaramillo, Ivan Nourdin, Giovanni Peccati}
\date{\today}
\maketitle
%==========================================================================================================================
%Introduction
%==========================================================================================================================

\section{Introduction}\label{sec:intro}
\subsection{Overview} Let $X = \{X_t : t\geq 0\}$ be a fractional Brownian motion (fBm) with Hurst parameter $H\in(0,1)$ (see Section \ref{sec:prelimn} for technical definitions). The aim of this paper is to study the asymptotic behaviour (as $n \to \infty$) of empirical processes derived from the high-frequency observations of $X$, that is, of mappings with the form
\begin{align}\label{eq:generalhighfreq}
t\mapsto  G^{(n)}_t := b_{n}\sum_{i=1}^{\lfloor nt\rfloor}f\big(n^{a}(X_{\frac{i}{n}}-\lettrequivabien)\big), \quad t\geq 0,
\end{align}	
where $f$ is a real-valued kernel, $(a, \lettrequivabien )\in\R_+\times \R$, and $\{b_{n}\}_{n\geq 1}$ is a numerical sequence satisfying $b_{n}\rightarrow0$. Our specific aim is to study the {\it first and second order fluctuations} of such random functions, with specific emphasis on the `rough range' $0 < H < 1/3$ --- see Section \ref{ss:introfrac} for a discussion about the relevance of such a set of values. 

Our approach is based on the use of Malliavin calculus and Fourier analysis, and makes use of the \blue{\it derivatives of the local time of} $X$ (see Section \ref{ref:Appendix}). \blue{The existence of local times derivatives for fBm was first established in \cite[Section~28]{GerHor}, partially building on the findings of \cite{Ber2}: several novel properties of these bivariate random fields are proved in Section 5 of the present paper}. \blue{ Our use of the derivatives of local times of $X$ is close in spirit to the study of derivatives of self-intersection local times, which was initiated by Rosen in \cite{Rosen} for the Brownian motion,} and further developed by Markowsky \cite{MDILTCHAOS} and Jung et al. \cite{JMTANAKA} for the fractional Brownian motion case (see also \cite{GuoHuXiao} for the case of a self-intersection of independent fractional Brownian motions and \cite{Nasy} for related results on the Brownian sheet). The derivatives of the local time for a fractional Ornstein Uhlenbeck process are formally discussed in \cite{GuoXiao}.

We briefly observe that values of the Hurst parameter $H$ in the range $(0,\frac12)$ have recently become relevant for stochastic volatility modelling --- see e.g. \cite{GJR}.

\subsection{Statistical motivations and the semimartingale case}\label{ss:semim}  In the last three decades, the study of processes such as \eqref{eq:generalhighfreq} (for $X$ a generic stochastic process, whose definition possibly depends on $n$) has gained particular traction in the statistical literature, since these random functions emerge both as natural approximations of the local time of $X$, and (after a suitable change of variable in the sum) as scaled version of kernel estimates for regression functions in non-stationary time-series --- see e.g. \cite{AP, Ako, Bor, CW, GNS, Jac, J17, LX, NW, ParPhi, Phi, W, WP} for a sample of the available literature on these tightly connected directions of research. When $X$ is a diffusion process or, more generally, a semimartingale, the fluctuations of $G^{(n)}$ are remarkably well understood: a typical result in such a framework states that, under adequate assumptions, $G^{(n)}$ converges uniformly in probability over compact intervals towards a process of the type $c L_t(\lettrequivabien)$ where $c$ is a scalar and $L_t(\lettrequivabien)$ is the {\it local time} of $X$ at $\lettrequivabien$, up to time $t$. Moreover, when suitably normalized, the difference $G^{(n)}_t - c L_t(\lettrequivabien)$ stably converges (as a stochastic process) towards a mixture of Gaussian martingales. The latter result is particularly useful for developing testing procedures. See e.g. \cite[Theorems 1.1 and 1.2]{Jac} for two well-known representative statements, applying to the case where $X$ is a Brownian semimartingale. {\color{black} {\blue {We also mention} that, in the limit case $a=0$, the statistic $G^{(n)}$ has been used in \cite{Alt} for estimating occupation time functionals for $X$.}}

\subsection{Local times}\label{ss:introlt} Let $X$ be a generic real-valued stochastic process defined on a probability space $(\Omega,\mathcal{F},\Pb)$. We recall that, for $t>0$ and $\lettrequivabien\in\R$, the {\it local time} of $X$ up to time $t$ at $\lettrequivabien$ is formally defined as
\begin{align}\label{e:lt}
L_{t}(\lettrequivabien)
  :=\int_{0}^{t}\delta_{0}(X_{s}-\lettrequivabien)ds,
\end{align}
where $\delta_{0}$ denotes the Dirac delta function. A rigorous definition of $L_{t}(\lettrequivabien)$ is obtained by replacing  $\delta_{0}$ by the Gaussian kernel $\phi_{\varepsilon}(x):=(2\pi\varepsilon)^{-\frac{1}{2}}\exp\{-\frac{1}{2\varepsilon}x^2\}$ and taking the limit in probability as $\varepsilon\rightarrow0$ (provided that such limit exists). The random variable $L_{t}(\lettrequivabien)$ is a recurrent object in the theory of stochastic processes, as it naturally arises in connection with several fundamental topics, such as the extension of It\^o's formula to convex functions, the absolute continuity of the occupation measure of $X$ with respect to the Lebesgue measure, and the study of limit theorems for additive functionals of $X$ --- see \cite{Ber0, Ber, CoNuTu, GerHor, HuOk, RY} for some general references on the subject. It is a well-known fact (see e.g. \cite{Ber0, Ber, GerHor}) that, if $X$ is a fBm with Hurst parameter $H\in (0,1)$, then the local time \eqref{e:lt} exists for every $\lettrequivabien$. Moreover, by Theorems 3.1 and 4.1 in \cite{AyWuYi}, the application $(\lettrequivabien, t) \mapsto L_t(\lettrequivabien)$ admits a jointly continuous modification such that the mapping $t\mapsto L_t(\lettrequivabien)$ is $\mathbb{P}$-a.s. locally $\gamma$-H\"older continuous for every $0<\gamma < 1-H$, and one has the following well-known {\it occupation density formula}: for every Borel set $A\subset \R$ and every $t>0$,
$$
\int_0^t {\bf 1}_A(X_s) ds = \int_A L_t(\lettrequivabien) d\lettrequivabien, \quad {\rm a.s.-} \mathbb{P}.
$$
See also Lemma \ref{Lem:existLocalderiv} below. The functional limit theorems evoked in the previous Section \ref{ss:semim} can be regarded as natural extensions of the classical contributions \cite{Skorohod} by Skorokhod, \cite{ErKac} by Erd\"{o}s and Kac, and \cite{Knight} by Knight, that established the convergence of \eqref{eq:generalhighfreq} towards a scalar multiple of the Brownian local time, in the case where $X$ is either a random walk or a Brownian motion. See also Borodin \cite{Bor}.

\subsection{The fractional case}\label{ss:introfrac} The starting point of our analysis is the influential paper by Jeganathan \cite{Jeg}, focussing on the case where, in \eqref{eq:generalhighfreq}, one has that $X_t = X^n_t := \frac{1}{\gamma_n}{S}_{\lfloor nt \rfloor}$,
with $\{S_k : k\geq 1\}$ a discrete-time process, and $\gamma_n$ a normalising sequence such that $X^n$ converges in distribution to an $\alpha$-stable L\'evy motion, for some $\alpha\in (0,2]$. In the case $\alpha=2$, the results of \cite{Jeg} enter the framework of the present paper: in particular, \cite[Theorem 4]{Jeg} yields that, if $X$ is a fractional Brownian motion with Hurst parameter $H\in(0,1)$ and $f\in L^{1}(\R)\cap L^{2}(\R)$, then, for every $\lettrequivabien\in \R$ and every $t>0$,
\begin{align}\label{eq:Jeg}
n^{H-1}\sum_{i=1}^{\lfloor nt\rfloor}f(n^{H}(X_{\frac{i-1}{n}}-\lettrequivabien))
  &\stackrel{L^{2}(\Omega)}{\longrightarrow}L_{t}(\lettrequivabien)\int_{\R}f(x)dx, \quad n\to \infty.
\end{align}
A continuous version of \eqref{eq:Jeg} can be inferred from \cite[Theorem 5]{Jeg}, stating that, under the exact same assumptions on $f$,
\begin{align}\label{eq:Jeg2}
n^{H-1} \int_0^{nt} f ( n^H (X_{s/n} - \lettrequivabien ))ds
  &\stackrel{L^{2}(\Omega)}{\longrightarrow}L_{t}(\lettrequivabien)\int_{\R}f(x)dx, \quad n\to \infty.
\end{align}
One sees immediately that \eqref{eq:Jeg} and \eqref{eq:Jeg2} imply a trivial conclusion in the case of a `zero-energy' function $f$, that is, when $\int_\R f(x) dx = 0$ (we borrowed the expression `zero-energy' from reference \cite{WP11}, that we find particularly illuminating on the matter). 

A refinement of \eqref{eq:Jeg2} in the zero-energy case was first obtained in \cite[Theorem 1]{Jeg2}, where it is proved that, if 
\begin{equation}\label{e:zop}
\frac13 < H<1, \quad \int_\R\big(|f(x)| +|xf(x)|\big) dx <\infty, \quad \mbox{and  } \int_\R f(x)dx = 0,
\end{equation} 
then, as $n\to\infty$,
\begin{equation}\label{e:j2006}
n^{\frac{H-1}{2}} \int_0^{nt} f ( n^H (X_{s/n} - \lettrequivabien ))ds \Longrightarrow \sqrt{b}\, W_{L_t(\lettrequivabien)},
\end{equation}
where $W$ is a Brownian motion independent of $X$, $b>0$ is an explicit constant (depending on $H$ and $f$), and $\Longrightarrow$ indicates convergence in distribution in the Polish space $C([0,1])$. In the case $\lettrequivabien = 0$, a similar statement can be deduced as a special case of \cite[Theorem 1]{NX} (see also \cite{HNX}), implying that the functional convergence \eqref{e:j2006} continues to hold if $1/3<H< 1$ and $\int_\R | f(x) | |x|^{1/H-1} dx<\infty$. We also notice that the results from \cite{HNX, Jeg, NX} are all extensions of the well-known {\it Papanicolau-Stroock-Varadhan Theorem}, as stated e.g. in \cite[Theorem XIII-(2.6) and Proposition XIII-(2.8)] {RY}. 

An extension of \eqref{eq:Jeg}, in the zero energy case and for $\lettrequivabien = 0$, was obtained in \cite[Theorem 4]{Jeg3}, where it is shown that, if $1/3 < H<1$, $\int_\R(|f(x)|^p +|xf(x)|) dx <\infty$ for $p=1,2,3,4$, and $\int_\R f(x)dx = 0$, then, as $n\to\infty$,
\begin{equation}\label{e:j2008}
n^{\frac{H-1}{2}} \sum_{i=1}^{\lfloor nt\rfloor}f(X_{i-1} ) \stackrel{Law}{=}  n^{\frac{H-1}{2}} \sum_{i=1}^{\lfloor nt\rfloor}f(n^H X_{\frac{i-1}{n}} ) \stackrel{f.d.d.}{\longrightarrow} \sqrt{b}\, W_{L_t(0)},
\end{equation}
where $\stackrel{f.d.d.}{\longrightarrow}$ indicates convergence in the sense of finite-dimensional distributions. Although we did not check the details, it seems reasonable that the finite-dimensional convergence in \cite[Theorem 4]{Jeg3} can be lifted to convergence in the Skorohod space $D([0,T])$, for every $T>0$, and that the conditions on $f$ can be relaxed so that they match \eqref{e:zop}.

We finally observe that the limit result \eqref{eq:Jeg} has been recently extended (in a fully functional setting) in \cite[Theorem 1.1]{PoRo} to the case of sequences of the type  \eqref{eq:generalhighfreq}, where the summand $f(n^{H}(X_{\frac{i-1}{n}}-\lettrequivabien) )$  is replaced by a more general bivariate mapping $f\big( n^{H}( X_{\frac{i-1}{n}} -\lettrequivabien ), n^H (X_{\frac{i}{n}} - X_{\frac{i-1}{n}} ) \big ) $. It is interesting to notice that the arguments used in the proof of \cite[Theorem 1.1]{PoRo} yield that the convergence \eqref{eq:Jeg} also holds uniformly in probability on compact intervals.

% 
%
%
%In view of the convergence \eqref{eq:Jeg}, it is natural to ask for the behavior of the associated fluctuations. Although this problem hasn't yet been addressed for the fractional Brownian motion, some partial results have been studied in \cite{Bor} for the case where $X$ is a Brownian motion and in \cite{Jac} for the case where $X$ is a diffusion process. More precisely, it has been shown that in these instances, a suitable normalization of the fluctuations of \eqref{eq:Jeg} converges in distribution towards a mixture of Gaussian random variables.
%
%
%
%
%
%
%
%Despite its relevance the financial literature (see e.g. \cite{GJR}, as well as Remark \ref{r:ff} below), the case in which $X$ is a fBm -- so that in particular, $X$ is not a semimartingale -- is particularly ill-studied in the existing literature. \cite{D, Jeg, PoRo, TyuPhi, WP}
%
%
%
%
%
%The aim of the present paper is to substantially extend the findings of \cite{Jeg, PoRo, TyuPhi}, so to refine our understanding of the properties of \eqref{eq:generalhighfreq} as estimators of the local time process of $X$ as well as of its weak derivatives, whenever the latter are well-defined. To the best of our expertise, such weak derivatives are defined and studied here for the first time. 

\subsection{A representative statement} One of the crucial aims of the present paper is to explore in full generality the asymptotic behaviour of \eqref{eq:generalhighfreq}, in the case where $f$ has zero energy, and $X$ is a fBm with Hurst parameter in the range $0<  H <1/3$. Apart from the critical case $H=1/3$ (to which our techniques do not apply), this exactly corresponds to the values of $H$ that are not covered by the references \cite{HNX, Jeg2, Jeg3, NX} discussed in the previous section. As anticipated, one of the methodological breakthroughs of our work is the {\color{black}use} of the {\it derivatives} of the local times of $X$; the existence of such objects, as well as some of their basic properties, is discussed in Section \ref{ref:Appendix}. 

\begin{Remark}{\rm

{\color{black} 

\blue{As already recalled,} the \blue{existence of derivatives (DLT) for local times of fBm (with a suitably small Hurst parameter) was first proved in \cite[Section~28]{GerHor}, by using the general results of \cite{Ber2}}. \green{In this manuscript, our results concerning such a topic} (see Lemma \ref{Lem:existLocalderiv} in the Appendix) differ from those of the existing literature in several ways: (i) we introduce the DLT, not as an almost sure derivative of the local time with respect to the space variable, but rather as the limit a of a suitable sequence of approximating variables under the topology of $L^{2}(\Omega)$, which is the most natural framework for the purpose of our application (additionally, in order to make the comparison with the results from \cite{GerHor} and \cite{Ber2} more transparent, in Lemma \ref{Lem:existLocalderiv} we prove that our definition of DLT and the one from \cite{GerHor} are equivalent);  (ii) we \blue{provide} sharp conditions for the existence of the DLT, which allows us to use it under situations more general than those presented in \cite{GerHor}; (iii) we determine the time regularity of the DLT, which is a key ingredient for \blue{lifting} our finite-dimensional results for $G_{t}^{(n)}$  to a functional level.}}
\end{Remark}

Our main findings are stated in full generality in Theorem \ref{thm:main2} and Theorem \ref{thm:main2p} below, and require a non negligible amount of further notation, introduced in Section \ref{ss:notation}. In order to motivate the reader, we will now state some immediate consequences of such general statements, that directly capture the spirit of our work. In particular, the forthcoming Theorem \ref{t:rep} illuminates the meaning of the threshold $1/3$ observed in \cite{HNX, Jeg2, Jeg3, NX}, by connecting such a value to the existence of  derivatives for the local time of $X$.

\begin{thm}[Special case of Theorems \ref{thm:main2} and \ref{thm:main2p}]\label{t:introex}\label{t:rep} Let $X$ denote a fBm with Hurst parameter $H\in (0,1)$. Let $f : \R\to \R$ be a continuous function with compact support satisfying $\int_\R f(x) dx = 0$. Then, $f$ admits a unique antiderivative $F$ verifying $F\in L^1(\R)$. Writing $\mu[F] : = \int_\R F(x)dx$, the following conclusions {\rm (1)--(4)} hold, as $n\to\infty$.
\begin{itemize}
\item[\rm (1)] For every $0< H<1/3$, the first derivative of the local time of $X$ (defined as in Section \ref{ss:notation}), noted $L^{(1)} := \{L_t^{(1)}(\lettrequivabien) : (t,\lettrequivabien)\in [0,\infty) \times \R\}$, exists, verifies $\E[L_t^{(1)}(\lettrequivabien)^2] \in (0,\infty) $ for every $t>0$ and $\lettrequivabien\in\R$, and moreover
\begin{eqnarray}\label{e:t1}
\E\left[ \left (n^{2H-1}\sum_{i=1}^{\lfloor nt\rfloor}f(n^{H}(X_{\frac{i-1}{n}}-\lettrequivabien)) - L_t^{(1)}(\lettrequivabien) \mu[F]  \right)^2\, \right] = O(n^{ - 2H\kappa}),
\end{eqnarray}
for every $\kappa <\frac12(\frac1H-3) \wedge \frac12$, where the constant involved in the ` $O(\cdot)$' notation possibly depends on $t$.
\item[\rm (2)] For every $0<H<1/4$, one has also that, for every $T>0$,
$$
\sup_{t\in [0,T]} \left| n^{2H-1}\sum_{i=1}^{\lfloor nt\rfloor}f(n^{H}(X_{\frac{i-1}{n}}-\lettrequivabien)) - L_t^{(1)}(\lettrequivabien)\mu[F]  \right| \stackrel{\mathbb{P}}{\longrightarrow} 0,
$$
where $\stackrel{\mathbb{P}}{\longrightarrow}$ stands for convergence in probability.
\item[\rm (3)] Fix $0<H < 1/5$ and assume that $\widetilde{F}\in L^1(\R)$, where $\widetilde{F}(x) := xF(x)$. Then , the second derivative of the local time of $X$, written $L^{(2)} := \{L_t^{(2)}(\lettrequivabien) : (t,\lettrequivabien)\in [0,\infty) \times \R\}$, exists, verifies $\E[L_t^{(2)}(\lettrequivabien)^2] \in (0,\infty) $ for every $t>0$ and $\lettrequivabien\in\R$, and 
\begin{eqnarray}\label{e:t2}
&&\E\Big[ n^H\Big (\Big( n^{2H-1}\sum_{i=1}^{\lfloor nt\rfloor}f(n^{H}(X_{\frac{i-1}{n}}-\lettrequivabien)) - L_t^{(1)}(\lettrequivabien) \mu[F] \Big) \\
 && \quad\quad\quad\quad\quad\quad\quad\quad\quad\quad\quad\quad\quad+   L_t^{(2)} \mu[\widetilde{F}] \Big)^2\, \Big]  = O(n^{ - 2H\kappa}),\notag
\end{eqnarray}
for every $\kappa <\frac12(\frac1H-5)\wedge \frac12$.
\item[\rm (4)] For every $0<H < 1/6$, the asymptotic relation at Point {\rm (3)} takes also place in the sense of uniform convergence in probability over compact intervals.
\end{itemize}

\end{thm}

 The contents of Lemma \ref{Lem:existLocalderiv} and Lemma \ref{lem:nonexistenceLell} below implies that, for every $\ell = 1,2,...$, the  derivative $L^{(\ell)}$ exists if and only if $H< 1/(2\ell +1)$. We will see later on (Corollary \ref{c:cimbalon}) that our findings can also be used in order to study the fluctuations of Jeganathan's limit result \eqref{eq:Jeg} in the case $\int_\R f(x) dx \neq 0$.
%
%\begin{thm}[Special case of Theorems \ref{thm:main2} and \ref{thm:main2p}]\label{t:introex}\label{t:rep} Let ...
%
%
%\end{thm}

%As a by-product of our analysis, we find the following interesting phenomena that contrasts with the results on the current literature of high-frequency observations: while the limit of \eqref{eq:generalhighfreq} has been previously proven to be a (possibly trivial) multiple of the local time of $X$ at $\lettrequivabien$, we show that under suitable conditions on $H$, and provided that $f$ is the derivative of order $\ell$ of an integrable function, there exists a suitable choice of $b_{n}$ such that the processes \eqref{eq:generalhighfreq} converge to a constant multiple of the random function
%\begin{align}\label{eq:derlocaltimes}
%t\mapsto \int_{0}^{t}\delta_{0}^{(\ell)}(B_{s}-\lettrequivabien)ds, \quad t\geq 0,
%\end{align}
%where $\delta_{0}^{(\ell)}$ denotes the $\ell$-th derivative of the Dirac delta function at zero. The topology under which this convergence takes place and the precise conditions that $f$ and $H$ must satisfy will be presented in Theorem \ref{thm:main2}. Differently from previous results, we establish conditions under which a functional uniform convergence in probability holds, not only for the process \eqref{eq:generalhighfreq}, but also for the associated fluctuations around its limit.
%

\begin{Remark}{\rm Combining \eqref{e:t1} with e.g. \eqref{e:j2008} one sees that, choosing $H>1/3$ and $a=H$ in \eqref{eq:generalhighfreq}, the correct normalisation in the zero energy case is given by $b_n = n^{\frac{H-1}{2}}$, whereas for $H<1/3$ the normalisation has to be $b_n  = n^{2H-1}$. The two exponents $\frac{H-1}{2}$ and $2H-1$ coincide for the critical value $H=1/3$. As anticipated, the study of \eqref{eq:generalhighfreq} for $H=1/3$ is outside the scope both of our techniques (since in this case, the derivative of the local time of $X$ is not well-defined, by virtue of Lemma \ref{lem:nonexistenceLell} in the Appendix), and of those of \cite{HNX, Jeg2, Jeg3, NX} (e.g., since the constant $b$ appearing in \eqref{e:j2006} and \eqref{e:j2008} equals infinity, see \cite{Jeg2, Jeg3} as well as \cite[Theorem 1.1]{NX}).
}
\end{Remark}

\subsection{Some heuristic considerations}\label{s:introheu}
\noindent In order to make more transparent the connection between \eqref{eq:generalhighfreq} and the derivatives of the local time of $X$, we present here some heuristic argument. First of all, as already observed, if the function $f$ appearing in \eqref{eq:Jeg} is such that $\int_{\R}f(x)dx=0$, then the right hand side of \eqref{eq:Jeg} is equal to zero, which implies that the normalization $n^{H-1}$ is not adequate for deducing a non trivial limit. Notice that all functions $f$ of the form $f=g^{\prime}$ with $g,g'\in L^1(\R)$ satisfy the property $\int_{\R}f(x)dx=0$ (see indeed Remark \ref{Rk1.5}(a) for a proof), which suggests that, in order to have a non-trivial limit for \eqref{eq:generalhighfreq}, one must distinguish the case where $f$ is the weak derivative of an integrable function or, more generally, the case where it is the weak derivative of order $\ell$ of such a function. With this in mind, for all function $g$ with weak derivatives of order $\ell\geq0$ and all $a,t>0$, we define 
\begin{align}\label{e:gdef}
G_{t,\lettrequivabien,a}^{(n,\ell)}[g]
  &:=\sum_{i=2}^{\lfloor nt\rfloor}g^{(\ell)}(n^{a}(X_{\frac{i-1}{n}}-\lettrequivabien)),
\end{align}
with the convention that the above sum is equal to zero when $nt<2$. We observe that the definition of $G_{t,\lettrequivabien,a}^{(n,\ell)}[g]$ in \eqref{e:gdef} is unambiguously given, even if the weak derivative $g^{(\ell)}$ is only defined up to sets of zero Lebesgue measure, since the argument $n^a ( X_{\frac{i-1}{n}}-\lettrequivabien ) $ is a random variable whose distribution has a density, for every $i>1$. Now, at a purely heuristic level, we can write
\begin{align*}
\frac{1}{n}G_{t,\lettrequivabien,a}^{(n,\ell)}[g]
  &=\frac{1}{n}\sum_{i=2}^{\lfloor nt\rfloor}g^{(\ell)}(n^{a}(X_{\frac{i-1}{n}}-\lettrequivabien))\approx\int_{0}^tg^{(\ell)}(n^{a}(X_{s}-\lettrequivabien))ds\\
	&=\int_{0}^t\int_{\R}g^{(\ell)}(n^{a}x)\delta_{0}(X_{s}-\lettrequivabien-x)dxds
	=n^{-a}\int_{0}^t\int_{\R}g^{(\ell)}(x)\delta_{0}\left(X_{s}-\lettrequivabien-\frac{x}{n^a}\right)dxds,
\end{align*}
which, by a formal integration by parts, yields
\begin{align}\label{eq:approxGinformal}
\frac{1}{n}G_{t,\lettrequivabien,a}^{(n,\ell)}[g]
  &\approx n^{-a(\ell+1)}\int_{0}^t\int_{\R}g(x)\delta_{0}^{(\ell)}(X_{s}-\lettrequivabien-\frac{x}{n^a})dxds \approx n^{-a(\ell+1)}\bigg(\int_{\R}g(x)dx\bigg)L_{t}^{(\ell)}(\lettrequivabien),
\end{align}
where the random variable $L_{t}^{(\ell)}(\lettrequivabien)$ (when it exists) is given by
\begin{align*}
L_{t}^{(\ell)}(\lettrequivabien)
  &:=\int_{0}^t\delta_{0}^{(\ell)}(X_{s}-\lettrequivabien)ds.
\end{align*}
From here we can conjecture that under suitable hypotheses, the sequence $n^{a(\ell+1)-1}G_{t,\lettrequivabien,a}^{(n,\ell)}[g]$ converges to a scalar  multiple of $L_{t}^{(\ell)}(\lettrequivabien)$. However, one should observe that the approximation \eqref{eq:approxGinformal} can only be true under special conditions, in order for it to be consistent with the results evoked in Section \ref{ss:introfrac}. One should also notice that the above heuristic is based on the use of the generalized function $\delta_{0}^{(\ell)}$, which makes computations very hard to be rigorously justified. To overcome this difficulty, we use the Fourier representation $\delta_{0}$, and the Fourier inversion formula to rewrite $\frac{1}{n}G_{t,\lettrequivabien,a}^{(n,\ell)}[g]$ and $\int_{0}^t\delta_{0}^{(\ell)}(X_{s}-\lettrequivabien)ds$ as a mixture of terms of the form $e^{\textbf{i}\xi X_{s}}$ and $e^{\textbf{i}\xi X_{\frac{i}{n}}}$. Such a representation will facilitate algebraic manipulations, in view of the Gaussianity of $X$.\\

\noindent In addition to the verification of the conjecture above, on the limit of $n^{a(\ell+1)-1}G_{t,\lettrequivabien,a}^{(n,\ell)}[g]$, it is interesting to address the following natural problems arising from the approximation \eqref{eq:approxGinformal}.
%error associated to the convergence $n^{a(\ell+1)-1}G_{t,\lettrequivabien,a}^{(n,\ell)}[g]$ towards its limit

 %\eqref{eq:Jeg}, and allow the scalings $n^{H}$ and $n^{H-1}$ to be replaced by a more general normalization of the form $n^{a}$ and $n^{a-1}$, with $a>0$. These observations motivate the study of the following problems: if $g:\R\rightarrow\R$ is a suitable test function, $a,t>0$ and $G_{t,\lettrequivabien,a}^{(n,\ell)}[g]$ is defined by
%\begin{align*}
%G_{t,\lettrequivabien,a}^{(n,\ell)}[g]
  %&:=\sum_{i=2}^{\lfloor nt\rfloor}g^{(\ell)}(n^{a}(X_{\frac{i-1}{n}}-\lettrequivabien)),
%\end{align*}
%with the convention that the above sum is equal to zero when $nt<2$,
\begin{enumerate}
\item[(i)] Provided that $\{n^{a(\ell+1)-1}G_{t,\lettrequivabien,a}^{(n,\ell)}[g]\ \processsymb\ t\geq0\}$ has a non-trivial limit, what is the nature of the fluctuations of $G_{t,\lettrequivabien,a}^{(n,\ell)}[g]$ around such a limit?

\item[(ii)] If we are interested in estimating $L_{t}^{(\ell)}(\lettrequivabien)$ with a suitable normalization of the statistic $\{G_{t,\lettrequivabien,a}^{(n,\ell)}[g]\ :\ t\geq0\}$, how do we choose $a$ in order to minimize the associated mean-square error? 
\end{enumerate}
The behavior of $n^{a(\ell+1)-1}G_{t,\lettrequivabien,a}^{(n,\ell)}[g]$ will be described in Theorem \ref{thm:main2}, while the answer to (i) and (ii) will be provided in Theorem \ref{thm:main2p} and Remark \ref{rem:main}, respectively. Before presenting the precise statement of our results, we need to introduce some further notation and definitions. 

\subsection{Further notation}\label{ss:notation} Let $\ell\geq0$ be a non-negative integer satisfying $H<\frac{1}{2\ell+1}$. By Lemma \ref{Lem:existLocalderiv} in the Appendix, for every fixed $\lettrequivabien\in\R$, the 
collection of processes 
\begin{align}\label{eq:mollifierderivlt}
\left\{\int_{0}^{t}\phi_{\varepsilon}^{(\ell)}(X_{s}-\lettrequivabien)ds\ \processsymb\ t\geq0\right\}
\end{align}
converges pointwise in $L^{2}(\Omega)$ as $\varepsilon $ goes to zero, to a limit process $\{L_{t}^{(\ell)}(\lettrequivabien)\processsymb\ t\geq0\}$ that has a modification with H\"older continuous trajectories (in the variable $t$) of order $\gamma$, for all $0<\gamma<1-H(\ell+1)$. In this paper, the resulting continuous modification is called the {\it derivative} of order $\ell$ of the local time of $X$ at $\lettrequivabien$. Further properties for the process $L_{t}^{(\ell)}(\lettrequivabien)$ will be presented in Section \ref{sec:Fourierderivlocal}

\begin{Remark}
{\textup {The variable $L_{t}^{(0)}(\lettrequivabien) = L_t(\lettrequivabien)$ is of course the local time of $X$ at $\lettrequivabien$ and it has been widely studied (see \cite{CoNuTu, Ber, GerHor, HuOk}). {\color{black} The forthcoming Lemma \ref{Lem:existLocalderiv} provides a range of values of $H$ for which the variable $L_{t}^{(\ell)}(\lettrequivabien)$ exists as the limit in $L^{2}(\Omega)$ of a suitable mollification, for $\ell\geq 1$ (see also \cite[Section~28]{GerHor}). In addition to this existence result, in Lemma \ref{lem:nonexistenceLell} we prove that the condition $H<\frac{1}{2\ell+1}$ is sharp, in the sense that if $H\geq \frac{1}{2\ell+1}$, then $\int_{0}^{t}\phi_{\varepsilon}^{(\ell)}(X_{s})ds$ doesn't converge in $L^2(\Omega)$. Finally, in Lemma \ref{lem:nonexistenceLell} we prove that  $L_{t}^{(\ell)}(\lettrequivabien)$ can be regarded as the space derivative of $L_{t}^{(\ell-1)}(\lettrequivabien)$ with respect to the $L^{2}(\Omega)$-topology; this observation \blue{yields that the random mapping $ \lambda \mapsto L_{t}^{(\ell)}(\lettrequivabien)$ introduced above coincides --- up to the choice of a suitable modification --- with the definition of the a.s. spatial derivative of the local time of $X$ used in \cite[Section~28]{GerHor}}.}}}
\end{Remark} 

  For $r\in\N$ and $p\geq 1$, we will denote by $W^{r,p}$ the set of functions $g:\R\rightarrow\R$ with weak derivatives of order $r$, such that $g^{(i)}\in L^{p}(\R)$ for all $i=0,\dots, r$. We will endow $W^{r,p}$ with the norm $\|\cdot\|_{W^{r,p}}$ given by
\begin{align*}
\|g\|_{W^{r,p}}
  &:=\bigg(\sum_{i=0}^{r}\int_{\R}|g^{(i)}(x)|^{p}dx\bigg)^{\frac{1}{p}}.
\end{align*} 
Define the function $w:\R\rightarrow\R_{+}$ by
\begin{align*}
 w(x):=1+|x|.
\end{align*}
For a given $\kappa>0$ and a non-negative integer $\ell\geq 0$, we will denote by  $\mathcal{K}_{\kappa}^{\ell}$ the collection
\begin{align*}
\mathcal{K}_{\kappa}^{\ell}
  &:=\{g\in W^{\ell,1}\ |\  w^{\kappa}g\in L^{1}(\R)\ \text{ and } g^{(\ell)}\in L^{2}(\R)\}.
\end{align*}
%In Lemma \ref{Lem:existLocalderiv} at the appendix, it is proved that if $H<\frac{1}{2\ell+1}$ and $0<\gamma<1-H(\ell+1)$, the limit 
%\begin{align*}
%L_{t}^{(\ell)}(\lettrequivabien)
  %:=\lim_{\varepsilon\rightarrow0}\int_{0}^{t}\phi_{\varepsilon}^{(\ell)}(X_{s}-\lettrequivabien)ds
%\end{align*}
%is well defined, and the process $\{L_{t}^{(\ell)}(\lettrequivabien)\ ;\ t\geq0\}$ has a modification with H\"older trajectories of order $\gamma$. We will refer to this limit as the derivative of order $\ell$ of the local time of $X$ at $\lettrequivabien$.

\noindent Finally, for a fixed function $g\in W^{\ell,1}$ and a positive constant $a>0$, we set
\begin{align*}
\mu[g]
  &:=\int_{\R}g(x)dx.
\end{align*}

\begin{Remark}{\rm For every $\kappa>0$, one has clearly that $\mathcal{K}_{\kappa}^{0} \subset L^1(\R)\cap L^2(\R)$, where the inclusion is strict. This implies that our forthcoming statements (in particular, relation \eqref{eq:thmmain22} in the case $\ell = 0$) contain a version of Jeganathan's limit theorem \eqref{eq:Jeg} under slightly more stringent assumptions. Nonetheless, we stress that \eqref{eq:Jeg} is a purely qualitative statement, whereas the forthcoming estimate \eqref{eq:thmmain22} also displays an explicit upper bound on the mean-square difference between the two terms.
}
\end{Remark}

\subsection{Main results}
We now present one of the main results of the paper, which is a functional law of large numbers for $\{G_{t,\lettrequivabien,a}^{(n,\ell)}[g]\ \processsymb\ t\geq0\}$.
%===========================================================
\begin{Theorem}\label{thm:main2}
If $0<a\leq H<\frac{1}{2\ell+1}$ and $\kappa\in(0,\frac{1}{2})$ is such that $H(2\ell+2\kappa+1)<1$, then for all $g\in \mathcal{K}_{\kappa}^{\ell}$  there exists a constant $C_t>0$ independent of $n$ and $g$, such that 
\begin{align}\label{eq:thmmain22}
\| n^{a(\ell+1)-1}G_{t,\lettrequivabien,a}^{(n,\ell)}[g]-L_{t}^{(\ell)}(\lettrequivabien)\mu[g] \|_{L^{2}(\Omega)}
  &\leq C_tn^{-(2a\kappa)\wedge \kappa}(\|w^{\kappa}g\|_{L^{1}(\R)}+\|g^{(\ell)}\|_{L^{2}(\R)}^2).
\end{align}
In addition, the processes $\{G_{t,\lettrequivabien,a}^{(n,\ell)}[g]\ \processsymb\ t\geq0\} $ satisfy the following functional convergences:
\begin{enumerate}
\item[(i)] If $\ell=0$ and $T>0$, 
\begin{align}\label{eq:ucpmain0}
\sup_{t\in[0,T]}|n^{a-1}G_{t,\lettrequivabien,a}^{(n,0)}[g]-L_{t}^{(0)}(\lettrequivabien)\mu[g]|
  &\stackrel{\mathbb{P}}{\rightarrow}0
	\quad\quad\quad  \text{ as }\  n\rightarrow\infty.
\end{align}

\item[(ii)] If  $\ell\geq 1$ and $T_{1},T_{2}>0$, then 
\begin{align}\label{eq:ucpmain}
\sup_{t\in[T_1,T_2]}|n^{a(\ell+1)-1}G_{t,\lettrequivabien,a}^{(n,\ell)}[g]-L_{t}^{(\ell)}(\lettrequivabien)\mu[g]|
  &\stackrel{\mathbb{P}}{\rightarrow}0
	\quad\quad\quad  \text{ as }\  n\rightarrow\infty.
\end{align}

\item[(iii)] If $\ell\geq 1$, $H<\frac{1}{2\ell+2}$  and $T>0$, then 
\begin{align}\label{eq:ucpmain2}
\sup_{t\in[0,T]}|n^{a(\ell+1)-1}G_{t,\lettrequivabien,a}^{(n,\ell)}[g]-L_{t}^{(\ell)}(\lettrequivabien)\mu[g]|
  &\stackrel{\mathbb{P}}{\rightarrow}0
	\quad\quad\quad  \text{ as }\  n\rightarrow\infty.
\end{align}
\end{enumerate}
\end{Theorem}

\begin{Remark}[Consistency between different values of $\ell$]\label{Rk1.5} {\rm
\begin{itemize}
\item[(a)] If $g\in\mathcal{K}_{\kappa}^{\ell}\subset W^{\ell,1}$, then $g$ and its weak derivatives $g^{(1)},....,g^{(\ell)}$ are integrable. An application of dominated convergence and of Rodrigues' formula for Hermite polynomials consequently yields that, for $ j = 1,..., \ell$,
\begin{align*}
\left|\int_{\R}g^{ (j) }(x)dx\right| 
  & = \lim_{\varepsilon\rightarrow0}\left|\int_{\R}g^{ (j) }(x)e^{-(\varepsilon x)^2/2}dx\right|
	\\ & =\lim_{\varepsilon\rightarrow0}\left|(2\varepsilon)^j \int_{\R}g(x)H_j(\varepsilon x) e^{-(\varepsilon x)^2/2}dx\right|
	\leq 2^j c_j\|g\|_{L^1(\R)}\, \lim_{\varepsilon\rightarrow0}\varepsilon^j =0,
\end{align*}
where $\{ H_j\}$ indicates the sequence of Hermite polynomials, and $$c_j := \sup_{z\in \R} \left| H_j(z) e^{-z^2/2}\right| <\infty .$$

\item[(b)]From the definition of $G_{t,\lettrequivabien,a}^{(n,\ell)}[g]$ it easily follows that, for $g\in\mathcal{K}_{\kappa}^{\ell}\subset W^{\ell,1}$ and for $j=1,..., \ell$,
\begin{align*}
G_{t,\lettrequivabien,a}^{(n,\ell)}[g]
  &=G_{t,\lettrequivabien,a}^{(n,\ell-j )}[g^{(j)}].
\end{align*}
In view of Point {\rm (a)}, such a relation is consistent with the content of Theorem \ref{thm:main2}. Indeed, combining Point (a) with Theorem \ref{thm:main2} we infer that, for $a\leq H < 1/(2\ell+1)$ and $j=1,..., \ell$,
\begin{equation}\label{e:rr}
n^{a(\ell-j +1)-1}G_{t,\lettrequivabien,a}^{(n,\ell - j)}[g^{(j )}] = n^{a(\ell-j +1)-1}G_{t,\lettrequivabien,a}^{(n,\ell)}[g]
\rightarrow    L_{t}^{(\ell)}(\lettrequivabien)\!\! \int_\R g^{(j)}(x) dx = 0,
\end{equation}
as $n\to\infty$, where the convergence takes place in $L^2(\Omega)$; one sees immediately that the relation $n^{a(\ell-j +1)-1}G_{t,\lettrequivabien,a}^{(n,\ell - j)}[g^{(j )}] \to 0$, for every $j=1,...,\ell$, is also a direct consequence of the fact that 
$$
\left\| n^{a(\ell+1)-1}G_{t,\lettrequivabien,a}^{(n,\ell)}[g]-L_{t}^{(\ell)}(\lettrequivabien)\mu[g] \right\|_{L^{2}(\Omega)}\longrightarrow 0,
$$
which one can deduce from Theorem \ref{eq:thmmain22}.
%On the other hand,
%
%
%\begin{align*}
%n^{a(\ell+1)-1}G_{t,\lettrequivabien,a}^{(n,\ell)}[g^{\prime}]
%  \rightarrow L_{t}^{(\ell)}(\lettrequivabien)C_{g^{\prime}}
%	\  \text{ and } \ 
%n^{a(\ell+2)-1}G_{t,\lettrequivabien,a}^{(n,\ell+1)}[g]
%  \rightarrow L_{t}^{\blue{ (\ell +1) }}(\lettrequivabien)C_{g},
%\end{align*}
%for some constants $C_{g^{\prime}},C_{g}\in\R$, then \blue{necessarily} $C_{g^{\prime}}=0$. Thus, in order to have results analogous to Theorem \ref{thm:main2}, where $C_{g^{\prime}}=\mu[g^{\prime}]=\int_{\R}g^{\prime}(x)dx$, the test function $g$ must satisfy 
%\begin{align}\label{eq:consist}
%\int_{\R}g^{\prime}(x)dx=0,
%\end{align}
%which is a property that does not necessarily hold even in the case where $g^{\prime}\in L^{1}(\R)$. 
\end{itemize}
}
\end{Remark}

\begin{Remark}[Regarding uniform convergence in Theorem \ref{thm:main2}]\label{rem:tightness}
\noindent 
\textup{Using Dini's theorem, we can easily check that the pointwise convergence of $\{n^{a -1}G_{t,\lettrequivabien,a}^{(n,0)}[g]\ \processsymb\ t\geq0\}$ towards $\{L_{t}^{(0)}(\lettrequivabien)\mu[g]\ \processsymb\ t\geq0\}$ implies its uniform convergence over compact subsets of $[0,\infty)$. To verify this claim, it suffices to write
\begin{align}\label{eq:tightDini}
n^{a-1}G_{t,\lettrequivabien,a}^{(n,0)}[g]
  =n^{a -1}G_{t,\lettrequivabien,a}^{(n,0)}[g_{+}]-n^{a -1}G_{t,\lettrequivabien,a}^{(n,0)}[g_{-}],
\end{align}
where $g_{+}(x):=g(x)\Indi{(0,\infty)}(g(x))$ and $g_{-}(x):=g(x)\Indi{(-\infty,0]}(g(x))$. Both $n^{2a -1}G_{t,\lettrequivabien,a}^{(n,0)}[(g^{\prime})_{+}]$ and $n^{2a -1}G_{t,\lettrequivabien,a}^{(n,0)}[(g^{\prime})_{-}]$ are increasing in $t$ and converge pointwise in probability to continuous processes, which implies that they converge uniformly, by virtue of Dini's theorem. Unfortunately, this argument does not work in the general case $\ell\geq1$. To check this, consider the test function $g(x):=\exp\{-x^2\}$. In this instance, the analog of the decomposition \eqref{eq:tightDini} is
\begin{align}\label{eq:tightDini2}
n^{2a-1}G_{t,\lettrequivabien,a}^{(n,1)}[g]
  =n^{2a -1}G_{t,\lettrequivabien,a}^{(n,0)}[(g^{\prime})_{+}]-n^{2a -1}G_{t,\lettrequivabien,a}^{(n,0)}[(g^{\prime})_{-}].
\end{align}
Notice that $\mu[g^{\prime}_{-}],\mu[g^{\prime}_{-}]\neq0$, and thus, by inequality \eqref{eq:thmmain22}, the terms $\|n^{2a -1}G_{t,\lettrequivabien,a}^{(n,0)}[(g^{\prime})_{+}]\|_{L^2(\Omega)}$ and $\|n^{2a -1}G_{t,\lettrequivabien,a}^{(n,0)}[(g^{\prime})_{-}]\|_{L^2(\Omega)}$ diverge to infinity while $n^{2a-1}G_{t,\lettrequivabien,a}^{(n,1)}[g]$ converges in $L^{2}(\Omega)$. This prevents us from using the decomposition \eqref{eq:tightDini2} for analyzing the uniform convergence of $n^{2a-1}G_{t,\lettrequivabien,a}^{(n,1)}[g]$. For this reason, instead using Dini's theorem, we tackle the tightness property for the case $\ell\geq 1$ by means of the Billingsley's criterion (see \cite[Theorem~12.3]{Billin}). Due to the high level difficulty of the application of this methodology, we were only able to prove uniform convergence either over compact subsets of $(0,\infty)$ in the general case $H<\frac{1}{2\ell+1}$ or over compact subsets of $[0,\infty)$ in the more restrictive case $H<\frac{1}{2\ell+2}$. We conjecture that the uniform convergence over compact subsets of $[0,\infty)$ in the general case $H<\frac{1}{2\ell+1}$ can eventually be shown by finding a suitable estimation on the moments of arbitrarily large order for the increments of $\{G_{t,\lettrequivabien,a}^{(n,\ell)}[g]\ \processsymb\ t\geq0\}$.}
\end{Remark}

\begin{Remark}
\textup{Theorem \ref{thm:main2} provides a full description of all the possible pointwise limits in probability of the high-frequency observations $\{G_{t,\lettrequivabien,a}^{(n,\ell)}[g]\ \processsymb\ t\geq0\}$, provided that the index $\ell$ associated to $g$ satisfies the condition $H<\frac{1}{2\ell+1}$. However, it is still unclear to us what is its behavior in the case $\ell\geq 1$ and $H>\frac{1}{2\ell+1}$, so we have left this interesting problem open for future research.  }
\end{Remark}

\noindent The following result describes the limit behavior for the error associated with the convergences stated in Theorem \ref{thm:main2}, under suitable conditions on $H$ and $\ell$.

%=========================================================
\begin{Theorem}\label{thm:main2p}
If $0<a\leq H<\frac{1}{2\ell+3}$ and $\kappa\in(0,\frac{1}{2})$ is such that $H(2\ell+2\kappa+3)<1$, then for all $g\in \mathcal{K}^{\ell}_{1+\kappa}$, there exists a constant $C_t>0$ independent of $n$ and $g$, such that  
\begin{multline}\label{eq:thmmain223p}
\|n^a(n^{a(\ell+1)-1}G_{t,\lettrequivabien,a}^{(n,\ell)}[g]-L_{t}^{(\ell)}(\lettrequivabien)\mu[g])+ L_{t}^{(\ell +1)}(\lettrequivabien)\mu[\tilde{g}]\|_{L^{2}(\Omega)}\\
  \leq C_tn^{-2a\kappa}(\|w^{1+\kappa}g\|_{L^{1}(\R)}+\|g\|_{W^{2,1}} +\|g^{(\ell)}\|_{L^{2}(\R)}),
\end{multline}
where $\tilde{g}(x):=xg(x)$. In addition, the  processes $\{G_{t,\lettrequivabien,a}^{(n,\ell)}[g]\ \processsymb\ t\geq0\} $ satisfy the following functional convergences:

\begin{enumerate}
\item[(i)] If  $T_{1},T_{2}>0$, then 
\begin{align}\label{eq:ucpmainp}
\sup_{t\in[T_1,T_2]}|n^a(n^{a(\ell+1)-1}G_{t,\lettrequivabien,a}^{(n,\ell)}[g]-L_{t}^{(\ell)}(\lettrequivabien)\mu[g])+ L_{t}^{(\ell+1)}(\lettrequivabien)\mu[\tilde{g}]|
  &\stackrel{\mathbb{P}}{\rightarrow}0
	\quad\quad\quad  \text{ as }\  n\rightarrow\infty.
\end{align}

\item[(ii)] If $H<\frac{1}{2(\ell+2)}$  and $T>0$, then 
\begin{align}\label{eq:ucpmain2p}
\sup_{t\in[0,T]}|n^a(n^{a(\ell+1)-1}G_{t,\lettrequivabien,a}^{(n,\ell)}[g]-L_{t}^{(\ell)}(\lettrequivabien)\mu[g])+ L_{t}^{(\ell+1)}(\lettrequivabien)\mu[\tilde{g}]|
  &\stackrel{\mathbb{P}}{\rightarrow}0
	\quad\quad\quad  \text{ as }\  n\rightarrow\infty.
\end{align}
\end{enumerate}
\end{Theorem}

\begin{Remark}
\textup{In the case where $G(x):=\int_{-\infty}^{x}g(y)dy$ satisfies $G(x)=o(x)$ as $|x|$ approaches to infinity, we have that 
$$\int_{\R}G(y)dy=-\int_{\R}yg(y)dy.$$
Thus, in this situation we can obtain Theorem \ref{thm:main2p} from Theorem \ref{thm:main2}, by replacing $\ell$ by $\ell+1$ and $g$ by $G$. Notice however that there are examples of test functions $g$ satisfying $\mu[g],\mu[\tilde{g}]\neq0$ (take for instance $g(x):=(x+1)\phi_1(x)$), in such a way the above argument does not provide an equivalence between Theorems \ref{thm:main2p} and  \ref{thm:main2}.}
\end{Remark}

As anticipated, the next statement shows that Theorem \ref{thm:main2p} in the case $\ell = 0$ yields a second order counterpart to Jenagathan's result \eqref{eq:Jeg}, in the range $0<H<\frac13$. The case $H\geq \frac13$ is outside the scope of the techniques developed in the present paper: we prefer to think about this issue as a separate problem, and leave it open for future research.

\begin{cor} \label{c:cimbalon}Fix $0<H<\frac13$, consider $\kappa>0$ such that $\kappa < \frac12(H^{-1} - (2\ell+3)^{-1})$, and fix $f\in K_{\kappa}^0$ such that $\mu[f], \mu[\tilde{f}]\neq 0$. Then, the convergence \eqref{eq:Jeg} takes place for every $0<a\leq H$, and moreover, as $n\to\infty$,
$$
n^a\left( n^{a-1}\sum_{i=1}^{\lfloor nt\rfloor}f(n^{a}(X_{\frac{i-1}{n}}-\lettrequivabien))
   - L_{t}(\lettrequivabien)\int_{\R}f(x)dx\right) \stackrel{L^2(\Omega)}{\longrightarrow} \mu[\tilde{f}]L_t^{(1)}(\lettrequivabien).
$$
\end{cor}

\begin{Remark}\label{rem:main}
\textup{In general, if we are interested in estimating $L_{t}^{(\ell)}$ with the statistic $$n^{a(\ell+1)-1}G_{t,\lettrequivabien,a}^{(n,\ell)},$$ for $a\leq H$, it is relevant to choose $a$ in such a way that the $L^{2}(\Omega)$-norm of the error $n^{a(\ell+1)-1}G_{t,\lettrequivabien,a}^{(n,\ell)}[g]-L_{t}^{(\ell)}(\lettrequivabien)\mu[g]$ is as small as possible. This problem is closely related to Theorem \ref{thm:main2p}, due to the fact that under the condition $\mu[\tilde{g}]\neq 0$, the convergence 
$$n^a(n^{a(\ell+1)-1}G_{t,\lettrequivabien,a}^{(n,\ell)}[g]-L_{t}^{(\ell)}(\lettrequivabien)\mu[g])\stackrel{L^2(\Omega)}{\longrightarrow}-L_{t}^{(\ell)}(\lettrequivabien)\mu[\tilde{g}]$$
implies that the $L^{2}(\Omega)$-norm of $n^{a(\ell+1)-1}G_{t,\lettrequivabien,a}^{(n,\ell)}[g]-L_{t}^{(\ell)}(\lettrequivabien)\mu[g]$ is of the order $n^{-a}$. Consequently, the value of $a$ that optimizes the rate at which $n^{a(\ell+1)-1}G_{t,\lettrequivabien,a}^{(n,\ell)}[g]$ converges to $L_{t}^{(\ell)}(\lettrequivabien)\mu[g]$ in $L^2(\Omega)$ is $a=H$.}
\end{Remark}

\subsection{Plan} The rest of the paper is organized as follows: In Section \ref{sec:prelimn} we present some preliminary results on the fractional Brownian motion, Malliavin calculus and local non-determinism. In Sections \ref{sec:firstfluct} and \ref{sec:firstfluc2t2} we prove Theorems \ref{thm:main2} and \ref{thm:main2p}. Finally, in Section \ref{ref:Appendix} we present some results related to the properties of $L_{t}^{(\ell)}(\lettrequivabien)$, and prove some technical identities for the proofs of Theorems \ref{thm:main2} and \ref{thm:main2p}.

\medskip

\noindent{\bf Acknowledgments}. We thank Mark Podolskij for a number of illuminating discussions. AJ is supported by the FNR grant R-AGR-3410-12-Z (MISSILe) at Luxembourg and Singapore Universities; GP is supported by the FNR grant R-AGR-3376-10 (FoRGES) at Luxembourg University.

%=======================================================================================================
\section{Preliminaries}\label{sec:prelimn}
\subsection{Malliavin calculus for classical Gaussian processes}\label{subsec:chaos}
In this section, we provide some notation and introduce the basic operators of the theory of Malliavin calculus. The reader is referred to \cite{NoP11, Nualart} for full details. Throughout this section, 
$X=\{X_{t}\ \processsymb\ t\geq0\}$ denotes a fractional Brownian motion defined on a probability space $(\Omega,\Fc,\Pb)$. Namely, $X$ is a centered Gaussian process with covariance function $\E\left[X_{s} X_{t} \right]=R(s,t)$, where
\begin{align*}
R(s,t)
  &=\frac{1}{2}(s^{2H}+t^{2H}-|t-s|^{2H}).
\end{align*}
We denote by $\Hg$ the Hilbert space obtained by taking the completion of the space of step functions over $[0,T]$, endowed with the inner product 
\begin{align*}
\Ip{\Indi{[0,s]},\Indi{[0,t]}}_{\Hg}
  &=  \E\left[X_{s} X_{t} \right],\qquad \text{ for }\quad 0\leq s,t. 
\end{align*}
The mapping $\Indi{[0,t]} \mapsto X_{t}$ can be extended to a linear isometry between $\Hg$ and the linear Gaussian subspace of ${L}^{2}\left(\Omega\right)$ generated by the process $X$. We will denote this isometry by $X(h)$, for $h\in\Hg$. For any integer $q\geq1$, we denote by $\Hg^{\otimes q}$ and $\Hg^{\odot q}$ the $q$-th tensor product of $\Hg$, and the $q$-th symmetric tensor product of $\Hg$ respectively. The $q$-th Wiener chaos, denoted by $\Hc_{q}$, is the closed subspace of ${L}^{2}(\Omega)$ generated by the variables 
\[
\left \{H_{q}(X(v))\ \Big|\ v\in\Hg,\Norm{v}_{\Hg}=1 \right\},
\]
 where $H_{q}$ is the $q$-th Hermite polynomial, defined by 
\begin{align*}
H_{q}(x)
  &:=(-1)^{q}e^{\frac{x^{2}}{2}}\frac{\text{d}^{q}}{\text{d}x^{q}}e^{-\frac{x^{2}}{2}}.
\end{align*}
For $q\in\N$ satisfying $q\geq1$, and $h\in \Hg$ such that $\Norm{h}_{\Hg}=1$, we define the mapping $I_{q}(h^{\otimes q}):=H_{q}(X(h)).$ 
%The range of $I_{q}$ is contained in $\Hc_{q}$. Furthermore, this mapping 
It can be extended to a linear isometry between $\Hg^{\odot q}$ (equipped with the norm $\sqrt{q!}\Norm{\cdot}_{\Hg^{\otimes q}}$) and $\Hc_{q}$ (equipped with the ${L}^{2}(\Omega)$-norm).\\

\noindent From now on, we assume that $\Fc$ coincides with the $\sigma$-field generated by $X$. By the celebrated chaos decomposition theorem, every element $F$ belonging to the space ${L}^{2}(\Omega,\Fc) = L^2(\Omega)$ of $\Fc$-measurable, square-integrable random variables can be written as
\begin{align*}
F=\E\left[F\right]+\sum_{q=1}^{\infty}I_{q}(h_{q}),
\end{align*}
for some unique sequence $\{h_{q}\}$ such that $h_{q}\in \Hg^{\odot q}$. Notice that for every 2-dimensional centered Gaussian vector $\vec{Y}=(Y_{1},Y_{2})$ satisfying $\E[Y_{1}^2]=\E[Y_{2}^2]=1$, we can find elements $h_1,h_{2}\in\Hg$ such that $\|h_{1}\|_{\Hg}=  \|h_{2}\|_{\Hg}=1$ and $\vec{Y}\stackrel{Law}{=}(X(h_{1}),X(h_{2}))$. Consequently, by the isometry property of $I_{q}$, we have that for all $q,\tilde{q}\in\N$, 

\begin{align}\label{eq:Hqorthog}
\E[H_{q}(Y_{1})H_{\tilde{q}}(Y_{2})]
  &=\E[I_{q}(h_{1}^{\otimes q})I_{\tilde{q}}(h_{2}^{\otimes \tilde{q}})]
	=q!\delta_{q,\tilde{q}}\langle h_{1} , h_{2} \rangle^q=\delta_{q,\tilde{q}}\E[Y_{1}Y_{2}]^q
\end{align}

In what follows, for every integer $q\geq1$, we will denote by 
$$J_{q}:{L}^{2}(\Omega)\rightarrow \mathcal{H}_q\subset {L}^{2}(\Omega)$$ 
the projection over the $q$-th Wiener chaos $\Hc_{q}$. In addition, we  denote by $J_{0}(F)$ the expectation of $F$. Let $\Sf$ denote the set  of all cylindrical random variables of the form
\begin{align*}
F= g(X(h_{1}),\dots, X(h_{n})),
\end{align*} 
where $g:\R^{n}\rightarrow\R$ is an infinitely differentiable function with compact support  and $h_{1},\dots, h_{n}$ are step functions defined over $[0,\infty)$. In the sequel, we  refer to the elements of $\Sf$ as ``smooth random variables''. For every $r\geq2$, the Malliavin derivative of order $r$ of $F$ with respect to $X$ is the element of ${L}^{2}(\Omega; \Hg^{\odot r})$ defined by 
\begin{align*}
D^rF
  &=\sum_{i_1,\dots, i_{r}=1}^{n}\frac{\partial^{r}g}{\partial x_{i_1}\cdots \partial x_{i_r}}(X(h_{1}),\dots, X(h_{n}))h_{i_1}\otimes\cdots\otimes h_{i_r}.
\end{align*}
For $p\geq1$ and $r\geq1$, the space $\D^{r,p}$ denotes the closure of $\Sf$ with respect to the norm $\Norm{\cdot}_{\D^{r,p}}$, defined by 
\begin{align}\label{eq:seminorm}
\Norm{F}_{\D^{r,p}}
  &:=\left(\E\left[\Abs{F}^{p}\right]+\sum_{i=1}^{r}\E\left[\Norm{D^{i}F}_{(\Hg^{d})^{\otimes i}}^{p}\right]\right)^{\frac{1}{p}}.
\end{align}
The operator $D^{r}$ can be consistently extended to the space $\D^{r,p}$. 

Let ${L}^{2}(\Omega;\Hg)$ denote the space of square integrable random variables with values in $\Hg$. A random element $u\in {L}^{2}(\Omega;\Hg)$ belongs to the domain of the divergence operator $\delta$, denoted by $\mathrm{Dom} \, \delta $, if and only if it satisfies
\begin{align*}
\Abs{\E\left[\Ip{DF,u}_{\Hg }\right]}
  &\leq C_{u}\E\left[F^{2}\right]^{\frac{1}{2}},\ \text{ for every } F\in\D^{1,2},
\end{align*}
where $C_{u}$ is  a constant only depending on $u$. If $u\in \mathrm{Dom} \,\delta$, then the random variable $\delta(u)$ is defined by the duality relationship
\begin{align*}
\E\left[F\delta(u)\right]=\E\left[\Ip{DF,u}_{\Hg}\right],
\end{align*}
which holds for every $F\in\D^{1,2}$. The divergence satisfies the property that for all $F\in\D^{1,2}$ and $u$ belonging to the domain of $\delta$ such that $Fu\in L^{2}(\Omega,\Hg)$, the $\Hg$-valued random variable $Fu$ belongs to the domain of $\delta$ and 
\begin{align}\label{eq:prodfor}
\delta(Fu)
  &=F\delta(u)-\langle DF,u\rangle_{\Hg}.
\end{align}
The reader is referred to \cite[Proposition~1.3.3]{Nualart} for a proof of this identity. The operator $L$ is the unbounded operator from $\D^{2,2}(\Omega)$ to $L^{2}(\Omega)$ given by the formula
$$LF:=\sum_{q=1}^{\infty}-q J_{q}F.$$
It is the infinitesimal generator of the Ornstein-Uhlenbeck semigroup $\{P_{\theta}\}_{\theta\geq0}$, which is defined as follows
\begin{align*}
\begin{array}{cccc}P_{\theta}: & L^{2}(\Omega)&\rightarrow & L^{2}(\Omega)\\
                               & F                &\mapsto     & \sum_{q=0}^{\infty}e^{-q\theta}J_{q}F.\end{array}
  %&:=\sum_{q=0}^{\infty}e^{-q\theta}J_{q}.
\end{align*}
Moreover, for any $F\in\D^{2,2}(\Omega)$ one has
\begin{align}\label{eq:deltaDFI}
\delta DF
  &=-LF.
\end{align}
We also define the operator $L^{-1}:L^{2}(\Omega)\rightarrow L^{2}(\Omega)$ by
$$L^{-1}F=\sum_{q=1}^{\infty}-\frac{1}{q}J_{q}F.$$
Notice that $L^{-1}$ is a bounded operator and satisfies $LL^{-1}F=F-\E\left[F\right]$ for every $F\in L^{2}(\Omega)$, so that $L^{-1}$ acts as a pseudo-inverse of $L$.  Assume that $\widetilde{X}$ is an independent copy of $X$ such that $X,\widetilde{X}$ are defined in the product space $(\Omega\times \widetilde{\Omega}, \Fc \otimes\widetilde{\Fc},\Pb\otimes\widetilde{\Pb})$. Given a random variable $F\in L^{2}(\Omega)$, we can write $F=\Psi_{F}(X)$, where $\Psi_{F}$ is a measurable mapping from $\R^{\Hg}$ to $\R$, determined $\Pb$-a.s. Then, for every $\theta\geq0$ we have the Mehler formula 
\begin{align}\label{eq:Mehler}
P_{\theta}F
  &=\widetilde{\E}\left[\Psi_{F}(e^{-\theta}X+\sqrt{1-e^{-2\theta}}\widetilde{X})\right],
\end{align}
where $\widetilde{\E}$ denotes the expectation with respect to $\widetilde{\Pb}$. The operator $-L^{-1}$ can be expressed in terms of $P_{\theta}$, as follows
\begin{align}\label{eq:Mehler2}
-L^{-1}F
  &=\int_{0}^{\infty}P_{\theta}Fd\theta,~ \text{ for any } F \text{ such that }\E\left[F\right]=0.
\end{align}
From \eqref{eq:deltaDFI}-\eqref{eq:Mehler2} it follows that if $F=g(X(h))$ for some $h\in\Hg$ and $g:\R\rightarrow\R$ such that $g^{\prime}(X(h))\in L^{2}(\Omega)$, then
\begin{align}\label{eq:writeFasdeltaq}
F-\E[F]
  &=-\delta DL^{-1}F=\delta\bigg(\int_{0}^{\infty}e^{-\theta}P_{\theta}DFd\theta\bigg)\nonumber\\
	&=\delta\big(h\int_{0}^{\infty}e^{-\theta}\tilde{\mathbb{E}}[g^{\prime}(e^{-\theta}X(h)+\sqrt{1-e^{-2\theta}}\tilde{X}(h)]d\theta\big)\nonumber\\
  &=\delta\big(h\int_{0}^{1} \tilde{\mathbb{E}}[g^{\prime}(zX(h)+\sqrt{1-z^2}\tilde{X}(h)]dz\big).
\end{align}
Consequently, we deduce from \eqref{eq:prodfor} that for all $h_1,h_2\in\Hg$ and all differentiable functions $g_1,g_2:\R\rightarrow\R$ such that $g_1^{\prime}(X(h_1)),g_2^{\prime}(X(h_2))\in L^{2}(\Omega)$,  the random variables $F_1=g_1(X(h_1))$ and $F_2=g_2(X(h_2))$ satisfy
\begin{align}\label{eq:Fasdeltaq}
F_2(F_1-\E[F_1])
  &=\delta\big(h_1F_2\int_{0}^{1}\tilde{\E}[g_1^{\prime}(zX(h_1)+\sqrt{1-z^2}\tilde{X}(h_1))]dz\big)\nonumber\\
	&+\langle h_1,h_2\rangle_{\Hg} g_2^{\prime}(X(h_2))\int_{0}^{1}\tilde{\E}[g_1^{\prime}(zX(h_1)+\sqrt{1-z^2}\tilde{X}(h_1))]dz.
\end{align}

%==================================================================================================================

\subsection{Properties of the covariance of Gaussian vectors}\label{subsec:variance_estimation}
We next present some estimations for the increments of $X$ and identities for the determinant of covariance matrix of Gaussian vectors. We start with estimates that will be repeatedly used throughout the paper.
\begin{Lemma}
For all $T>0$, there exists a constant $C>0$ only depending on  $H$, such that for all $u,v,w\in[0,T]$ and $h,k>0$ satisfying $2h\leq v-u$, 

\begin{align}
|\langle\Indi{[u,u+h]},\Indi{[0,w]}\rangle_{\Hg}|
  &\leq \Indi{\{H<\frac{1}{2}\}}h^{2H}+T^{2H-1}\Indi{\{H\geq \frac{1}{2}\}}h ,\label{eq:estimationip2p}\\
|\langle \Indi{[u,u+h]},\Indi{[v,v+k]}\rangle_{\Hg}|
  &\leq 2^{2-2H}H|2H-1|hk|v-u|^{2H-2}.\label{eq:estimationip}
\end{align}
\end{Lemma}
\begin{proof}
Since $2h\leq v-u$, we have that $\Indi{[u,u+h]}$ and $\Indi{[v,v+k]}$ have disjoint supports and
\begin{align*}
|\langle\Indi{[u,u+h]},\Indi{[v,v+k]}\rangle_{\Hg}|
  &=\left|H(2H-1)hk\int_{[0,1]^2}|v-u+\eta k-\xi h|^{2H-2}d\xi d\eta\right|\\
	&\leq \left|H(2H-1)hk\int_{[0,1]^2}|\frac{v-u}{2}|^{2H-2}d\xi d\eta\right|,
\end{align*}
which gives \eqref{eq:estimationip}.  To show \eqref{eq:estimationip2p}, we observe that 
\begin{align*}
\langle \Indi{[u,u+h]},\Indi{[0,w]}\rangle_{\Hg}
  &=\frac{1}{2}((u+h)^{2H}-u^{2H}+|w-u|^{2H}-|w-u-h|^{2H}).
\end{align*} 
Using either that,  if $H<\frac{1}{2}$ and all $a,b\geq 0$, 
\begin{align*}
(a+b)^{2H}
  &\leq a^{2H}+b^{2H},
\end{align*}
or that, if $H\geq \frac{1}{2}$ and $0<a\leq b\leq T$,
\begin{eqnarray*}
b^{2H}=a^{2H}+2H\int_0^{b-a}(x+a)^{2H-1}dx\leq a^{2H}+2HT^{2H-1}(b-a),
\end{eqnarray*}
the desired conclusion \eqref{eq:estimationip2p} follows.
\end{proof}

\noindent Next we describe some properties of the conditional variances of general Gaussian vectors. In the sequel, for all $r\in\N$ and all non-negative definite matrices $C\in\R^{r\times r}$, $|C|$ will denote the determinant of $C$ and $\Phi_{C}:\R^{r}\rightarrow\R^r$ will denote the centered Gaussian kernel of dimension $r$ with covariance $C$, defined by 
\begin{align*}
\Phi_{C}(x_{1},\dots, x_{r})
  &:=(2\pi)^{-\frac{r}{2}}|C|^{-\frac{1}{2}}\exp\{-\frac{1}{2}\sum_{i,j=1}^rx_{i}C_{i,j}x_{j}\}.
\end{align*}
Let $\vec{N}=(N_1,\dots, N_r)$ be a centered Gaussian vector of dimension $r\in\N$ and covariance matrix $\Sigma$, defined in $(\Omega, \Fc, \Pb)$. Denote by $\mathcal{G}$ the $\sigma$-algebra generated by $\vec{N}$. If $F$ is a $\mathcal{G}$-measurable, square integrable random variable and $\mathcal{A}$ is a subalgebra of $\mathcal{G}$, the conditional variance of $F$ given $\mathcal{A}$ is the random variable defined by 
\begin{align*}
\text{Var}[F\ |\ \mathcal{A}]
  &:=\E[\ (F-\E[F\ |\ \mathcal{A}])^2|\ \mathcal{A}].
\end{align*}
In the case where $\mathcal{A}$ is generated by random variables $F_1,\dots, F_{n}$, we will use the notation $\text{Var}[F\ |\ F_1,\dots, F_n]$ instead of $\text{Var}[F\ |\ \mathcal{A}]$.  It is well known that in the case where $F_1,\dots, F_n$ are jointly Gaussian, the conditional variance $\text{Var}[F_1\ |\ F_2,\dots, F_n]$ is deterministic. Consequently, by using the fact that $\E[(F-\E[F\ |\ \mathcal{H}])^2]\leq \E[(F-\E[F\ |\ \mathcal{K}])^2]$ for every $\sigma$-algebras $\mathcal{H}$ and $\mathcal{K}$ satisfying $\mathcal{K}\subset\mathcal{H}\subset\mathcal{F}$, we have that
\begin{align}\label{eq:varreduceinfo}
\text{Var}[N_1\ |\ N_2,\dots, N_r]
  &=\E[(N_{1}-\E[N_1\ |\ N_2,\dots, N_r])^2]
	\leq \E[(N_{1}-\E[N_1\ |\ N_3,\dots, N_r])^2]\nonumber\\
	&= \text{Var}[N_1\ |\ N_3,\dots, N_r] \leq \text{Var}[N_1].
\end{align}
In addition, the determinant of $C$ can be represented as
\begin{align}\label{eq:detdecomp}
|\Sigma|
  &=\text{Var}[N_1]\prod_{j=1}^r\text{Var}[N_r\ |\ N_1,\dots, N_{r-1}]. 
\end{align} 
This identity can be easily obtained by first expressing the probability density $\Phi_{\Sigma}$ of $\vec{N}$ as the product of the conditional densities of its components, and then evaluating at zero the resulting decomposition.\\

\noindent Finally we recall the sectorial local non-determinism property for the fractional Brownian motion, which states that there exists a constant $\delta>0$, only depending on $H$, such that for all $n\in\N$ and $t,t_1,\dots, t_{n}\in(0,\infty)$, 
\begin{align}\label{ineq:localnd}
\text{Var}[X_{t}  |\ X_{t_1} ,\dots, X_{t_{n}} ]
  &\geq \delta\min\{|t-t_i|^{2H}\ |\ 0\leq i\leq n\},
\end{align}
where $t_{0}:=0$.

%=============================
\subsection{Fourier representation for the derivatives of the fBm}\label{sec:Fourierderivlocal}
 In Lemma \ref{lem:nonexistenceLell}, it is proved as well that the local time and its derivatives can be represented as
\begin{align}
L_{t}^{(\ell)}(\lettrequivabien)
  &=\int_{\R}\int_{0}^{t}(\textbf{i}\xi)^{\ell}e^{\textbf{i}\xi(X_{s}-\lettrequivabien)}dsd\xi, \label{eq:FourierrepLprime}
\end{align}
meaning that, as $N\to \infty$, the sequence $\int_{-N}^{N}\int_{0}^{t}(\textbf{i}\xi)^{\ell}e^{\textbf{i}\xi(X_{s}-\lettrequivabien)}dsd\xi$, converges in $L^2(\Omega)$ to $L_{t}^{(\ell)}(\lettrequivabien)$. Notice that the type of limit appearing in the representation \eqref{eq:FourierrepLprime} belongs to the class of functions of the form 
\begin{align}\label{eq:indefint}
\lim_{M\rightarrow\infty}\int_{-M}^{M}\int_{\R}\int_{0}^{t}e^{\textbf{i}\xi X_{s}}g(s,\xi,y)ds dyd\xi
  &=\lim_{M\rightarrow\infty}\int_{\R}\int_{-M}^{M}\int_{0}^{t}e^{\textbf{i}\xi X_{s}}g(s,\xi,y)ds d\xi dy,
\end{align}
where $t>0$ and $g:\R^{3}\rightarrow\R$ is such that $g$ is absolutely integrable over $[-M,M]\times\R\times[0,t]$ for all $M>0$, and the limit \eqref{eq:indefint} exists in the $L^{2}(\Omega)$ sense. For convenience on the notation, we will denote the limit of \eqref{eq:indefint} simply by 
\begin{align}\label{eq:indefint2}
\int_{\R^2}\int_{0}^{t}e^{\textbf{i}\xi X_{s}}g(s,\xi,y)ds dyd\xi.
\end{align}
We will often require bounds on the $L^{2}(\Omega)$-norms of expressions of the form \eqref{eq:indefint2}. These type of estimations can be obtained in the following way: for a given $M>0$, we can write
\begin{multline*}
\E\big[\big(\int_{-M}^{M}\int_{\R}\int_{0}^{t}e^{\textbf{i}\xi X_{s}}g(s,\xi,y)ds dyd\xi\big)^2\big]\\
  =\int_{[-M,M]^2}\int_{\R^2}\int_{[0,t]^2}e^{-\frac{1}{2}\vec{\xi}\Lambda \vec{\xi}}g(s,\xi,y)g(\tilde{s},\tilde{\xi},\tilde{y})d\vec{s} d\vec{y}d\vec{\xi},
\end{multline*}  
where $\Lambda(s,\tilde{s})$ is the covariance matrix of $(X_s,X_{\tilde{s}})$ and $\vec{s}:=(s_1,s_2)$, $\vec{\xi}:=(\xi,\tilde{\xi})$ and $\vec{y}:=(y,\tilde{y})$. Thus, provided that $e^{-\frac{1}{2}\vec{\xi}\Lambda \vec{\xi}}g(s,\xi,y)g(\tilde{s},\tilde{\xi},\tilde{y})$ is integrable over $\vec{s}\in[0,T]^2,\vec{\xi},\vec{y}\in\R^{2}$, by the dominated convergence theorem we have that 
\begin{align}\label{eq:indefint3}
\E\big[\big(\int_{\R^2}\int_{0}^{t}e^{\textbf{i}\xi X_{s}}g(s,\xi,y)ds dyd\xi\big)^2\big]
  &=\int_{\R^4}\int_{[0,t]^2}e^{-\frac{1}{2}\vec{\xi}\Lambda \vec{\xi}}g(s,\xi,y)g(\tilde{s},\tilde{\xi},\tilde{y})d\vec{s} d\vec{y}d\vec{\xi}.
\end{align}
Taking this discussion into consideration, in the sequel we will adopt the notation \eqref{eq:indefint2} for describing the limit \eqref{eq:indefint} and use the formula \eqref{eq:indefint3} for describing the associated moment of order 2.
%Indeed, by taking $g(s,\xi,y)=e^{-\textbf{i}\xi\lambda}(\textbf{i}\xi)\phi_{1}(y)$, we recover the right-hand side of \eqref{eq:FourierrepLprime}. 

We are now ready for the proofs Theorems \ref{thm:main2} and \ref{thm:main2p}, which will be presented in the following two sections.
%=======================================================================================================
\section{Proof of Theorem \ref{thm:main2}}\label{sec:firstfluct}
%=======================================================================================================
\noindent Suppose that $0<a\leq H<\frac{1}{2\ell+1}$. In the sequel, for every integrable function $\psi:\R\rightarrow\R$, we denote by $\hat{\psi}$ its Fourier transform. \\

The proofs of \eqref{eq:thmmain22}-\eqref{eq:ucpmain2} rely on a suitable decomposition of the process $G_{t,\lettrequivabien,a}^{(n,\ell)}[g]$, which we describe next. By the Fourier inversion formula, for all $x\in\R$ we have that 
\begin{align*}
g^{(\ell)}(n^ax)
  &=\frac{1}{2\pi}\int_{\R}(\textbf{i}\xi)^{\ell}\hat{g}(\xi)e^{\textbf{i}\xi n^{a}x}d\xi =\frac{1}{2\pi n^{a(\ell+1)}}\int_{\R}(\textbf{i}\xi)^{\ell}\hat{g}(\frac{\xi}{n^a})e^{\textbf{i}\xi x}d\xi.
\end{align*}
Thus, using the fact that 
\begin{align*}
\hat{g}(\frac{\xi}{n^a})
  &=\int_{\R}e^{-\textbf{i}\frac{\xi}{n^a}y}g(y)dy,
\end{align*}
we get that 
\begin{align}\label{eq:gfourier}
g^{(\ell)}(n^ax)
  &=\frac{1}{2\pi n^{a(\ell+1)}}\int_{\R^{2}}(\textbf{i}\xi)^{\ell}e^{\textbf{i}\xi (x-\frac{y}{n^{a}})}g(y)dyd\xi.
\end{align}
As a consequence,
\begin{align}\label{eq:Gtnormdecompzero}
n^{a(\ell+1)-1}G_{t,\lettrequivabien,a}^{(n,\ell)}[g]
  &=n^{a(\ell+1)-1}\sum_{i=2}^{\lfloor nt\rfloor}g^{(\ell)}(n^a(X_{\frac{i-1}{n}}-\lettrequivabien))
  =\frac{1}{2\pi n}\int_{\R^{2}}\sum_{i=2}^{\lfloor nt\rfloor}(\textbf{i}\xi)^{\ell}e^{\textbf{i}\xi (X_{\frac{i-1}{n}}-\lettrequivabien-\frac{y}{n^{a}})}g(y)dyd\xi\nonumber\\
	&=\frac{1}{2\pi}\int_{\R^{2}}\sum_{i=2}^{\lfloor nt\rfloor}\int_{\frac{i-1}{n}}^{\frac{i}{n}}(\textbf{i}\xi)^{\ell}e^{\textbf{i}\xi (X_{\frac{i-1}{n}}-\lettrequivabien-\frac{y}{n^{a}})}g(y)dsdyd\xi.
\end{align}
Using \eqref{eq:Gtnormdecompzero} as well as the Fourier representation \eqref{eq:FourierrepLprime} of $L_t^{(\ell)}(\lettrequivabien)$, we get that
\begin{align}\label{eq:GtnormminusLoct}
n^{a(\ell+1)-1}G_{t,\lettrequivabien,a}^{(n,\ell)}[g]-\mu[g]L_{t}^{(\ell)}(\lettrequivabien)
  &=\frac{1}{2\pi}\int_{\R^{2}}\sum_{i=2}^{\lfloor nt\rfloor}\int_{\frac{i-1}{n}}^{\frac{i}{n}}(\textbf{i}\xi)^{\ell}(e^{\textbf{i}\xi (X_{\frac{i-1}{n}}-\lettrequivabien-\frac{y}{n^{a}})}-e^{\textbf{i}\xi(X_s-\lettrequivabien)})g(y)dsdyd\xi\nonumber\\
	&-\mu[g]L_{\frac{1}{n}}^{(\ell)}(\lettrequivabien)-\mu[g](L_{t}^{(\ell)}(\lettrequivabien)-L_{\frac{\lfloor nt\rfloor}{n}}^{(\ell)}(\lettrequivabien)).
\end{align}

Let us define 
$\Delta_{i,s}X:=X_{s}-X_{\frac{i-1}{n}}$
and 
$\sigma_{i,s}^2
  :=\text{Var}[\Delta_{i,s}X]=\big(s-\frac{i-1}{n}\big)^{2H}.$
The difference of exponentials of the first term in the right hand side can be written as

\begin{align}\label{eq:sumofexponentials}
e^{\textbf{i}\xi (X_{\frac{i-1}{n}}-\frac{y}{n^{a}})}-e^{\textbf{i}\xi X_s}
  &=e^{\textbf{i}\xi(X_{\frac{i-1}{n}}-\frac{y}{n^a})}(e^{-\frac{1}{2}\xi^2\sigma_{s,i}^2}-e^{\textbf{i}\xi\Delta_{i,s}X})
	+e^{\textbf{i}\xi X_{s}}(e^{-\textbf{i}\xi\frac{y}{n^{a}}}-1 )\\
	&+e^{\textbf{i}\xi(X_{\frac{i-1}{n}}-\frac{y}{n^a})}(1-e^{-\frac{1}{2}\xi^2\sigma_{s,i}^2}) ,\nonumber
\end{align}
which by \eqref{eq:GtnormminusLoct}, leads to  
\begin{align}\label{eq:step2decompwithR1}
n^{a(\ell+1)-1}G_{t,\lettrequivabien,a}^{(n,\ell)}[g]-\mu[g]L_{t}^{(\ell)}(\lettrequivabien)
  &=R_{t,0}^{(n)}+R_{t,1}^{(n)}+R_{t,2}^{(n)}+{R}_{t,3}^{(n)},
\end{align}
where  
\begin{align}
R_{t,0}^{(n)}
  &:=-\mu[g] L_{\frac{1}{n}}^{(\ell)}(\lettrequivabien)-\mu[g](L_{t}^{(\ell)}(\lettrequivabien)-L_{\frac{\lfloor nt\rfloor}{n}}^{(\ell)}(\lettrequivabien))\label{eq:R0deff1}\\
R_{t,1}^{(n)}
  &:=\frac{1}{2\pi}\int_{\R^{2}}\sum_{i=2}^{\lfloor nt\rfloor}\int_{\frac{i-1}{n}}^{\frac{i}{n}}(\textbf{i}\xi)^{\ell}e^{\textbf{i}\xi(X_{\frac{i-1}{n}}-\lettrequivabien-\frac{y}{n^a})}(e^{-\frac{1}{2}\xi^2\sigma_{i,s}^2}-e^{\textbf{i}\xi\Delta_{i,s}X})g(y)dsd\xi dy\label{eq:R1deff1}\\
R_{t,2}^{(n)}
  &:=\frac{1}{2\pi}\int_{\R^{2}}\sum_{i=2}^{\lfloor nt\rfloor}\int_{\frac{i-1}{n}}^{\frac{i}{n}}(\textbf{i}\xi)^{\ell}e^{\textbf{i}\xi(X_{\frac{i-1}{n}}-\lettrequivabien-\frac{y}{n^a})}(1-e^{-\frac{1}{2}\xi^2\sigma_{i,s}^2})g(y)dsd\xi dy\label{eq:R2deff1}\\
R_{t,3}^{(n)}
  &:=\frac{1}{2\pi}\int_{\R^{2}}\sum_{i=2}^{\lfloor nt\rfloor}\int_{\frac{i-1}{n}}^{\frac{i}{n}}(\textbf{i}\xi)^{\ell}e^{\textbf{i}\xi (X_{s}-\lettrequivabien)}(e^{-\textbf{i}\xi\frac{y}{n^{a}}}-1)	g(y)dsdyd\xi.\label{eq:R3tildedef}
\end{align}
Notice that the decomposition \eqref{eq:step2decompwithR1} reduces the problem of proving \eqref{eq:thmmain22} to finding bounds for the $L^{2}(\Omega)$-norm of $R_{t,0}^{(n)},\dots, R_{t,4}^{(n)}$ for $t>0$ fixed and $\ell\geq 0$.

Moreover, assume for the moment that the proof of \eqref{eq:thmmain22} is concluded, and also that we have shown that the family of processes
\begin{align}\label{eq:processmain}
n^{a(\ell+1)-1}G_{t,\lettrequivabien,a}^{(n,\ell)}[g]-\mu[g]L_{t}^{(\ell)}(\lettrequivabien)
\end{align}
is tight in the following two cases: (i) when regarded as a collection of random elements with values in $D[\delta, T]$ (that is, the class of c\`adl\`ag mappings on $[\delta,T]$ endowed with the Skorohod topology, for some $0<\delta<T$), in the case where $H<\frac{1}{2\ell+1}$ and $\ell\geq 1$, and (ii) when regarded as a collection of random elements with values in $D[0, T]$ (for some $0<\delta<T$),  in the case where either $H<\frac{1}{2\ell+2}$ and $\ell\geq 1$ or $\ell=0$. 
Using that the finite dimensional distributions of the process 
 \eqref{eq:processmain}
 converge to those of the zero process by \eqref{eq:thmmain22}, and using the classical discussion contained in \cite[p. 124]{Billin} (see also \cite{WW}), we can therefore conclude that:

\begin{enumerate}
\item[-] If $0< T_{1} <T_{2}$, then 
\begin{align*}
\left\{\sup_{t\in[T_1,T_2]}|n^{a(\ell+1)-1}G_{t,\lettrequivabien,a}^{(n,\ell)}[g]-L_{t}^{(\ell)}(\lettrequivabien)\mu[g]|\right\}_{n\geq 1}
  &\stackrel{Law}{\rightarrow}0
	\quad\quad\quad  \text{ as }\  n\rightarrow\infty.
\end{align*}

\item[-]  If $H<\frac{1}{2\ell+2}$  and $T>0$, then 
\begin{align*}
\left\{\sup_{t\in[0,T]}|n^{a(\ell+1)-1}G_{t,\lettrequivabien,a}^{(n,\ell)}[g]-L_{t}^{(\ell)}(\lettrequivabien)\mu[g]|\right\}_{n\geq 1}
  &\stackrel{Law}{\rightarrow}0
	\quad\quad\quad  \text{ as }\  n\rightarrow\infty.
\end{align*}
\end{enumerate}
Parts (i) and (ii) of Theorem \ref{thm:main2} then follow from the fact that every sequence of random variables $\{\xi_{n}\}_{n\geq 1}$ defined in $(\Omega,\mathcal{F},\mathbb{P})$ such that
$
\xi_{n}
  \stackrel{Law}{\rightarrow}L,$ 
for some deterministic $L\in\R$, satisfies as well the convergence 
$\xi_{n}
  \stackrel{\mathbb{P}}{\rightarrow}L$.\\

\noindent In order to examine the tightness property we distinguish between the cases $\ell=0$ and $\ell\geq 1$.  In the case $\ell=0$, the  property follows from Remark \ref{rem:tightness} (Dini's theorem). For handling the case $\ell\geq 1$ we proceed as follows. Let $0\leq T_{1}<T_{2}$ fixed. By \eqref{eq:step2decompwithR1}, it suffices to show that the processes
\begin{align}\label{eq:tightnessgoal2}
\{R_{t,i}^{(n)}\ \processsymb\ t\in[T_1,T_2]\}
\end{align} 
are tight for $i=0,\dots, 3$. Notice that in order to prove such property,  it suffices to show that the processes $\{\tilde{R}_{t,i}^{(n)}\ \processsymb\ t\in[0,T_2-T_{1}]\}$, with $\tilde{R}_t:=R_{T_{1}+t}$ and $i=0,\dots, 3$, are tight. To verify this, we shall use the Billingsley criterion (see \cite[Theorem~12.3]{Billin}),  in order to reduce the problem of proving tightness for \eqref{eq:tightnessgoal2}, to showing that there exist constants $C,p>0$ such that for all $t_1,t_{2}\in[T_1,T_2]$ and $i=0,\dots, 3$,
\begin{align}\label{eq:tightnessgoal1}
\E\left[\Abs{R_{t_2,i}^{(n)}-R_{t_1,i}^{(n)}}^{p}\right]\leq C\rho_n(t_1,t_2)^{\alpha(p)},
\end{align} 
where $\alpha(p)>1$ is some constant only depending on $p$ and $H$, and 
\begin{align*}
\rho_n(t_1,t_2)
  &:=\Abs{\frac{\lfloor nt_{2}\rfloor-\lfloor nt_{1}\rfloor}{n}}\vee|t_2-t_1|.
\end{align*}

In what follows, to keep the length of this paper within bounds we concentrate only on the case where either $p=2$ and $i\neq0$, or when $p\geq 0$ is arbitrary and $i=0$, which are  two cases representative of the difficulty.\\

As a summary of the discussion above,
we obtain that, 
 in order to conclude the proof of Theorem \ref{thm:main2}, it remains
 to check that, with $\beta=\beta(H,\ell)>0$ defined as
\begin{align*}
\beta:=
  &\left\{\begin{array}{lll}\kappa&\text{ if }&H\in[\frac{1}{2(\ell+1)},\frac{1}{2\ell+1})\\\kappa\wedge(\frac{1-2H(\ell+1)}{4H})&\text{ if }& H\in(0,\frac{1}{2(\ell+1)}),\end{array}\right.
\end{align*}
then for all $T>0$, there exists a constant $C_{T}>0$ only depending on $H$, $T$ and $\ell$, such that the following claims hold true: 

\begin{enumerate}
\item[-] For every $\delta>0$, $t_1,t_2\in[0,T]$ and all $i=1,2,3$, we have that
\begin{align}\label{eq:remindsfortight12}
\|R_{t_2,i}^{(n)}-R_{t_1,i}^{(n)}\|_{L^{2}(\Omega)}^2
  & \leq   C_{T}(\|w^{\kappa}g\|_{L^{1}(\R)}^2 + \|g^{(\ell)}\|_{L^{2}(\R)}^2)\nonumber\\
	&\times \left\{\begin{array}{ccc}
\delta^{-H}\rho_n(t_1,t_2)^{2-H(2\ell+2\beta+1)}  &\text{ if }  & t_1,t_2\in[\delta,T]\\
\rho_n(t_1,t_2)^{2-H(2\ell+2\beta+2)}             &\text{ if }  & t_1,t_2\in[0,T].
\end{array}\right.
\end{align}

\item[-] For every $t_1,t_2>0$ and $p\in\N$, 
\begin{align}\label{eq:remindsfortight12tre}
\|R_{t_2,0}^{(n)}-R_{t_1,0}^{(n)}\|_{L^{2p}(\Omega)}^{2p}
  & \leq  C_{T}^p\| g\|_{L^{1}(\R)}^{2p}\rho_n(t_1,t_2)^{2p(1-H(1+\ell))}.
\end{align}

\item[-] For every $t>0$ and $i=0,1,2,3$,
\begin{align}\label{eq:remindsfortight12asd}
\|R_{t,i}^{(n)} \|_{L^{2}(\Omega)}^2
  &\leq C_t (\|w^{\kappa}g\|_{L^{1}(\R)}^2 + \|g^{(\ell)}\|_{L^{2}(\R)}^2)n^{-\kappa a},
\end{align}
for some constant $C_{t}>0$ depending on $t$, but not on $n$ and $g$.

\end{enumerate}

Indeed, the estimate \eqref{eq:thmmain22} follows from \eqref{eq:step2decompwithR1} and \eqref{eq:remindsfortight12asd}. Moreover, if $H(2\ell+1)<1$ and $T_{1}>0$, then by an application of \eqref{eq:remindsfortight12} with $\delta=T_{1}$ and $T=T_{2}$, we obtain
\begin{align}\label{ineq:Billingsleyc1}
\|R_{t_2,i}^{(n)}-R_{t_1,i}^{(n)}\|_{L^{2}(\Omega)}^2
  &\leq C_{T_1}\rho_n(t_1,t_2)^{2-H(2\ell+2\beta+1)},
\end{align}
for $i=1,2, 3$ and $t_1,t_2\in[T_1,T_2]$, where $C_{T_1}$ is a constant independent of $t_1,t_2$ and $n$. Then, using the fact that 
$\beta\leq \kappa<\frac{1-H(2\ell+1)}{2H}$, we deduce that  $2-H(2\ell+2\beta+1)>1$, which by \eqref{ineq:Billingsleyc1}  implies that the Billingsley criterion holds for $R_{t,1}^{(n)}$, $R_{t,2}^{(n)}$ and $R_{t,3}^{(n)}$.  Similarly, if $2H(\ell+1)<1$, then by \eqref{eq:remindsfortight12},
\begin{align}\label{ineq:Billingsleyc2}
\|R_{t_2,i}^{(n)}-R_{t_1,i}^{(n)}\|_{L^{2}(\Omega)}^2
  &\leq C\rho_n(t_1,t_2)^{2-2H(\ell+\beta+1)},
\end{align}
for $i=1,2,3$, where $C$ is a constant independent of $t_1,t_2$ and $n$. Then, using the fact that $\beta\leq\frac{1-2H(\ell+2)}{4H}<\frac{1-2H(\ell+2)}{2H},$ we deduce that $2-2H(\ell+\beta+1)>1$, which by \eqref{ineq:Billingsleyc2} implies that the Billingsley criterion holds for $R_{t,1}^{(n)},\dots, R_{t,3}^{(n)}$. Finally, by applying \eqref{eq:remindsfortight12tre} with $p>\frac{1}{2(1-H(1+\ell))}$, we obtain the Billingsley condition for the process $\{R_{t,0}^{(n)}\ \processsymb\ t\in[T_1,T_2]\}$ in the case $H<\frac{1}{2\ell+1}$, regardless of the value of $T_{1}$.

It thus remain to prove \eqref{eq:remindsfortight12}-\eqref{eq:remindsfortight12asd}. For proving \eqref{eq:remindsfortight12asd} in the case $i=0$, we use the inequality \eqref{eq:momentsofincLt} in Lemma \ref{Lem:existLocalderiv}, as well as the fact that $|\mu[g]|\leq \| g\|_{L^{1}(\R)}$, to deduce that there exists a constant $C>0$, such that for all $t>0$,
\begin{align}\label{inbeq:R0R4l2boundprob1}
\|R_{t,0}^{(n)}\|_{L^{2}(\Omega)}
  &\leq\| g\|_{L^{1}(\R)}(\|L_{\frac{1}{n}}^{(\ell)}(\lettrequivabien)\|_{L^{2}(\Omega)}+\|L_{t}^{(\ell)}(\lettrequivabien)-L_{\frac{\lfloor nt\rfloor}{n}}^{(\ell)}(\lettrequivabien)\|_{L^{2}(\Omega)})\nonumber\\
	&\leq C\| g\|_{L^{1}(\R)}^2 n^{H(2\ell+2)-2}.
\end{align}
Recall that $H(2\ell+2\kappa+1)<1$, so that the power of $n$ in the right hand side satisfies the inequality $H(2\ell+2)-2=1-H-2\kappa+H(2\ell+2\kappa+1)<-2\kappa$. Relation \eqref{eq:remindsfortight12asd} then follows from \eqref{inbeq:R0R4l2boundprob1}. To prove \eqref{eq:remindsfortight12tre} we combine \eqref{eq:momentsofincLt} in Lemma \ref{Lem:existLocalderiv} with the inequality $|\mu[g]|\leq \| g\|_{L^{1}(\R)}$,  to write 
\begin{align*}
\|R_{t_2,0}^{(n)}-R_{t_1,0}^{(n)}\|_{L^{2p}(\Omega)}
  &\leq\| g\|_{L^{1}(\R)}(\|L_{t_2}^{(\ell)}(\lettrequivabien)-L_{t_1}^{(\ell)}(\lettrequivabien)\|_{L^{2p}(\Omega)}+\|L_{\frac{\lfloor nt_2\rfloor}{n}}^{(\ell)}(\lettrequivabien)-L_{\frac{\lfloor nt_1\rfloor}{n}}^{(\ell)}(\lettrequivabien)\|_{L^{2p}(\Omega)})\nonumber\\
	&\leq C\| g\|_{L^{1}(\R)}  \rho_n(t_1,t_2)^{1-H(1+\ell)},
\end{align*}
as required. This finishes the proof of \eqref{eq:remindsfortight12tre} and \eqref{eq:remindsfortight12asd} in the case $i=0$.

Next we prove \eqref{eq:remindsfortight12} and \eqref{eq:remindsfortight12asd} in the case $i=1$. Take $0\leq t_1\leq t_2\leq T$ and notice that, by \eqref{eq:R1deff1},  
\begin{align*}
\|R_{t_2,1}^{(n)}-R_{t_1,1}^{(n)}\|_{L^{2}(\Omega)}^2
  &=\frac{1}{4\pi^2}\int_{\R^{4}}\sum_{2\vee(nt_1)\leq i,j\leq  nt_2}\int_{\frac{i-1}{n}}^{\frac{i}{n}}\int_{\frac{j-1}{n}}^{\frac{j}{n}}(\textbf{i}\xi)^{\ell}(\textbf{i}\tilde{\xi})^{\ell}e^{\textbf{i}\xi(X_{\frac{i-1}{n}}-\lettrequivabien-\frac{y}{n^a})+\textbf{i}\tilde{\xi}(X_{\frac{j-1}{n}}-\lettrequivabien-\frac{\tilde{y}}{n^a})}\\
	&\times (e^{-\frac{1}{2}\xi^2\sigma_{i,s}^2}-e^{\textbf{i}\xi\Delta_{i,s}X}) (e^{-\frac{1}{2}\tilde{\xi}^2\sigma_{j,\tilde{s}}^2}-e^{\textbf{i}\tilde{\xi}\Delta_{j,\tilde{s}}X})g(y)g(\tilde{y})d\vec{s}d\vec{\xi} d\vec{y},
\end{align*}
where $\vec{s}:=(s,\tilde{s})$, $\vec{\xi}=(\xi,\tilde{\xi})$, $\vec{y}:=(y,\tilde{y})$. The sum above can be decomposed as 
\begin{align}\label{eq:R1decompA0A0tilde}
\|R_{t_2,1}^{(n)}-R_{t_1,1}^{(n)}\|_{L^{2}(\Omega)}^2
  &=A_{0}^{(n)}+\tilde{A}_{0}^{(n)},
\end{align}
where 
\begin{align}\label{eq:A0defint1}
A_{0}^{(n)}
  &:=\frac{1}{4\pi^2}\int_{\R^{4}}\sum_{\substack{2\vee(nt_1)\leq i,j\leq  nt_2\\|i-j|\geq 2}}\int_{\frac{i-1}{n}}^{\frac{i}{n}}\int_{\frac{j-1}{n}}^{\frac{j}{n}}(\textbf{i}\xi)^{\ell}(\textbf{i}\tilde{\xi})^{\ell}\E[e^{\textbf{i}\xi(X_{\frac{i-1}{n}}-\lettrequivabien-\frac{y}{n^a})+\textbf{i}\tilde{\xi}(X_{\frac{j-1}{n}}-\lettrequivabien-\frac{\tilde{y}}{n^a})}\nonumber\\
	&\times (e^{-\frac{1}{2}\xi^2\sigma_{i,s}^2}-e^{\textbf{i}\xi\Delta_{i,s}X}) (e^{-\frac{1}{2}\tilde{\xi}^2\sigma_{j,\tilde{s}}^2}-e^{\textbf{i}\tilde{\xi}\Delta_{j,\tilde{s}}X})]g(y)g(\tilde{y})d\vec{s}d\vec{\xi} d\vec{y},
\end{align}
and
\begin{align}\label{eq:A0ndef}
\tilde{A}_{0}^{(n)}
  &:=\frac{1}{4\pi^2}\int_{\R^{4}}\sum_{\substack{2\vee(nt_1) \leq i,j\leq  nt_2\\|i-j|\leq 1}}\int_{\frac{i-1}{n}}^{\frac{i}{n}}\int_{\frac{j-1}{n}}^{\frac{j}{n}}(\textbf{i}\xi)^{\ell}(\textbf{i}\tilde{\xi})^{\ell}\E[e^{\textbf{i}\xi(X_{\frac{i-1}{n}}-\lettrequivabien-\frac{y}{n^a})+\textbf{i}\tilde{\xi}(X_{\frac{j-1}{n}}-\lettrequivabien-\frac{\tilde{y}}{n^a})}\nonumber\\
	&\times (e^{-\frac{1}{2}\xi^2\sigma_{i,s}^2}-e^{\textbf{i}\xi\Delta_{i,s}X}) (e^{-\frac{1}{2}\tilde{\xi}^2\sigma_{j,\tilde{s}}^2}-e^{\textbf{i}\tilde{\xi}\Delta_{j,\tilde{s}}X})]g(y)g(\tilde{y})d\vec{s}d\vec{\xi} d\vec{y}.
\end{align}
In order to bound the term $A_{0}^{(n)}$ above, we proceed as follows. First we notice that the expectation appearing in the right-hand side of \eqref{eq:A0defint1} satisfies 
\begin{multline}\label{eq:P1termpaux1aux1}
|\E[e^{\textbf{i}\xi X_{\frac{i-1}{n}}+\textbf{i}\tilde{\xi}X_{\frac{j-1}{n}}}
(e^{-\frac{1}{2}\xi^2\sigma_{i,s}^2}-e^{\textbf{i}\xi\Delta_{i,s}X}) (e^{-\frac{1}{2}\tilde{\xi}^2\sigma_{j,\tilde{s}}^2}-e^{\textbf{i}\tilde{\xi}\Delta_{j,\tilde{s}}X})]|\\
  \begin{aligned}
	&\leq |e^{-\frac{1}{2}\vec{\xi}^*\Sigma(1,1)\vec{\xi}}-e^{-\frac{1}{2}\vec{\xi}^*\Sigma(1,0)\vec{\xi}}|
	+|e^{-\frac{1}{2}\vec{\xi}^*\Sigma(0,0)\vec{\xi}}-e^{-\frac{1}{2}\vec{\xi}^*\Sigma(0,1)\vec{\xi}}|,
\end{aligned}
\end{multline}
where $\Sigma(z,\tilde{z})$ denote the covariance matrix
\begin{align}\label{eq:SigmaZdef}
\Sigma(z,\tilde{z})
  &=\text{Cov}[(X_{\frac{i-1}{n}}+z\Delta_{i,s}X+(1-z^2)^{\frac{1}{2}}\Delta_{i,s}X^{\prime}, X_{\frac{j-1}{n}}+\tilde{z}\Delta_{j,\tilde{s}}X+(1-z^2)^{\frac{1}{2}}\Delta_{j,\tilde{s}}X^{\prime\prime})]
\end{align}
and $X^{\prime},X^{\prime\prime}$ are independent copies of $X$. In order to bound the right hand side of \eqref{eq:P1termpaux1aux1}, we observe that if $z\in\{0,1\}$ and $\vec{\xi}=(\xi,\tilde{\xi})$,
\begin{align*}
|e^{-\frac{1}{2}\vec{\xi}^*\Sigma(z,1)\vec{\xi}}-e^{-\frac{1}{2}\vec{\xi}^*{\Sigma}(z,0)\vec{\xi}}|
  &\leq |\vec{\xi}^*(\Sigma(z,1)-{\Sigma}(z,0))\vec{\xi}|(e^{-\frac{1}{2}\vec{\xi}^*\Sigma(z,1)\vec{\xi}}+e^{-\frac{1}{2}\vec{\xi}^*{\Sigma}(z,0)\vec{\xi}}).
\end{align*}
We deduce from \eqref{eq:estimationip2p} and the fact that
$\Sigma(z,1)_{1,1}=\Sigma(z,0)_{1,1}$ that there exists a constant $C>0$ only depending on $H$, such that 
\begin{align*}
|\vec{\xi}^*(\Sigma(z,1)-{\Sigma}(z,0))\vec{\xi}|+
  &\leq2|\xi\tilde{\xi}||\Sigma(z,1)_{1,2}-\Sigma(z,0)_{1,2}| + \tilde{\xi}^2|\Sigma(z,1)_{2,2}-\Sigma(z,0)_{2,2}|\\
	&\leq 2|\xi\tilde{\xi}||\E[(X_{\frac{i-1}{n}}+z\Delta_{i,s}X )\Delta_{j,\tilde{s}}X]|\\
	&+\tilde{\xi}^2|\text{Var}[X_{\tilde{s}}]-\text{Var}[X_{\frac{j-1}{n}}]-\text{Var}[\Delta_{j,\tilde{s}}\tilde{X}]|
	\leq Cn^{-(2H)\wedge 1}(|\xi\tilde{\xi}|+|\tilde{\xi}|^2).
\end{align*}
Thus, using the fact that for every $\alpha\in [0,1]$ and $x,y\geq0$, 
\begin{align*}
|e^{-x}-e^{-y}|
  &\leq |x-y|^{\alpha}(e^{-x}+e^{-y}),
\end{align*}
we obtain 
\begin{align}\label{eq:diffFouriertransfGaussian}
|e^{-\frac{1}{2}\vec{\xi}^*\Sigma(z,1)\vec{\xi}}-e^{-\frac{1}{2}\vec{\xi}^*{\Sigma}(z,0)\vec{\xi}}|
  &\leq  Cn^{-(2H\kappa)\wedge \kappa}(|\tilde{\xi}|^{2\kappa}+|\xi\tilde{\xi}|^{\kappa})
	(e^{-\frac{1}{2}\vec{\xi}^*\Sigma(z,1)\vec{\xi}}+e^{-\frac{1}{2}\vec{\xi}^*{\Sigma}(z,0)\vec{\xi}})\nonumber\\
	&=  Cn^{-(2H\kappa)\wedge \kappa}(|\xi|^{\kappa}+|\tilde{\xi}|^{\kappa})|\tilde{\xi}|^{\kappa}
	(e^{-\frac{1}{2}\vec{\xi}^*\Sigma(z,1)\vec{\xi}}+e^{-\frac{1}{2}\vec{\xi}^*{\Sigma}(z,0)\vec{\xi}}).
\end{align}
Therefore, by using  \eqref{eq:A0defint1}, \eqref{eq:P1termpaux1aux1} and \eqref{eq:diffFouriertransfGaussian},
\begin{align}\label{eq:R12bound1aux}
|A_{0}^{(n)}|
  &\leq n^{-(2H\kappa)\wedge \kappa}\int_{\R^{4}}\sum_{\substack{2\vee(nt_1)\leq i,j\leq  nt_2\\|i-j|\geq 2}}\int_{\frac{i-1}{n}}^{\frac{i}{n}}\int_{\frac{j-1}{n}}^{\frac{j}{n}}\sum_{z\in\{0,1\}}(|\xi|^{\kappa}+|\tilde{\xi}|^{\kappa})|\xi|^{\ell}|\tilde{\xi}|^{\ell+\kappa}|g(y)g(\tilde{y})|\nonumber\\
	&\times (e^{-\frac{1}{2}\vec{\xi}^{*}\Sigma(z,1)\vec{\xi}}+e^{-\frac{1}{2}\vec{\xi}^{*}{\Sigma}(z,0)\vec{\xi}})d\vec{s}d\vec{y}d\vec{\xi}\nonumber\\
  &\leq  Cn^{-(2H\kappa)\wedge \kappa} \int_{\R^{4}}\sum_{\substack{2\vee(nt_1)\leq i,j\leq nt_2\\|i-j|\geq 2}}\int_{\frac{i-1}{n}}^{\frac{i}{n}}\int_{\frac{j-1}{n}}^{\frac{j}{n}}\sum_{z\in\{0,1\}}(|\xi|^{\kappa}+|\tilde{\xi}|^{\kappa})|\xi|^{\ell}|\tilde{\xi}|^{\ell+\kappa}|g(y)g(\tilde{y})|\nonumber\\
	&\times  (e^{-\frac{1}{2}\vec{\xi}^{*}\Sigma(z,1)\vec{\xi}}+e^{-\frac{1}{2}\vec{\xi}^{*}{\Sigma}(z,0)\vec{\xi}}) d\vec{s}d\vec{y}d\vec{\xi},
\end{align}
which by equation \eqref{eq:Lemmaauxgoal} in Lemma \ref{lem:techintegralbound}, implies that 
\begin{align}\label{eq:A0nfinalbound}
|A_{0}^{(n)}|
& \leq  C\|g\|_{L^1(\R)}^2n^{-(2H\kappa)\wedge \kappa}  \sum_{\substack{2\vee(nt_1)\leq i,j\leq nt_2\\|i-j|\geq 2}} \big(\frac{i\vee j}{n}\big)^{-H}\big(\frac{i\wedge j\wedge|j-i|}{n}\big)^{-H(2\ell+2\kappa+1)}\nonumber\\
& \leq  C\|g\|_{L^1(\R)}^2n^{-(2H\kappa)\wedge \kappa}  \int_{[\frac{\lfloor nt_1\rfloor}{n},\frac{\lfloor nt_2\rfloor}{n}]^2}(u\vee v)^{-H}(u\wedge v\wedge|v-u|)^{-H(2\ell+2\kappa+1)}du dv.
\end{align} 
Using basic calculus techniques, we can show that there exists a constant $C>0$ only depending on $H$, such that for all $0\leq a<b$ and $\alpha\in(0,1)$,

\begin{align}\label{eq:integralbound}
 \int_{[a,b]^2}(u\vee v)^{-H}(u\wedge v\wedge|v-u|)^{-\alpha}du dv
 \leq C\times 
\left\{\begin{array}{ccc}
\delta^{-H}|b-a|^{2-\alpha}  &\text{ if }  & a,b\in[\delta,T]\\
|b-a|^{2-\alpha-H}             &\text{ if }  & a,b\in[0,T].
\end{array}\right.
\end{align}
By \eqref{eq:A0nfinalbound} and \eqref{eq:integralbound},  there exists a constant $C>0$ such that

\begin{align}\label{eq:boundA0final} 
|A_{0}^{(n)}|
& \leq  Cn^{-(2H\kappa)\wedge \kappa}\|g\|_{L^1(\R)}^2\times 
\left\{\begin{array}{ccc}
\delta^{-H}\rho_n(t_1,t_2)^{2-H(2\ell+2\kappa+1)}  &\text{ if }  & t_1,t_2\in[\delta,T]\\
\rho_n(t_1,t_2)^{2-2H(\ell+\kappa+1)}             &\text{ if }  & t_1,t_2\in[0,T].
\end{array}\right.
\end{align}

\noindent Next we bound $\tilde{A}_{0}^{(n)}$. To this end, we notice that by \eqref{eq:A0ndef},
\begin{align*}
|\tilde{A}_{0}^{(n)}|
  &\leq \frac{1}{4\pi^2}\int_{\R^{4}}\sum_{\substack{2\vee(nt_1)\leq i,j\leq  nt_2\\|i-j|\leq 1}}\sum_{z,\tilde{z}\in\{0,1\}}\int_{\frac{i-1}{n}}^{\frac{i}{n}}\int_{\frac{j-1}{n}}^{\frac{j}{n}}|\xi \tilde{\xi}|^{\ell}e^{-\frac{1}{2}\vec{\xi}^*\Sigma(z,\tilde{z})\vec{\xi}}|g(y)g(\tilde{y})|d\vec{s}d\vec{\xi} d\vec{y}.
\end{align*}
Therefore, using equation \eqref{eq:Lemmaauxgoal2} in Lemma \ref{lem:techintegralbound},  we get that 
\begin{align*}
|\tilde{A}_{0}^{(n)}|
  &\leq C\|g\|_{L^1(\R)}^2n^{2H(\ell+1)-2}\sum_{\substack{2\vee(nt_1)\leq i,j\leq  nt_2\\|i-j|\leq 1}}(i\vee j)^{-H}.
\end{align*}

\noindent Thus, using the fact that every $i,j\geq 2$ such that $|j-i|\leq 1$, satisfies the inequality 
\begin{align}\label{eq:bounimaxjint}
n^{H(2\ell+2\kappa+1)}\big(\frac{i\vee j}{n}\big)^{-H}
  \leq C\big(\frac{i\vee j}{n}\big)^{-H}\big(\frac{1+ |j-i|}{n}\big)^{-H(2\ell+2\kappa+1)},
\end{align}
for some $C>0$ only depending on $H$, we obtain
\begin{align*} 
|\tilde{A}_{0}^{(n)}|
  &\leq C\|g\|_{L^1(\R)}^2n^{-2H\kappa-2}\sum_{\substack{2\vee(nt_1)\leq i,j\leq  nt_2\\|i-j|\leq 1}}(i\vee j)^{-H}\big(\frac{1+|j-i|}{n}\big)^{-H(2\ell+2\kappa+1)}\\
	&\leq C\|g\|_{L^1(\R)}^2n^{-2H\kappa}\int_{[\frac{\lfloor nt_1\rfloor}{n},\frac{\lfloor nt_2\rfloor}{n}]^2}(u\vee v)^{-H}|v-u|^{-H(2\ell+2\kappa+1)}dudv\\
	&\leq C\|g\|_{L^1(\R)}^2n^{-2H\kappa}\int_{[\frac{\lfloor nt_1\rfloor}{n},\frac{\lfloor nt_2\rfloor}{n}]^2}(u\vee v)^{-H}(u\wedge v\wedge|v-u|)^{ -H(2\ell+2\kappa+1)}dudv.
\end{align*}
Combining the previous inequality with \eqref{eq:integralbound}, we obtain 
\begin{align}\label{eq:A0tildefinalbound}
|\tilde{A}_{0}^{(n)}|
  &\leq  C\|g\|_{L^1(\R)}^2n^{-2H\kappa}\times
\left\{\begin{array}{ccc}
\delta^{-H}\rho_n(t_1,t_2)^{2-H(2\ell+2\kappa+1)}  &\text{ if }  & t_1,t_2\in[\delta,T]\\
\rho_n(t_1,t_2)^{2-2H(\ell+\kappa+1)}             &\text{ if }  & t_1,t_2\in[0,T].
\end{array}\right.
\end{align}
From \eqref{eq:R1decompA0A0tilde}, \eqref{eq:boundA0final} and \eqref{eq:A0tildefinalbound}, we obtain  
\begin{align}\label{eq:R1nfinalbound}
\|R_{t_2,1}^{(n)}-R_{t_1,1}^{(n)}\|_{L^{2}(\Omega)}^2
& \leq  Cn^{-(2H\kappa)\wedge\kappa}\|g\|_{L^1(\R)}^2	\left\{\begin{array}{ccc}
\delta^{-H}\rho_n(t_1,t_2)^{2-H(2\ell+2\kappa+1)}  &\text{ if }  & t_1,t_2\in[\delta,T]\\
\rho_n(t_1,t_2)^{2-2H(\ell+\kappa+1)}             &\text{ if }  & t_1,t_2\in[0,T].
\end{array}\right.
\end{align} 
Relations \eqref{eq:remindsfortight12} and \eqref{eq:remindsfortight12asd} for the case $i=1$ follow from \eqref{eq:R1nfinalbound}.\\

%=========================================================
\noindent Next we prove \eqref{eq:remindsfortight12} and \eqref{eq:remindsfortight12asd} the case $i=2$. By \eqref{eq:R2deff1},

\begin{align*}
\|R_{t_2,2}^{(n)}-R_{t_1,2}^{(n)}\|_{L^{2}(\Omega)}^2
  &=\frac{(-1)^{\ell}}{4\pi^2}\int_{\R^{4}}\sum_{2\vee(nt_1)\leq i,j\leq nt_2}\int_{\frac{i-1}{n}}^{\frac{i}{n}}\int_{\frac{j-1}{n}}^{\frac{j}{n}}(\xi \tilde{\xi})^{\ell}\E[e^{\textbf{i}\xi(X_{\frac{i-1}{n}}-\lettrequivabien-\frac{y}{n^a})+\textbf{i}\tilde{\xi}(X_{\frac{j-1}{n}}-\lettrequivabien-\frac{\tilde{y}}{n^a})}]\\
	&\times(1-e^{-\frac{1}{2}\xi^2\sigma_{i,s}^2})(1-e^{-\frac{1}{2}\tilde{\xi}^2\sigma_{j,\tilde{s}}^2})g(y)g(\tilde{y})d\vec{s}d\vec{y}d\vec{\xi}.
\end{align*}
As before, we decompose this sum as  
\begin{align}\label{eq:R2A5Atilde5}
\|R_{t_2,2}^{(n)}-R_{t_1,2}^{(n)}\|_{L^{2}(\Omega)}^2
  &=A_{5}^{(n)}+\tilde{A}_5^{(n)},
\end{align}
where
\begin{align*}
A_{5}^{(n)}
  &:=\frac{(-1)^{\ell}}{4\pi^2}\int_{\R^{4}}\sum_{\substack{2\vee(nt_1)\leq i,j\leq nt_2\\|i-j|\geq 2}}\int_{\frac{i-1}{n}}^{\frac{i}{n}}\int_{\frac{j-1}{n}}^{\frac{j}{n}}(\xi \tilde{\xi})^{\ell}\E[e^{\textbf{i}\xi(X_{\frac{i-1}{n}}-\lettrequivabien-\frac{y}{n^a})+\textbf{i}\tilde{\xi}(X_{\frac{j-1}{n}}-\lettrequivabien-\frac{\tilde{y}}{n^a})}]\\
	&\times(1-e^{-\frac{1}{2}\xi^2\sigma_{i,s}^2})(1-e^{-\frac{1}{2}\tilde{\xi}^2\sigma_{j,\tilde{s}}^2})g(y)g(\tilde{y})d\vec{s}d\vec{y}d\vec{\xi}
\end{align*}
and 
\begin{align}\label{eq:A5tildedef}
\tilde{A}_{5}^{(n)}
  &:=\frac{(-1)^{\ell}}{4\pi^2}\int_{\R^{4}}\sum_{\substack{2\vee(nt_1)\leq i,j\leq nt_2\\|i-j|\leq 1}}\int_{\frac{i-1}{n}}^{\frac{i}{n}}\int_{\frac{j-1}{n}}^{\frac{j}{n}}(\xi \tilde{\xi})^{\ell}\E[e^{\textbf{i}\xi(X_{\frac{i-1}{n}}-\lettrequivabien-\frac{y}{n^a})+\textbf{i}\tilde{\xi}(X_{\frac{j-1}{n}}-\lettrequivabien-\frac{\tilde{y}}{n^a})}]\nonumber\\
	&\times(1-e^{-\frac{1}{2}\xi^2\sigma_{i,s}^2})(1-e^{-\frac{1}{2}\tilde{\xi}^2\sigma_{j,\tilde{s}}^2})g(y)g(\tilde{y})d\vec{s}d\vec{y}d\vec{\xi}.
\end{align}
To bound $A_{5}^{(n)}$ we proceed as follows. Let $\Lambda$ denote the covariance matrix of $(X_{\frac{i-1}{n}},X_{\frac{j-1}{n}})$. Using the inequality
\begin{align}\label{eq:auxexpfuture}
|e^{-x}-1|
  \leq 2|x|^{\frac{\kappa}{2}},
\end{align}
valid for every $x\geq0$, we deduce that
\begin{align}\label{eq:R2boundpref}
|A_{5}^{(n)}|
  &\leq C\int_{\R^{4}}\sum_{\substack{2\leq i,j\leq nt\\|i-j|\geq 2}}\sigma_{s,i}^{\kappa}\sigma_{\tilde{s},j}^{\kappa}\int_{\frac{i-1}{n}}^{\frac{i}{n}}\int_{\frac{j-1}{n}}^{\frac{j}{n}}|\xi \tilde{\xi}|^{\ell+\kappa}|g(y)g(\tilde{y})|e^{-\frac{1}{2}\vec{\xi}^*\Lambda \vec{\xi}}d\vec{s}d\vec{y}d\vec{\xi}.
\end{align}
Therefore, using Lemma \ref{lem:techintegralbound}, as well as the fact that $\sigma_{s,i}^2,\sigma_{\tilde{s},j}^2\leq n^{-2H}$, we get
\begin{align*}
|A_{5}^{(n)}|
	&\leq Cn^{-2H\kappa}\|g\|_{L^{1}(\R)}^2\frac{1}{n^2}\sum_{\substack{2\vee(nt_1)\leq i,j\leq nt_2\\|i-j|\geq 2}}\big(\frac{i\vee j}{n} \big)^{-H}\big(\frac{i\wedge j\wedge|j-i|}{n}\big)^{-H(2\ell+2\kappa+1)}\\
	&\leq Cn^{-2H\kappa}\|g\|_{L^{1}(\R)}^2  \int_{[\frac{\lfloor nt_1\rfloor}{n},\frac{\lfloor nt_2\rfloor}{n}]^2}(u\vee v)^{-H}(u\wedge v\wedge|v-u|)^{-H(2\ell+2\kappa+1)}dudv.
\end{align*}
Thus, using \eqref{eq:integralbound} we get that 
\begin{align}\label{eq:A5finalbound}
|A_{5}^{(n)}|
	&\leq Cn^{-2H\kappa}\|g\|_{L^{1}(\R)}^2
	\left\{\begin{array}{ccc}
\delta^{-H}\rho_{n}(t_1,t_2)^{2-H(2\ell+2\kappa+1)}  &\text{ if }  & t_1,t_2\in[\delta,T]\\
\rho_{n}(t_1,t_2)^{2-H(2\ell+2\kappa+2)}             &\text{ if }  & t_1,t_2\in[0,T].
\end{array}\right.
\end{align}

\noindent For handling the term $|\tilde{A}_{5}^{(n)}|$, the Fourier transform approach does not give sharp enough bounds, so we will undo the Fourier transform procedure in the following way: first we write 
\begin{align*}
\tilde{A}_{5}^{(n)}
  &=\frac{(-1)^{\ell}}{4\pi^2}\int_{\R^{4}}\sum_{\substack{2\vee(nt_1)\leq i,j\leq nt_2\\|i-j|\leq 1}}\int_{\frac{i-1}{n}}^{\frac{i}{n}}\int_{\frac{j-1}{n}}^{\frac{j}{n}}(\xi \tilde{\xi})^{\ell}\E[e^{\textbf{i}\xi(X_{\frac{i-1}{n}}-\lettrequivabien-\frac{y}{n^H})+\textbf{i}\tilde{\xi}(X_{\frac{j-1}{n}}-\lettrequivabien-\frac{\tilde{y}}{n^H})}]\\
	&\times(1-e^{\textbf{i}\xi\Delta_{i,s}X^{\prime}})(1-e^{\textbf{i}\tilde{\xi}\Delta_{j,\tilde{s}}X^{\prime\prime}})g(y)g(\tilde{y})d\vec{s}d\vec{y}d\vec{\xi},
\end{align*}
where $X^{\prime}$ and $X^{\prime\prime}$ are independent copies of $X$. Then, by  \eqref{eq:gfourier}, 
\begin{align*}
\tilde{A}_{5}^{(n)}
  &=  n^{2a(\ell+1)}\sum_{\substack{2\vee(nt_1)\leq i,j\leq nt_2\\|i-j|\leq 1}}\int_{\frac{i-1}{n}}^{\frac{i}{n}}\int_{\frac{j-1}{n}}^{\frac{j}{n}} \E\big[\big(g^{(\ell)}(n^H(X_{\frac{i-1}{n}}-\lettrequivabien))-g^{(\ell)}(n^H(X_{\frac{i-1}{n}}+\Delta_{i,s}X^{\prime}-\lettrequivabien))\big)\\
	&\quad\quad\quad\quad\times \big(g^{(\ell)}(n^H(X_{\frac{j-1}{n}}-\lettrequivabien))-g^{(\ell)}(n^H(X_{\frac{j-1}{n}}+\Delta_{j,\tilde{s}}X^{\prime\prime}-\lettrequivabien))\big)\big]d\vec{s}.
\end{align*}
Applying Cauchy-Schwarz in the previous inequality we deduce that
\begin{align}\label{eq:tildeA5firstbound}
|\tilde{A}_{5}^{(n)}|
  &\leq  n^{2a(\ell+1)}\sum_{\substack{2\vee(nt_1)\leq i,j\leq nt_2\\|i-j|\leq 1}}\int_{\frac{i-1}{n}}^{\frac{i}{n}}\int_{\frac{j-1}{n}}^{\frac{j}{n}} \E\big[\big(g^{(\ell)}(n^H(X_{\frac{i-1}{n}}-\lettrequivabien))-g^{(\ell)}(n^H(X_{\frac{i-1}{n}}+\Delta_{i,s}X^{\prime}-\lettrequivabien))\big)^2]^{\frac{1}{2}}\nonumber\\
	&\quad\quad\quad\quad\times \E\big[\big(g^{(\ell)}(n^H(X_{\frac{j-1}{n}}-\lettrequivabien))-g^{(\ell)}(n^H(X_{\frac{j-1}{n}}+\Delta_{j,\tilde{s}}X^{\prime\prime}-\lettrequivabien))\big)^2\big]^{\frac{1}{2}}d\vec{s}.
\end{align}
The two expectations in the right-hand side can be bounded in the following manner
\begin{multline*}
\E\big[\big(g^{(\ell)}(n^a(X_{\frac{i-1}{n}}-\lettrequivabien))-g^{(\ell)}(n^a(X_{\frac{i-1}{n}}+\Delta_{i,s}X^{\prime}-\lettrequivabien))\big)^2]\\
\begin{aligned}
  &\leq 2\E\big[ g^{(\ell)}(n^a(X_{\frac{i-1}{n}}-\lettrequivabien))^2]
	+2\E\big[ g^{(\ell)}(n^a(X_{\frac{i-1}{n}}+\Delta_{i,s}X^{\prime}-\lettrequivabien))^2]\\
	&= 2\int_{\R}g^{(\ell)}(x)^2(\phi_{n^{2a-2H}( i-1 )^{2H}}(x)+\phi_{n^{2a-2H}((i-1)^{2H}+ (ns-i+1)^{2H})}(x))dx\\
	&\leq  2\|g^{(\ell)}\|_{L^2(\R)}^2\phi_{n^{2a-2H}( i-1 )^{2H}}(0),
\end{aligned}
\end{multline*}
and thus,
\begin{align*}
\E\big[\big(g^{(\ell)}(n^a(X_{\frac{i-1}{n}}-\lettrequivabien))-g^{(\ell)}(n^a(X_{\frac{i-1}{n}}+\Delta_{i,s}X^{\prime}-\lettrequivabien))\big)^2]
  &\leq C n^{H-a}i^{-H}\|g^{(\ell)}\|_{L^2(\R)}^2.	
 \end{align*}
Similarly, we have that
\begin{align*}
\E\big[\big(g^{(\ell)}(n^a(X_{\frac{j-1}{n}}-\lettrequivabien))-g^{(\ell)}(n^a(X_{\frac{j-1}{n}}+\Delta_{j,\tilde{s}}X^{\prime\prime}-\lettrequivabien))\big)^2\big]
  &\leq C n^{ H- a}j^{-H}\|g^{(\ell)}\|_{L^2(\R)}^2.	
\end{align*}
Therefore, by \eqref{eq:tildeA5firstbound}, there exist $C,C^{\prime}>0$  such that
\begin{align*}
|\tilde{A}_{5}^{(n)}|
  &\leq  Cn^{2a \ell-2+2H}\sum_{\substack{2\vee(nt_1)\leq i,j\leq nt_2\\|i-j|\leq 1}}(ij)^{-\frac{H}{2}}\|g^{(\ell)}\|_{L^2(\R)}^2\\
  &\leq  C^{\prime}n^{2H(\ell+1)-2}\|g^{(\ell)}\|_{L^2(\R)}^2\sum_{\substack{2\vee(nt_1)\leq i,j\leq nt_2\\|i-j|\leq 1}}\big(\frac{i\vee j}{n}\big)^{-H},
\end{align*}
where in the last inequality we used the condition $a\leq H$. Combining the previous inequality with \eqref{eq:bounimaxjint}, we obtain
\begin{align*}
|\tilde{A}_{5}^{(n)}|
	&\leq C\|g^{(\ell)}\|_{L^2(\R)}^2n^{-2H\kappa-2}\sum_{\substack{2\vee(nt_1)\leq i,j\leq  nt_2\\|i-j|\leq 1}}\big(\frac{i\vee j}{n}\big)^{-H}\big(\frac{1+|j-i|}{n}\big)^{-2H(\ell+\kappa+1)}\\
	&\leq C\|g^{(\ell)}\|_{L^2(\R)}^2n^{-2H\kappa}\int_{[\frac{\lfloor nt_1\rfloor}{n},\frac{\lfloor nt_2\rfloor}{n}]^2}(u\vee v)^{-H}|v-u|^{-2H(\ell+\kappa+1)}dudv\\
	&\leq C\|g^{(\ell)}\|_{L^2(\R)}^2n^{-2H\kappa}\int_{[\frac{\lfloor nt_1\rfloor}{n},\frac{\lfloor nt_2\rfloor}{n}]^2}(u\vee v)^{-H}(u\wedge v\wedge|v-u|)^{-2H(\ell+\kappa+1)}dudv.
\end{align*}
Thus, by \eqref{eq:integralbound} we conclude that 
\begin{align}\label{eq:Atilde5finalbound} 
|\tilde{A}_{5}^{(n)}|
 \leq C\|g^{(\ell)}\|_{L^2(\R)}^2n^{-2H\kappa}
\left\{\begin{array}{ccc}
\delta^{-H}\rho_n(t_1,t_2)^{2-H(2\ell+2\kappa+1)}  &\text{ if }  & t_1,t_2\in[\delta,T]\\
\rho_n(t_1,t_2)^{2-H(2\ell+2\kappa+2)}             &\text{ if }  & t_1,t_2\in[0,T].
\end{array}\right.
\end{align}
From \eqref{eq:A5finalbound} and \eqref{eq:Atilde5finalbound}, we conclude that 
\begin{multline}\label{ineq:case1}
\|R_{t_2,2}^{(n)}-R_{t_1,2}^{(n)}\|_{L^{2}(\Omega)}^2\\
\begin{aligned}
 \leq C(\|g\|_{L^1(\R)}^2+\|g^{(\ell)}\|_{L^2(\R)}^2)n^{-2H\kappa}
\left\{\begin{array}{ccc}
\delta^{-H}\rho_n(t_1,t_2)^{2-H(2\ell+2\kappa+1)}  &\text{ if }  & t_1,t_2\in[\delta,T]\\
\rho_n(t_1,t_2)^{2-H(2\ell+2\kappa+2)}             &\text{ if }  & t_1,t_2\in[0,T].
\end{array}\right.
\end{aligned}
\end{multline}
Relations \eqref{eq:remindsfortight12} and \eqref{eq:remindsfortight12asd} in the case $i=2$ are obtained by combining \eqref{eq:R2A5Atilde5}, \eqref{eq:A5finalbound} and \eqref{ineq:case1}.\\

\noindent It thus remain to prove \eqref{eq:remindsfortight12} and \eqref{eq:remindsfortight12asd} in the case  $i=3$. Notice that by \eqref{eq:R3tildedef},
\begin{align*}
\|R_{t_2,3}^{(n)}-R_{t_2,3}^{(n)}\|_{L^{2}(\Omega)}^2
  &=\frac{1}{4\pi^2}\int_{\R^{4}}\sum_{2\vee(nt_1)\leq i,j\leq nt_2}\int_{\frac{i-1}{n}}^{\frac{i}{n}}\int_{\frac{j-1}{n}}^{\frac{j}{n}}(\textbf{i}\xi)^{\ell}(\textbf{i}\tilde{\xi})^{\ell}\E[e^{\textbf{i}\xi X_{s}+\textbf{i}\tilde{\xi} X_{\tilde{s}}}]\\
	&\times (e^{-\textbf{i}\xi\frac{y}{n^{a}}}-1 )(e^{-\textbf{i}\tilde{\xi}\frac{\tilde{y}}{n^{a}}}-1 )	g(y)g(\tilde{y})d\vec{s}d\vec{y}d\vec{\xi}.
 \end{align*}
Thus, we can write
\begin{align}\label{eqLnormR3decompprob1}
\|R_{t_2,3}^{(n)}-R_{t_2,3}^{(n)}\|_{L^{2}(\Omega)}^2
  &=A_{6}^{(n)}+\tilde{A}_{6}^{(n)},
\end{align}
where 
\begin{align*}
A_{6}^{(n)}
  &=\frac{1}{4\pi^2}\int_{\R^{4}}\sum_{\substack{2\vee(nt_1)\leq i,j\leq nt_2\\|i-j|\geq 2}}\int_{\frac{i-1}{n}}^{\frac{i}{n}}\int_{\frac{j-1}{n}}^{\frac{j}{n}}(\textbf{i}\xi)^{\ell}(\textbf{i}\tilde{\xi})^{\ell}\E[e^{\textbf{i}\xi X_{s}+\textbf{i}\tilde{\xi} X_{\tilde{s}}}]\\
	&\times (e^{-\textbf{i}\xi\frac{y}{n^{a}}}-1 )(e^{-\textbf{i}\tilde{\xi}\frac{\tilde{y}}{n^{a}}}-1 )	g(y)g(\tilde{y})d\vec{s}d\vec{y}d\vec{\xi}\\
\tilde{A}_{6}^{(n)}
  &=\frac{1}{4\pi^2}\int_{\R^{4}}\sum_{\substack{2\vee(nt_1)\leq i,j\leq nt_2\\|i-j|\leq 1}}\int_{\frac{i-1}{n}}^{\frac{i}{n}}\int_{\frac{j-1}{n}}^{\frac{j}{n}}(\textbf{i}\xi)^{\ell}(\textbf{i}\tilde{\xi})^{\ell}\E[e^{\textbf{i}\xi X_{s}+\textbf{i}\tilde{\xi} X_{\tilde{s}}}]\\
	&\times (e^{-\textbf{i}\xi\frac{y}{n^{a}}}-1 )(e^{-\textbf{i}\tilde{\xi}\frac{\tilde{y}}{n^{a}}}-1 )	g(y)g(\tilde{y})d\vec{s}d\vec{y}d\vec{\xi}.
\end{align*}
 To bound the term $A_{6}^{(n)}$, we notice that 
\begin{align}\label{eq:complexexpbound}
|e^{-\textbf{i}x}-1|
  \leq 2 |x|^{\kappa},
\end{align}
for every $x\in\R$. From here it follows that 
\begin{align}\label{eq:R3boundprefs}
A_{6}^{(n)}
  &\leq Cn^{-2a \kappa }\int_{\R^{4}}\sum_{\substack{2\vee(nt_1)\leq i,j\leq nt_2\\|i-j|\geq 2}}\int_{\frac{i-1}{n}}^{\frac{i}{n}}\int_{\frac{j-1}{n}}^{\frac{j}{n}}
	|g(y)g(\tilde{y})||y\tilde{y}|^{ a}|\xi\tilde{\xi}|^{\ell+a}e^{-\frac{1}{2}\vec{\xi}^*\Sigma(1,1) \vec{\xi}}
	d\vec{s}d\vec{y}d\vec{\xi}.
\end{align}
Consequently, by first applying equation \eqref{eq:Lemmaauxgoal} in Lemma \ref{lem:techintegralbound} to the right hand side of \eqref{eq:R3boundprefs}, and then the condition $a\leq H$, we get 
\begin{align*}
A_{6}^{(n)}
  &\leq C\|w^{ \kappa}g\|_{L^{1}(\R)}^2n^{-2a \kappa-2}\sum_{\substack{2\vee(nt_1)\leq i,j\leq nt_2\\|i-j|\geq 2}}\big(\frac{i\vee j}{n}\big)^{-H}\big(\frac{i\wedge j\wedge|j-i|}{n}\big)^{-H(2\ell+1+2\kappa)}\\
	&\leq C\|w^{ \kappa}g\|_{L^{1}(\R)}^2n^{-2a \kappa }\int_{[\frac{\lfloor nt_1\rfloor}{n},\frac{\lfloor nt_2\rfloor}{n}]^2}(u\vee v\big)^{-H} (u\wedge w\wedge|v-u|)^{-H(2\ell+1+2\kappa)}dudv.
\end{align*}
Combining the previous inequality with \eqref{eq:integralbound}, we conclude that 
\begin{align}\label{eq:A6boundfinal}
A_{6}^{(n)}
  &\leq C\|w^{ \kappa}g\|_{L^{1}(\R)}^2n^{-2a \kappa }\times \left\{\begin{array}{ccc}
\delta^{-H}\rho_n(t_1,t_2)^{2-H(2\ell+2\kappa+1)}  &\text{ if }  & t_1,t_2\in[\delta,T]\\
\rho_n(t_1,t_2)^{2-H(2\ell+2\kappa+2)}             &\text{ if }  & t_1,t_2\in[0,T].
\end{array}\right.
\end{align}
To handle the term $\tilde{A}_{6}^{(n)}$, we use the bound \eqref{eq:complexexpbound} to get 
\begin{align}\label{eq:R3boundprefsprime}
\tilde{A}_{6}^{(n)}
  &\leq Cn^{-2a\kappa}\int_{\R^{4}}\sum_{\substack{2\vee(nt_1)\leq i,j\leq nt_2\\|i-j|\leq 1}}\int_{\frac{i-1}{n}}^{\frac{i}{n}}\int_{\frac{j-1}{n}}^{\frac{j}{n}}
	|g(y)g(\tilde{y})||y\tilde{y}|^{\kappa}|\xi\tilde{\xi}|^{\ell+\kappa}e^{-\frac{1}{2}\vec{\xi}^*\Sigma(1,1) \vec{\xi}}
	d\vec{s}d\vec{y}d\vec{\xi}.
\end{align}
Thus, proceeding as before, we can apply equation \eqref{eq:Lemmaauxgoal2} in Lemma \ref{lem:techintegralbound} as well as \eqref{eq:bounimaxjint},  to obtain the inequality 
\begin{align*}
|\tilde{A}_{6}^{(n)}|
  &\leq C\|w^{\kappa}g\|_{L^{1}(\R)}^2n^{-2a\kappa+H(2\ell+2\kappa+1)-2}\sum_{\substack{2\vee(nt_1)\leq i,j\leq nt_2\\|i-j|\leq 1}}\big(\frac{i\vee j}{n}\big)^{-H}\\
	&\leq C\|w^{\kappa}g\|_{L^{1}(\R)}^2n^{-2a\kappa-2}\sum_{\substack{2\vee(nt_1)\leq i,j\leq nt_2\\|i-j|\leq 1}}\big(\frac{i\vee j}{n}\big)^{-H}
	\big(\frac{1+|i-j|}{n}\big)^{-H(2\ell +2\kappa+1)}\\
	&\leq C\|w^{\kappa}g\|_{L^{1}(\R)}^2n^{-2a\kappa}\int_{[\frac{\lfloor nt_1\rfloor}{n},\frac{\lfloor nt_2\rfloor}{n}]^2} (u\vee v)^{-H}
  |v-u|^{-H(2\ell+2\kappa +1)}dudv.
\end{align*}
Thus, by \eqref{eq:integralbound} we get 
\begin{align}\label{eq:Atilde6boundfinal}
|\tilde{A}_{6}^{(n)}|
  &\leq C\|w^{ \kappa}g\|_{L^{1}(\R)}^2n^{-2a \kappa}\times \left\{\begin{array}{ccc}
\delta^{-H}\rho_n(t_1,t_2)^{2-H(2\ell+2\kappa+1)}  &\text{ if }  & t_1,t_2\in[\delta,T]\\
\rho_n(t_1,t_2)^{2-H(2\ell+2\kappa+2)}             &\text{ if }  & t_1,t_2\in[0,T].
\end{array}\right.
\end{align}
Combining \eqref{eqLnormR3decompprob1}, \eqref{eq:A6boundfinal} and \eqref{eq:Atilde6boundfinal}, we obtain 
\begin{align}\label{eq:Rt3case1final}
\|{R}_{t_2,3}^{(n)}-{R}_{t_1,3}^{(n)}\|_{L^{2}(\Omega)}^{2}
  &\leq C\|w^{ \kappa}g\|_{L^{1}(\R)}^2n^{-2a \kappa}\times 
	\left\{\begin{array}{ccc}
\delta^{-H}\rho_n(t_1,t_2)^{2-H(2\ell+2\kappa+1)}  &\text{ if }  & t_1,t_2\in[\delta,T]\\
\rho_n(t_1,t_2)^{2-H(2\ell+2\kappa+2)}             &\text{ if }  & t_1,t_2\in[0,T],
\end{array}\right.
\end{align}
which gives \eqref{eq:remindsfortight12} and \eqref{eq:remindsfortight12asd} in the case  $i=3$. The proof is now complete.

%=======================================================================================================
%=======================================================================================================
\section{Proof of Theorem \ref{thm:main2p}}\label{sec:firstfluc2t2} 
%=======================================================================================================
%=======================================================================================================
\noindent Suppose that $H<\frac{1}{2\ell+3}$ and let $\kappa\in(0,\frac{1}{2})$ be such that $H(2\ell+2\kappa+3)<1$. Define $R_{t,3}^{(n)}$  by \eqref{eq:R3tildedef}. Using the identity 
\begin{align*}
e^{\textbf{i}\xi (X_{s}-\lettrequivabien)}(e^{-\textbf{i}\xi\frac{y}{n^{a}}}-1)
  &=e^{\textbf{i}\xi (X_{s}-\lettrequivabien)}(e^{-\textbf{i}\xi\frac{y}{n^{a}}}-1+\textbf{i}\xi\frac{y}{n^{a}})-\textbf{i}\xi\frac{y}{n^{a}}e^{\textbf{i}\xi (X_{s}-\lettrequivabien)},
\end{align*}
we can show that  
\begin{align}\label{eq:R3tndecomp1q}
R_{t,3}^{(n)}
  &=R_{t,4}^{(n)}-\frac{1}{2\pi n^a}\int_{\R^{2}}\sum_{i=2}^{\lfloor nt\rfloor}\int_{\frac{i-1}{n}}^{\frac{i}{n}}(\textbf{i}\xi)^{\ell+1}e^{\textbf{i}\xi (X_{s}-\lettrequivabien)}yg(y)dsdyd\xi,
\end{align}
where 
\begin{align*}
R_{t,4}^{(n)}
  &:=\frac{1}{2\pi}\int_{\R^{2}}\sum_{i=2}^{\lfloor nt\rfloor}\int_{\frac{i-1}{n}}^{\frac{i}{n}}(\textbf{i}\xi)^{\ell}e^{\textbf{i}\xi (X_{s}-\lettrequivabien)}(e^{-\textbf{i}\xi\frac{y}{n^{a}}}-1+\textbf{i}\xi\frac{y}{n^{a}})g(y)dsdyd\xi.
\end{align*}
By applying the Fourier representation \eqref{eq:FourierrepLprime} in \eqref{eq:R3tndecomp1q}, and then combining the resulting identity with \eqref{eq:step2decompwithR1}, we obtain
\begin{align}\label{eq:step2decompwithRcase2qd}
n^a(n^{a(\ell+1)-1}G_{t,\lettrequivabien}^{(n,\ell)}[g]+\mu[g]L_{t}^{(\ell)}(\lettrequivabien))+\mu[\tilde{g}]L_{t}^{(\ell+1)}(\lettrequivabien)
  &=R_{t,0}^{(n)}+R_{t,1}^{(n)}+R_{t,2}^{(n)}+{R}_{t,4}^{(n)},
\end{align}
where $R_{t,i}^{(n)}$, for $i=0,1,2$ are given as in \eqref{eq:R0deff1}-\eqref{eq:R2deff1}.\\

\noindent By proceeding as in the proof of Theorem \ref{thm:main2} we deduce that, in order to prove Theorem \ref{thm:main2p}, we are left to show that if $\beta=\beta(H,\ell)$ is defined as
\begin{align*}
\beta:=
  &\left\{\begin{array}{lll}\kappa&\text{ if }&H\in[\frac{1}{2(\ell+2)},\frac{1}{2\ell+3})\\\kappa\wedge(\frac{1-2H(\ell+2)}{4H})&\text{ if }& H\in(0,\frac{1}{2(\ell+2)}),\end{array}\right.
\end{align*}
then, for all $T>0$, there exists a constant $C_T>0$ only depending on $T$, $H$ and $\ell$, such that the following claims hold true:

\begin{enumerate}
\item[-] For every $\delta>0$, $t_1,t_2>\delta$ and $i=1,2,4$,
\begin{align}\label{eq:remindsfortight1222d}
n^{2a}\|R_{t_2,i}^{(n)}-R_{t_1,i}^{(n)}\|_{L^{2}(\Omega)}^2
  & \leq  C_{T}(\|w^{\kappa+1}g\|_{L^{1}(\R)}^2+\|g\|_{W^{2,1}}^2+\|g^{(\ell)}\|_{L^{\infty}(\R)}^2)\nonumber\\
	&\times \left\{\begin{array}{ccc}
\delta^{-H}\rho_n(t_1,t_2)^{2-H(2\ell+2\beta+3)}  &\text{ if }  & t_1,t_2\in[\delta,T]\\
\rho_n(t_1,t_2)^{2-2H(\ell+\beta+2)}             &\text{ if }  & t_1,t_2\in[0,T].
\end{array}\right.
\end{align}

\item[-] For every $t_1,t_2>0$ and $p\in\N$, 
\begin{align}\label{eq:remindsfortight12tre22d}
n^{2a}\|R_{t_2,0}^{(n)}-R_{t_1,0}^{(n)}\|_{L^{2p}(\Omega)}^{2p}
  & \leq  C_{T}^p\| g\|_{L^{1}(\R)}^{2p}\rho_n(t_1,t_2)^{2p(1-H(2+\ell))}.
\end{align}

\item[-] For every $t>0$ and $i=0,1,2,4$,
\begin{align}\label{eq:remindsfortight12asdd}
n^{2a}\|R_{t,i}^{(n)} \|_{L^{2}(\Omega)}^2
  &\leq C_t (\|w^{1+\kappa}g\|_{L^{1}(\R)}+\|g\|_{W^{2,1}}+\|g^{(\ell)}\|_{L^{2}(\R)}^2)n^{-a\kappa},
\end{align}
for some constant $C_{t}>0$  depending on $t$, but not on $n$ and $g$.

\end{enumerate}
As in the proof of Theorem \ref{thm:main2}, to verify this simplification it suffices to prove \eqref{eq:thmmain223p} and show that there exist constants $C,p>0$ such that for all $t_1,t_{2}\in[T_1,T_2]$ and $i=0,1,2,4$,
\begin{align}\label{eq:tightnessgoal1dsa}
\E\left[\Abs{n^{a}(R_{t_2,i}^{(n)}-R_{t_1,i}^{(n)})}^{p}\right]\leq C\rho_n(t_1,t_2)^{\alpha(p)},
\end{align} 
where $\alpha(p)>1$ is some constant only depending on $p$ and $H$.\\

\noindent Relation \eqref{eq:thmmain223p} follows from \eqref{eq:step2decompwithRcase2qd} and \eqref{eq:remindsfortight12asdd}. Moreover, if $H(2\ell+3)<1$ and $T_{1}>0$, then by an application of \eqref{eq:remindsfortight1222d} with $\delta=T_{1}$ and $T=T_{2}$, we obtain
\begin{align}\label{ineq:Billingsleyc1d}
n^{2a}\|R_{t_2,i}^{(n)}-R_{t_1,i}^{(n)}\|_{L^{2}(\Omega)}^2
  &\leq C_{T_1}\rho_n(t_1,t_2)^{2-H(2\ell+2\beta+3)},
\end{align}
for $i=1,2,4$ and $t_1,t_2\in[T_1,T_2]$, where $C_{T_1}$ is a constant independent of $t_1,t_2$ and $n$. Then, using the fact that 
$\beta\leq\kappa<\frac{1-H(2\ell+3)}{2H}$, we deduce that  $2-H(2\ell+2\beta+3)>1$, which by \eqref{ineq:Billingsleyc1d}  implies that the Billingsley condition  holds for $R_{t,1}^{(n)},R_{t,2}^{(n)}, R_{t,4}^{(n)}$.  Similarly, if $2H(\ell+2)<1$, then by \eqref{eq:remindsfortight1222d},
\begin{align}\label{ineq:Billingsleyc2d}
\|R_{t_2,i}^{(n)}-R_{t_1,i}^{(n)}\|_{L^{2}(\Omega)}^2
  &\leq C\rho_n(t_1,t_2)^{2-2H(\ell+\beta+2)},
\end{align}
for $i=1,2,4$, where $C$ is a constant independent of $t_1,t_2$ and $n$. Then, using the fact that $\beta\leq\frac{1-2H(\ell+2)}{4H}< \frac{1-2H(\ell+2)}{2H}$ we deduce that $2-2H(\ell+\beta+2)>1$, which by \eqref{ineq:Billingsleyc2d} implies that the Billingsley criterion holds for $R_{t,1}^{(n)},R_{t,2}^{(n)}, R_{t,4}^{(n)}$. Finally, by applying \eqref{eq:remindsfortight12tre22d} with $p>\frac{1}{2(1-H(2+\ell))}$, we obtain the Billingsley condition for the process $\{R_{0,t}^{(n)}\ \processsymb\ t\geq0\}$ in the case $H<\frac{1}{2\ell+3}$, regardless of the value of $T_{1}$. This finishes the proof of the simplification.\\

\noindent It thus remain to prove \eqref{eq:remindsfortight1222d}-\eqref{eq:remindsfortight12asdd}. In the sequel, we will assume that $t_{1},t_{2}>0$ belong to a given interval of the form $[0,T]$, with $T>0$ and $C$ will denote a generic constant only depending on $T$, $H$ and $\ell$ that might change from line to line. For proving \eqref{eq:remindsfortight12asdd} in the case $i=0$, we use the inequality \eqref{eq:momentsofincLt} in Lemma \ref{Lem:existLocalderiv}, as well as the fact that $|\mu[g]|\leq \| g\|_{L^{1}(\R)}$, to deduce that there exists a constant $C>0$, such that for all $t>0$,
\begin{align}\label{inbeq:R0R4l2boundprob1d}
\|R_{t,0}^{(n)}\|_{L^{2}(\Omega)}^2
  &\leq\| g\|_{L^{1}(\R)}(\|L_{\frac{1}{n}}^{(\ell)}(\lettrequivabien)\|_{L^{2}(\Omega)}+\|L_{t}^{(\ell)}(\lettrequivabien)-L_{\frac{\lfloor nt\rfloor}{n}}^{(\ell)}(\lettrequivabien)\|_{L^{2}(\Omega)})\nonumber\\
	&\leq C\| g\|_{L^{1}(\R)}^2 n^{2H(\ell+1)-2}.
\end{align}
Recall that $H(2\ell+2\kappa+3)<1$, so that the power of $n$ in the right hand side of \eqref{inbeq:R0R4l2boundprob1d} satisfies the inequality 
$$H(2\ell+2)-2=1-2H-2\kappa+H(2\ell+2\kappa+3)<-2\kappa-2H\leq-2\kappa-2a.$$
Relation \eqref{eq:remindsfortight12asdd} then follows from \eqref{inbeq:R0R4l2boundprob1d}. To prove \eqref{eq:remindsfortight12tre22d} in the case $i=0$, we split our proof into the cases $\rho_{n}(t_1,t_2)\leq\frac{1}{n}$ and $\rho_{n}(t_1,t_2)>\frac{1}{n}$. If $\rho_{n}(t_1,t_2)\leq\frac{1}{n}$, we  combine \eqref{eq:momentsofincLt} in Lemma \ref{Lem:existLocalderiv} with the inequality $|\mu[g]|\leq \| g\|_{L^{1}(\R)}$,  to write 
\begin{align}\label{inbeq:R0R4l2boundprob1asdasd}
\|R_{t_2,0}^{(n)}-R_{t_1,0}^{(n)}\|_{L^{2p}(\Omega)}
  &\leq\| g\|_{L^{1}(\R)}(\|L_{t_2}^{(\ell)}(\lettrequivabien)-L_{t_1}^{(\ell)}(\lettrequivabien)\|_{L^{2p}(\Omega)}+\|L_{\frac{\lfloor nt_2\rfloor}{n}}^{(\ell)}(\lettrequivabien)-L_{\frac{\lfloor nt_1\rfloor}{n}}^{(\ell)}(\lettrequivabien)\|_{L^{2p}(\Omega)})\nonumber\\
	&\leq C\| g\|_{L^{1}(\R)}  \rho_n(t_1,t_2)^{1-H(1+\ell)}
	\leq C\rho_n(t_1,t_2)^{a}\| g\|_{L^{1}(\R)} \rho_n(t_1,t_2)^{1-H(2+\ell)}\nonumber\\
	&\leq Cn^{-a}\| g\|_{L^{1}(\R)} \rho_n(t_1,t_2)^{1-H(2+\ell)},
\end{align}
as required. On the other hand, if $\rho_{n}(t_1,t_2)>\frac{1}{n}$, by a further application of inequality \eqref{eq:momentsofincLt} in Lemma \ref{Lem:existLocalderiv}, we get
\begin{align}\label{inbeq:R0R4l2boujhjnhddlo}
\|R_{t_2,0}^{(n)}-R_{t_1,0}^{(n)}\|_{L^{2p}(\Omega)}
  &\leq \|L_{t_1}^{(\ell)}(\lettrequivabien)-L_{\frac{\lfloor nt_1\rfloor}{n}}^{(\ell)}(\lettrequivabien)\|_{L^2(\Omega)}+\|L_{t_2}^{(\ell)}(\lettrequivabien)-L_{\frac{\lfloor nt_2\rfloor}{n}}^{(\ell)}(\lettrequivabien)\|_{L^2(\Omega)}\nonumber\\
	&\leq Cn^{-(1-H(\ell+1))}\| g\|_{L^{1}(\R)}
	\leq Cn^{-a}n^{-(1-H(\ell+2))}\| g\|_{L^{1}(\R)}\nonumber\\
	&\leq Cn^{-a}\rho_{n}(t_1,t_2)^{1-H(\ell+2)}\| g\|_{L^{1}(\R)}.
\end{align}
This completes the proof of \eqref{eq:remindsfortight12tre22d} and \eqref{eq:remindsfortight12asdd} in the case $i=0$.\\

\noindent  For proving \eqref{eq:remindsfortight1222d} and \eqref{eq:remindsfortight12asdd} the case $i=1$, we need to rewrite in a suitable way the random variable 
$$e^{\textbf{i}\xi(X_{\frac{i-1}{n}}-\lettrequivabien-\frac{y}{n^H})}(e^{\textbf{i}\xi\Delta_{i,s}X}-e^{-\frac{1}{2}\xi^2\sigma_{s,i}^2}),$$
appearing in the definition of $R_{t,1}^{(n)}$. This can be done as follows: using \eqref{eq:Fasdeltaq}, we can write
\begin{multline*}
e^{\textbf{i}\xi(X_{\frac{i-1}{n}}-\lettrequivabien-\frac{y}{n^a})}(e^{\textbf{i}\xi\Delta_{i,s}X}-e^{-\frac{1}{2}\xi^2\sigma_{s,i}^2})\\
\begin{aligned}
  &=\delta\bigg(\Indi{[\frac{i-1}{n},s]}\int_{0}^{1}(\textbf{i}\xi)e^{\textbf{i}\xi (X_{\frac{i-1}{n}}-\lettrequivabien-\frac{y}{n^a}+z\Delta_{i,s}X)-\frac{1}{2}(1-z^2)\sigma_{i,s}^2\xi^2}dz\bigg)\\
	&+\langle\Indi{[\frac{i-1}{n},s]},\Indi{[0,\frac{i-1}{n}]}\rangle_{\Hg}\int_{0}^{1}(\textbf{i}\xi)^2e^{\textbf{i}\xi (X_{\frac{i-1}{n}}-\lettrequivabien-\frac{y}{n^a}+z\Delta_{i,s}X)-\frac{1}{2}(1-z^2)\sigma_{i,s}^2\xi^2}dz.
\end{aligned}
\end{multline*}
Consequently, by \eqref{eq:R1deff1}, $R_{t,1}^{(n)}=R_{t,1,2}^{(n)}+\delta (R_{t,1,1}^{(n,Sk)})$, where 
\begin{align}
R_{t,1,2}^{(n)}
  &:=\frac{1}{2\pi}\int_{\R^{2}}\sum_{i=2}^{\lfloor nt\rfloor}\int_{\frac{i-1}{n}}^{\frac{i}{n}}\langle\Indi{[\frac{i-1}{n},s]},\Indi{[0,\frac{i-1}{n}]}\rangle_{\Hg}\nonumber\\
	&\times \int_{0}^{1}(\textbf{i}\xi)^{\ell+2}e^{\textbf{i}\xi (X_{\frac{i-1}{n}}-\lettrequivabien-\frac{y}{n^a}+z\Delta_{i,s}X)-\frac{1}{2}(1-z^2)\sigma_{i,s}^2\xi^2}g(y)dzdsdyd\xi\label{eq:R12defq}\\
R_{t,1,1}^{(n,Sk)}
  &:=\frac{1}{2\pi }\int_{\R^{2}}\sum_{i=2}^{\lfloor nt\rfloor}\int_{\frac{i-1}{n}}^{\frac{i}{n}}\Indi{[\frac{i-1}{n},s]}\int_{0}^{1}(\textbf{i}\xi)^{\ell+1}e^{\textbf{i}\xi (X_{\frac{i-1}{n}}-\lettrequivabien-\frac{y}{n^a}+z\Delta_{i,s}X)-\frac{1}{2}(1-z^2)\sigma_{i,s}^2\xi^2}g(y)dzdsdyd\xi\nonumber.
\end{align}
Thus, in order to bound  $\|R_{t_2,1}^{(n)}-R_{t_1,1}^{(n)}\|_{L^{2}(\Omega)} $, it suffices to estimate
$\|R_{t_2,1,2}^{(n)}-R_{t_1,1,2}^{(n)}\|_{L^{2}(\Omega)}$ and $\|\delta( R_{t_2,1,1}^{(n,Sk)}-R_{t_1,1,1}^{(n,Sk)})\|_{L^{2}(\Omega)} $.\\

\noindent First we handle the term $\|\delta( R_{t_2,1,1}^{(n,Sk)}-R_{t_1,1,1}^{(n,Sk)})\|_{L^{2}(\Omega)}$. By the identity $R_{t,1}^{(n)}=R_{t,1,2}^{(n)}+\delta (R_{t,1,1}^{(n,Sk)})$, we have that 
\begin{align*}
\E[\delta( R_{t_2,1,1}^{(n,Sk)}-R_{t_1,1,1}^{(n,Sk)})^2]
  &=\E[\delta( R_{t_2,1,1}^{(n,Sk)}-R_{t_1,1,1}^{(n,Sk)})((R_{t_2,1}^{(n)}-R_{t_1,1}^{(n)})-(R_{t_2,1,2}^{(n)}-R_{t_1,1,2}^{(n)}))]\\
	&=\E[\langle R_{t_2,1,1}^{(n,Sk)}-R_{t_1,1,1}^{(n,Sk)},D(R_{t_2,1}^{(n)}-R_{t_1,1}^{(n)})-D(R_{t_2,1,2}^{(n)}-R_{t_1,1,2}^{(n)})\rangle_{\Hg}],
\end{align*}
which implies that 
\begin{align}\label{ineq:R11skboundq}
\|\delta( R_{t_2,1,1}^{(n,Sk)}-R_{t_1,1,1}^{(n,Sk)})\|_{L^2(\Omega)}^2
  &\leq |\E[\langle  R_{t_2,1,1}^{(n,Sk)}-R_{t_1,1,1}^{(n,Sk)},D(R_{t_2,1}^{(n)}-R_{t_1,1}^{(n)})\rangle_{\Hg}]|\nonumber\\
	&+|\E[\langle  R_{t_2,1,1}^{(n,Sk)}-R_{t_1,1,1}^{(n,Sk)},D(R_{t_2,1,2}^{(n)}-R_{t_1,1,2}^{(n)})\rangle_{\Hg}]|.
\end{align}
In order to estimate $|\E[\langle  R_{t_2,1,1}^{(n,Sk)}-R_{t_1,1,1}^{(n,Sk)},D(R_{t_2,1}^{(n)}-R_{t_1,1}^{(n)})\rangle_{\Hg}]|$ we proceed as follows. Notice that $D(R_{t_2,1}^{(n)}-R_{t_1,1}^{(n)})=P_1^{(n)}+P_2^{(n)}$ by \eqref{eq:R1deff1}, where 
\begin{align*}
P_1^{(n)}
  &:=\frac{1}{2\pi  }\int_{\R^{2}}\sum_{2\vee nt_1\leq i\leq nt_2}\int_{\frac{i-1}{n}}^{\frac{i}{n}}(\textbf{i}\xi)^{\ell+1} \Indi{[0,\frac{i-1}{n}]} e^{\textbf{i}\xi(X_{\frac{i-1}{n}}-\lettrequivabien-\frac{y}{n^a})}(e^{-\frac{1}{2}\xi^2\sigma_{s,i}^2}-e^{\textbf{i}\xi\Delta_{i,s}X})g(y)dzdsdyd\xi,\\
P_2^{(n)}
  &:=-\frac{1}{2\pi  }\int_{\R^{2}}\sum_{2\vee nt_1\leq i\leq nt_2}\int_{\frac{i-1}{n}}^{\frac{i}{n}}(\textbf{i}\xi)^{\ell+1} \Indi{[\frac{i-1}{n},s]} e^{\textbf{i}\xi(X_{s}-\lettrequivabien-\frac{y}{n^a})} g(y)dzdsdyd\xi.
\end{align*}
Therefore, in order to estimate $|\E[\langle R_{t_2,1,1}^{(n,Sk)}-R_{t_1,1,1}^{(n,Sk)},D(R_{t_2,1}^{(n)}-R_{t_1,1}^{(n)})\rangle_{\Hg}]|$, it suffices to analyze the terms
\begin{align}\label{eq:R12bound1sq}
\E[\langle R_{t_2,1,1}^{(n,Sk)}-R_{t_1,1,1}^{(n,Sk)},P_{1}^{(n)}\rangle_{\Hg}]
  &=\frac{(-1)^{\ell+1}}{4\pi^2}\int_{\R^{4}}\sum_{2\vee nt_1\leq i,j\leq nt_2}\int_{\frac{i-1}{n}}^{\frac{i}{n}}\int_{\frac{j-1}{n}}^{\frac{j}{n}}\int_{0}^1(\xi\tilde{\xi})^{\ell+1}
	\langle\Indi{[\frac{i-1}{n},s]},\Indi{[0,\frac{j-1}{n}]}\rangle_{\Hg}\nonumber
	\\
	&\times g(y)g(\tilde{y}) \E[e^{\textbf{i}\xi (X_{\frac{i-1}{n}}-\lettrequivabien-\frac{y}{n^a}+z\Delta_{i,s}X)-\frac{1}{2}(1-z^2)\sigma_{i,s}^2\xi^2
	+\textbf{i}\tilde{\xi} (X_{\frac{j-1}{n}}-\lettrequivabien-\frac{\tilde{y}}{n^a})}\nonumber \\
	&\times (e^{\textbf{i}\tilde{\xi}\Delta_{j,\tilde{s}}X}-e^{-\frac{1}{2}\tilde{\xi}^2\sigma_{\tilde{s},j}^2})] dzd\vec{s}d\vec{y}d\vec{\xi},
\end{align}
and 
\begin{multline}\label{eq:R12bound2q}
\E[\langle R_{t_2,1,1}^{(n,Sk)}-R_{t_1,1,1}^{(n,Sk)},P_{2}^{(n)}\rangle_{\Hg}]\\
\begin{aligned}
  &=\frac{(-1)^{\ell}}{4\pi^2n^{2H}}\int_{\R^{4}}\sum_{2\vee nt_1\leq i,j\leq nt_2}\int_{\frac{i-1}{n}}^{\frac{i}{n}}\int_{\frac{j-1}{n}}^{\frac{j}{n}}\int_{0}^1(\xi\tilde{\xi})^{ \ell+1}
	\langle\Indi{[i-1,ns]},\Indi{[j-1,n\tilde{s}]}\rangle_{\Hg} \\
	&\times g(y)g(\tilde{y}) \E[e^{\textbf{i}\xi (X_{\frac{i-1}{n}}-\lettrequivabien-\frac{y}{n^a}+z\Delta_{i,s}X)-\frac{1}{2}(1-z^2)\sigma_{i,s}^2\xi^2
	+\textbf{i}\tilde{\xi} (X_{\tilde{s}}-\lettrequivabien-\frac{\tilde{y}}{n^a} )}]d {z}d\vec{s}d\vec{y}d\vec{\xi},
\end{aligned}
\end{multline}
where $\vec{s}=(s,\tilde{s})$, $\vec{y}=(y,\tilde{y})$ and $\vec{\xi}=(\xi,\tilde{\xi})$. First we bound the right-hand side of \eqref{eq:R12bound1sq}. To this end, we first write 
\begin{align}\label{eq:RKvsP1decomp}
\E[\langle R_{t_2,1,1}^{(n,Sk)}-R_{t_1,1,1}^{(n,Sk)},P_{1}^{(n)}\rangle_{\Hg}]
  &=A_{1,1,1}^{(n)}+\tilde{A}_{1,1,1}^{(n)},
\end{align}
where 
\begin{align}\label{eq:A111defq}
A_{1,1,1}^{(n)}
  &:=\frac{(-1)^{\ell+1}}{4\pi^2}\int_{\R^{4}}\sum_{\substack{2\vee (nt_1)\leq i,j\leq nt_2\\|i-j|\geq 2}}\int_{\frac{i-1}{n}}^{\frac{i}{n}}\int_{\frac{j-1}{n}}^{\frac{j}{n}}\int_{0}^1(\xi\tilde{\xi})^{\ell+1}
	\langle\Indi{[\frac{i-1}{n},s]},\Indi{[0,\frac{j-1}{n}]}\rangle_{\Hg}\nonumber
	\\
	&\times g(y)g(\tilde{y}) \E[e^{\textbf{i}\xi (X_{\frac{i-1}{n}}-\lettrequivabien-\frac{y}{n^a}+z\Delta_{i,s}X)-\frac{1}{2}(1-z^2)\sigma_{i,s}^2\xi^2
	+\textbf{i}\tilde{\xi} (X_{\frac{j-1}{n}}-\lettrequivabien-\frac{\tilde{y}}{n^a})}\nonumber \\
	&\times (e^{\textbf{i}\tilde{\xi}\Delta_{j,\tilde{s}}X}-e^{-\frac{1}{2}\tilde{\xi}^2\sigma_{\tilde{s},j}^2})] dzd\vec{s}d\vec{y}d\vec{\xi}
\end{align}
and 
\begin{align}\label{eq:A111tildeeqq}
\tilde{A}_{1,1,1}^{(n)}
  &:=\frac{(-1)^{\ell+1}}{4\pi^2}\int_{\R^{4}}\sum_{\substack{2\vee (nt_1)\leq i,j\leq nt_2\\|i-j|\leq 1}}\int_{\frac{i-1}{n}}^{\frac{i}{n}}\int_{\frac{j-1}{n}}^{\frac{j}{n}}\int_{0}^1(\xi\tilde{\xi})^{\ell+1}
	\langle\Indi{[\frac{i-1}{n},s]},\Indi{[0,\frac{j-1}{n}]}\rangle_{\Hg}\nonumber
	\\
	&\times g(y)g(\tilde{y}) \E[e^{\textbf{i}\xi (X_{\frac{i-1}{n}}-\lettrequivabien-\frac{y}{n^a}+z\Delta_{i,s}X)-\frac{1}{2}(1-z^2)\sigma_{i,s}^2\xi^2
	+\textbf{i}\tilde{\xi} (X_{\frac{j-1}{n}}-\lettrequivabien-\frac{\tilde{y}}{n^a})}\nonumber \\
	&\times (e^{\textbf{i}\tilde{\xi}\Delta_{j,\tilde{s}}X}-e^{-\frac{1}{2}\tilde{\xi}^2\sigma_{\tilde{s},j}^2})] dzd\vec{s}d\vec{y}d\vec{\xi}.
\end{align}
 
To bound the right hand side of \eqref{eq:A111defq}, we notice that 
\begin{multline}\label{eq:P1termpaux1aux1q}
|\E[e^{\textbf{i}\xi (X_{\frac{i-1}{n}}-\lettrequivabien-\frac{y}{n^a}+z\Delta_{i,s}X)-\frac{1}{2}(1-z^2)\sigma_{i,s}^2\xi^2
	+\textbf{i}\tilde{\xi} (X_{\frac{j-1}{n}}-\lettrequivabien-\frac{\tilde{y}}{n^a})} (e^{\textbf{i}\tilde{\xi}\Delta_{j,\tilde{s}}X}-e^{-\frac{1}{2}\xi^2\sigma_{\tilde{s},j}^2})]|\\
	\leq |e^{-\frac{1}{2}\vec{\xi}^*\Sigma(z,1)\vec{\xi}}-e^{-\frac{1}{2}\vec{\xi}^*{\Sigma}(z,0)\vec{\xi}}|,
\end{multline}
where $\Sigma(z,\tilde{z})$ is the covariance matrix given by \eqref{eq:SigmaZdef}. Using \eqref{eq:diffFouriertransfGaussian} and the fact that $H<\frac{1}{2\ell+3}<\frac{1}{2}$, we conclude that 
\begin{align*}
|e^{-\frac{1}{2}\vec{\xi}^*\Sigma(z,1)\vec{\xi}}-e^{-\frac{1}{2}\vec{\xi}^*{\Sigma}(z,0)\vec{\xi}}|
	&\leq  Cn^{-2H\kappa}(|\xi|^{\kappa}+|\tilde{\xi}|^{\kappa})|\tilde{\xi}|^{\kappa}
	(e^{-\frac{1}{2}\vec{\xi}^*\Sigma(z,1)\vec{\xi}}+e^{-\frac{1}{2}\vec{\xi}^*{\Sigma}(z,0)\vec{\xi}}).
\end{align*}
Consequently, by \eqref{eq:A111defq}, \eqref{eq:P1termpaux1aux1q} and \eqref{eq:estimationip2p},
\begin{align}\label{eq:R12bound1auxq}
A_{1,1,1}^{(n)}
  &\leq  Cn^{-2H(1+\kappa)} \int_{\R^{4}}\sum_{\substack{2\vee(nt_1)\leq i,j\leq nt_2\\|i-j|\geq 2}}\int_{\frac{i-1}{n}}^{\frac{i}{n}}\int_{\frac{j-1}{n}}^{\frac{j}{n}}\int_{0}^1|g(y)g(\tilde{y})|\nonumber\\
	&\times (|\xi|^{\kappa}+|\tilde{\xi}|^{\kappa})|\tilde{\xi}|^{\ell+\kappa+1}|\xi|^{\ell+1} (e^{-\frac{1}{2}\vec{\xi}^{*}\Sigma(z,1)\vec{\xi}}+e^{-\frac{1}{2}\vec{\xi}^{*}{\Sigma}(z,0)\vec{\xi}}) dzd\vec{s}d\vec{y}d\vec{\xi},
\end{align}
which by equation \eqref{eq:Lemmaauxgoal} in Lemma \ref{lem:techintegralbound}, implies that 
\begin{align}\label{eq:covR1P1finalq}
|A_{1,1,1}^{(n)}|
& \leq  C\|g\|_{L^1(\R)}^2n^{-2H(1+\kappa)} \frac{1}{n^2}\sum_{\substack{2\vee(nt_1)\leq i,j\leq nt_2\\|i-j|\geq 2}}\big(\frac{i\vee j}{n}\big)^{-H}\big(\frac{i\wedge j\wedge|j-i|}{n}\big)^{-H(2\ell+2\kappa+3)}\\
& \leq  C\|g\|_{L^1(\R)}^2n^{-2H(1+\kappa)} \int_{[\frac{\lfloor nt_1\rfloor}{n},\frac{\lfloor nt_2\rfloor}{n}]^2}(s\wedge{s})^{-H}(s\wedge{s}\wedge|s-\tilde{s}|)^{-H(2\ell+2\kappa+3)}dsd\tilde{s}\nonumber.
\end{align} 
Therefore, using the inequality \eqref{eq:integralbound}, as well as the condition $a\leq H$, we conclude that 
\begin{align}\label{eq:A111finalbound}
n^{2a}|A_{1,1,1}^{(n)}|
& \leq  C\|g\|_{L^1(\R)}^2\times 
\left\{\begin{array}{ccc}
\delta^{-H}\rho_n(t_1,t_2)^{2-H(2\ell+2\kappa+3)}  &\text{ if }  & t_1,t_2\in[\delta,T]\\
\rho_n(t_1,t_2)^{2-2H(\ell+\kappa+2)}             &\text{ if }  & t_1,t_2\in[0,T].
\end{array}\right.
\end{align} 
Similarily, from \eqref{eq:A111tildeeqq} and \eqref{eq:P1termpaux1aux1q}, it follows that 
\begin{align*}
|\tilde{A}_{1,1,1}^{(n)}|
  &\leq  n^{-2H} \int_{\R^{4}}\sum_{\substack{2\vee(nt_1)\leq i,j\leq nt_2\\|i-j|\geq 2}}\int_{\frac{i-1}{n}}^{\frac{i}{n}}\int_{\frac{j-1}{n}}^{\frac{j}{n}}\int_{0}^1|g(y)g(\tilde{y})|\nonumber\\
	&\times |\xi\tilde{\xi}|^{\ell+1} (e^{-\frac{1}{2}\vec{\xi}^{*}\Sigma(z,1)\vec{\xi}}+e^{-\frac{1}{2}\vec{\xi}^{*}{\Sigma}(z,0)\vec{\xi}}) dzd\vec{s}d\vec{y}d\vec{\xi},
\end{align*}
which by equation \eqref{eq:Lemmaauxgoal2} in Lemma \ref{lem:techintegralbound}, implies that

\begin{align}\label{eq:R12bound1auxqdsa}
|\tilde{A}_{1,1,1}^{(n)}|
  &\leq    C\|g\|_{L^1(\R)}^2n^{-2H -2}  \sum_{\substack{2\vee(nt_1)\leq i,j\leq nt_2\\|i-j|\leq 1}}\big(\frac{i\vee j}{n}\big)^{-H}n^{H(2\ell+3)}.
\end{align}
Therefore, by applying \eqref{eq:bounimaxjint} with $\kappa$ replaced by $\kappa+1$, we get

\begin{align}\label{eq:R12bound1auxqdsap2}
|\tilde{A}_{1,1,1}^{(n)}|
  &\leq C\|g\|_{L^1(\R)}^2n^{-2H(1+\kappa)} \frac{1}{n^2}\sum_{2\vee(nt_1)\leq i,j\leq nt_2}\big(\frac{i\vee j}{n}\big)^{-H}\big(\frac{i\wedge j\wedge(1+|j-i|)}{n}\big)^{-H(2\ell+2\kappa+3)}\\
& \leq  C\|g\|_{L^1(\R)}^2n^{-2H(1+\kappa)} \int_{[\frac{\lfloor nt_1\rfloor}{n},\frac{\lfloor nt_2\rfloor}{n}]^2}(s\wedge{s})^{-H}(s\wedge{s}\wedge|s-\tilde{s}|)^{-H(2\ell+2\kappa+3)}dsd\tilde{s}.\nonumber
\end{align} 
An application of the inequality \eqref{eq:integralbound} and the condition $a\leq H$ then leads to
\begin{align}\label{eq:Atilde111finalbound}
n^{2a}|\tilde{A}_{1,1,1}^{(n)}|
& \leq  Cn^{-2H\kappa}\|g\|_{L^1(\R)}^2\times 
\left\{\begin{array}{ccc}
\delta^{-H}\rho_n(t_1,t_2)^{2-H(2\ell+2\kappa+3)}  &\text{ if }  & t_1,t_2\in[\delta,T]\\
\rho_n(t_1,t_2)^{2-2H(\ell+\kappa+2)}             &\text{ if }  & t_1,t_2\in[0,T].
\end{array}\right.
\end{align} 
Finally, by \eqref{eq:RKvsP1decomp}, \eqref{eq:A111finalbound} and \eqref{eq:Atilde111finalbound} we obtain the bound 
\begin{multline}\label{eq:R11P1tightineqfinal}
|\E[\langle R_{t_2,1,1}^{(n,Sk)}-R_{t_1,1,1}^{(n,Sk)},P_{1}^{(n)}\rangle_{\Hg}]|\\
  \leq Cn^{-2H\kappa}\|g\|_{L^1(\R)}^2\times \left\{\begin{array}{ccc}
\delta^{-H}\rho_n(t_1,t_2)^{2-H(2\ell+2\kappa+3)}  &\text{ if }  & t_1,t_2\in[\delta,T]\\
\rho_n(t_1,t_2)^{2-2H(\ell+\kappa+2)}             &\text{ if }  & t_1,t_2\in[0,T].
\end{array}\right.
\end{multline}

\noindent Next we bound the term $\E[\langle R_{t_2,1,1}^{(n,Sk)}-R_{t_1,1,1}^{(n,Sk)},P_{2}^{(n)}\rangle_{\Hg}]$. By \eqref{eq:estimationip}, for all $s\in[\frac{i-1}{n},\frac{i}{n}]$ and $\tilde{s}\in[\frac{j-1}{n},\frac{j}{n}]$, it holds that 
\begin{align*}
|\langle\Indi{[i-1,ns]},\Indi{[j-1,n\tilde{s}]}\rangle_{\Hg}|
  &\leq |j-i-1|^{2H-2}\Indi{\{|i-j|\geq 2\}}+\Indi{\{|i-j|\leq 1\}}\leq C(1+|i-j|)^{2H-2}.
\end{align*}
Thus, by \eqref{eq:R12bound2q},
\begin{multline}\label{eq:R12bound2aux1qwh}
|\E[\langle R_{t_2,1,1}^{(n,Sk)}-R_{t_1,1,1}^{(n,Sk)},P_{2}^{(n)}\rangle_{\Hg}]|\\
\begin{aligned}
  &\leq\frac{1}{4\pi^2n^{2H}}\int_{\R^{4}}\sum_{2\vee (nt_1)\leq i,j\leq nt_2}\int_{\frac{i-1}{n}}^{\frac{i}{n}}\int_{\frac{j-1}{n}}^{\frac{j}{n}}\int_{0}^1|\xi\tilde{\xi}|^{ \ell+1} \\
	&\times |\langle\Indi{[i-1,ns]},\Indi{[j-1,n\tilde{s}]}\rangle_{\Hg}||g(y)g(\tilde{y})|e^{-\frac{1}{2}\vec{\xi}^*{\Sigma}(z,1)\vec{\xi}}dzd\vec{s}d\vec{y}d\vec{\xi}\\
  &\leq Cn^{-2H}\|g\|_{L^1(\R)}^2 \int_{\R^{2}}\sum_{2\vee(nt_1)\leq i,j\leq nt_2}\int_{\frac{i-1}{n}}^{\frac{i}{n}}\int_{\frac{j-1}{n}}^{\frac{j}{n}}\int_{0}^1(1+|i-j|)^{2H-2}|\xi\tilde{\xi}|^{ \ell+1}e^{-\frac{1}{2}\vec{\xi}^*{\Sigma}(z,1)\vec{\xi}}dzd\vec{s}d\vec{\xi}.
\end{aligned}
\end{multline}
By following the same arguments as in the proof of \eqref{eq:covR1P1finalq} and \eqref{eq:R12bound1auxqdsap2}, we can apply inequality \eqref{eq:Lemmaauxgoal} in Lemma \ref{lem:techintegralbound} to the indices in the right hand side of \eqref{eq:R12bound2aux1qwh} satisfying $|i-j|\geq 2$, and inequality \eqref{eq:Lemmaauxgoal2} to the indices satisfying satisfying $|i-j|\leq 1$, in order to obtain
\begin{multline*} 
|\E[\langle R_{t_2,1,1}^{(n,Sk)}-R_{t_1,1,1}^{(n,Sk)},P_{2}^{(n)}\rangle_{\Hg}]|\\
\begin{aligned}
  &\leq C\|g\|_{L^1(\R)}^2 n^{2H(\ell+1)-2} \sum_{2\vee(nt_1)\leq i,j\leq nt_2}(1+|j-i|)^{2H-2}(i\vee j)^{-H}(i\wedge j\wedge(1+|j-i|))^{-H(2\ell+3)}.
\end{aligned}
\end{multline*}
Using the previous inequality and the conditions $H<\frac{1}{2\ell+3}<\frac{1}{2}$, $H(2\ell+2\kappa+3)<1$, we obtain 
\begin{align}\label{eq:R111vsP2finalboundfdd} 
|\E[\langle R_{t_2,1,1}^{(n,Sk)}-R_{t_1,1,1}^{(n,Sk)},P_{2}^{(n)}\rangle_{\Hg}]|
  &\leq C\|g\|_{L^1(\R)}^2 n^{2H(\ell+1)-2} \sum_{2\vee(nt_1)\leq i,j\leq nt_2}(1+|j-i|)^{2H-2}(i\vee j)^{-H}\nonumber\\
	&\leq (1-H)^{-1}Ct_{2}^{1-H} \|g\|_{L^1(\R)}^2 n^{H(2\ell+1)-1} \sum_{r\in\Z}(1+|r|)^{2H-2}\nonumber\\
	&\leq C^{\prime}t_{2}^{1-H}\|g\|_{L^1(\R)}^2 n^{-2H\kappa}
\end{align}
and 
\begin{multline}\label{eq:R111vsP2finalboundprev} 
|\E[\langle R_{t_2,1,1}^{(n,Sk)}-R_{t_1,1,1}^{(n,Sk)},P_{2}^{(n)}\rangle_{\Hg}]|\\
\begin{aligned}
  &\leq C\|g\|_{L^1(\R)}^2 n^{2H(\ell+1)-2} \sum_{2\vee(nt_1)\leq i,j\leq nt_2}(1+|j-i|)^{2H-2}(i\vee j)^{-H}(i\wedge j\wedge(1+|j-i|))^{-H(2\ell+3)}\\
	&\leq C\|g\|_{L^1(\R)}^2 n^{-2H-2} \sum_{2\vee(nt_1)\leq i,j\leq nt_2}\big(\frac{i\vee j}{n})^{-H}\big(\frac{i\wedge j\wedge(1+|j-i|)}{n}\big)^{-H(2\ell+3)}\\
	& \leq  C\|g\|_{L^1(\R)}^2n^{-2H } \int_{[\frac{\lfloor nt_1\rfloor}{n},\frac{\lfloor nt_2\rfloor}{n}]^2}(s\wedge{s})^{-H}(s\wedge{s}\wedge|s-\tilde{s}|)^{-H(2\ell+2\kappa+3)}ds d\tilde{s}.
\end{aligned}
\end{multline}
Notice that due to  \eqref{eq:integralbound}, \eqref{eq:R111vsP2finalboundprev} and  condition $a\leq H$,  the previous inequality implies that
\begin{multline}\label{eq:R111vsP2finalbound}
n^{2a}|\E[\langle R_{t_2,1,1}^{(n,Sk)}-R_{t_1,1,1}^{(n,Sk)},P_{2}^{(n)}\rangle_{\Hg}]|\\
\begin{aligned}
& \leq  C\|g\|_{L^1(\R)}^2\times 
\left\{\begin{array}{ccc}
\delta^{-H}\rho_n(t_1,t_2)^{2-H(2\ell+2\kappa+3)}  &\text{ if }  & t_1,t_2\in[\delta,T]\\
\rho_n(t_1,t_2)^{2-2H(\ell+\kappa+2)}             &\text{ if }  & t_1,t_2\in[0,T].
\end{array}\right.
\end{aligned}
\end{multline} 
From relations \eqref{eq:R11P1tightineqfinal}, \eqref{eq:R111vsP2finalboundfdd} and the fact that $DR_{1}^{(n)}=P_1^{(n)}+P_2^{(n)}$, it follows that
\begin{align}\label{eq:R11DR1finalboundq}
n^{2a}|\E[\langle  R_{t_2,1,1}^{(n,Sk)}-R_{t_1,1,1}^{(n,Sk)},D(R_{t_2,1}^{(n)}-R_{t_1,1}^{(n)})\rangle_{\Hg}]|
	&\leq Ct_2^{1-H}n^{-2H\kappa}, 
\end{align}
for some constant $C>0$ independent of $t_1,t_2$ and $n$. In addition, by \eqref{eq:R11P1tightineqfinal} and \eqref{eq:R111vsP2finalbound} we get 
\begin{multline}\label{eq:R11vsDRfinalbound}
n^{2a}|\E[\langle  R_{t_2,1,1}^{(n,Sk)}-R_{t_1,1,1}^{(n,Sk)},D(R_{t_2,1}^{(n)}-R_{t_1,1}^{(n)})\rangle_{\Hg}]|\\
\begin{aligned}
& \leq  C\|g\|_{L^1(\R)}^2\times 
\left\{\begin{array}{ccc}
\delta^{-H}\rho_n(t_1,t_2)^{2-H(2\ell+2\kappa+3)}  &\text{ if }  & t_1,t_2\in[\delta,T]\\
\rho_n(t_1,t_2)^{2-2H(\ell+\kappa+2)}             &\text{ if }  & t_1,t_2\in[0,T].
\end{array}\right.
\end{aligned}
\end{multline} 

\noindent Next we proceed with the problem of bounding $|\E[\langle  R_{t_2,1,1}^{(n,Sk)}-R_{t_1,1,1}^{(n,Sk)},D(R_{t_2,1,2}^{(n)}-R_{t_1,1,2}^{(n)})\rangle_{\Hg}]|$. To this end, we use \eqref{eq:R12defq} to write
\begin{align*}
DR_{t,1,2}^{(n)}
  &=\frac{1}{2\pi}\int_{\R^{2}}\sum_{i=2}^{\lfloor nt\rfloor}\int_{\frac{i-1}{n}}^{\frac{i}{n}}\int_{0}^{1}(\Indi{[0,\frac{i-1}{n}]}+z\Indi{[\frac{i-1}{n},s]}) \langle\Indi{[\frac{i-1}{n},s]},\Indi{[0,\frac{i-1}{n}]}\rangle_{\Hg}\\
	&\times (\textbf{i}\xi)^{\ell+3}e^{\textbf{i}\xi (X_{\frac{i-1}{n}}-\frac{y}{n^a}+z\Delta_{i,s}X)-\frac{1}{2}(1-z^2)\sigma_{i,s}^2\xi^2}g(y)dzdsdyd\xi,
\end{align*}
which leads to 
\begin{align}\label{eq:R11SkDR12decomp}
\E[\langle  R_{t_2,1,1}^{(n,Sk)}-R_{t_1,1,1}^{(n,Sk)},D(R_{t_2,1,2}^{(n)}-R_{t_1,1,2}^{(n)})\rangle_{\Hg}]
  &=A_{1,1,2}+\tilde{A}_{1,1,2},
\end{align}
where 
\begin{align*}
A_{1,1,2}
  &=\frac{(-1)^{\ell+2}}{4\pi^2 }\int_{\R^{4}}\sum_{\substack{2\vee(nt_1)\leq i,j\leq nt_2\\|i-j|\geq 2}}\int_{\frac{i-1}{n}}^{\frac{i}{n}}\int_{\frac{j-1}{n}}^{\frac{j}{n}}\int_{[0,1]^2}\xi^{ \ell+1}\tilde{\xi}^{\ell+3}\\
	&\times\langle  \Indi{[\frac{i-1}{n},s]},\Indi{[0,\frac{j-1}{n}]}\rangle_{\Hg} \langle \Indi{[\frac{i-1}{n},s]},\Indi{[0,\frac{j-1}{n}]}+\tilde{z}\Indi{[\frac{j-1}{n},\tilde{s}]}\rangle_{\Hg}g(y)g(\tilde{y}) \\
	&\times \E[e^{\textbf{i}\xi (X_{\frac{i-1}{n}}-\lettrequivabien-\frac{y}{n^a}+z\Delta_{i,s}X)-\frac{1}{2}(1-z^2)\sigma_{i,s}^2\xi^2+
	\textbf{i}\tilde{\xi}(X_{\frac{j-1}{n}}-\lettrequivabien-\frac{\tilde{y}}{n^a}+\tilde{z}\Delta_{j,\tilde{s}}X-\frac{1}{2}(1-\tilde{z}^2)\sigma_{j,\tilde{s}}^2\tilde{\xi}^2}]d\vec{z}d\vec{s}d\vec{y}d\vec{\xi}
\end{align*}
and 
\begin{align}\label{eq:Atilde112def}
\tilde{A}_{1,1,2}
  &=\frac{(-1)^{\ell+2}}{4\pi^2 }\int_{\R^{4}}\sum_{\substack{2\vee(nt_1)\leq i,j\leq nt_2\\|i-j|\leq 1}}\int_{\frac{i-1}{n}}^{\frac{i}{n}}\int_{\frac{j-1}{n}}^{\frac{j}{n}}\int_{[0,1]^2}\xi^{ \ell+1}\tilde{\xi}^{\ell+3}\nonumber\\
	&\times\langle  \Indi{[\frac{i-1}{n},s]},\Indi{[0,\frac{j-1}{n}]}\rangle_{\Hg} \langle \Indi{[\frac{i-1}{n},s]},\Indi{[0,\frac{j-1}{n}]}+\tilde{z}\Indi{[\frac{j-1}{n},\tilde{s}]}\rangle_{\Hg}g(y)g(\tilde{y})\nonumber\\
	&\times \E[e^{\textbf{i}\xi (X_{\frac{i-1}{n}}-\lettrequivabien-\frac{y}{n^a}+z\Delta_{i,s}X)-\frac{1}{2}(1-z^2)\sigma_{i,s}^2\xi^2+
	\textbf{i}\tilde{\xi}(X_{\frac{j-1}{n}}-\lettrequivabien-\frac{\tilde{y}}{n^a}+\tilde{z}\Delta_{j,\tilde{s}}X-\frac{1}{2}(1-\tilde{z}^2)\sigma_{j,\tilde{s}}^2\tilde{\xi}^2}]d\vec{z}d\vec{s}d\vec{y}d\vec{\xi}.
\end{align}
Notice that
\begin{align}\label{eq:intofexpztildezcase}
\E[e^{\textbf{i}\xi (X_{\frac{i-1}{n}}+z\Delta_{i,s}X)-\frac{1}{2}(1-z^2)\sigma_{i,s}^2\xi^2+
	\textbf{i}\tilde{\xi}(X_{\frac{j-1}{n}}+\tilde{z}\Delta_{j,\tilde{s}}X)-\frac{1}{2}(1-\tilde{z}^2)\sigma_{j,\tilde{s}}^2\tilde{\xi}^2}]
	  &=e^{-\frac{1}{2}\vec{\xi}^{*}\Sigma(z,\tilde{z})\vec{\xi}}.
\end{align}
Thus, by \eqref{eq:estimationip2p} we deduce that 
\begin{multline*}
|A_{1,1,2}|
  \leq Cn^{-4H}\int_{\R^{4}}\sum_{\substack{2\vee(nt_1)\leq i,j\leq nt_2\\|i-j|\geq2}}\int_{\frac{i-1}{n}}^{\frac{i}{n}}\int_{\frac{j-1}{n}}^{\frac{j}{n}}\int_{[0,1]^2}|\xi|^{ \ell+1}|\tilde{\xi}|^{\ell+3} |g(y)g(\tilde{y})|e^{-\frac{1}{2}\vec{\xi}^{*}\Sigma(z,\tilde{z})\vec{\xi}}d\vec{z}d\vec{s}d\vec{y}d\vec{\xi},
\end{multline*}
which by equation \eqref{eq:Lemmaauxgoal} in Lemma \ref{lem:techintegralbound}, leads to 

\begin{align}\label{eq:A112boundmid}
|A_{1,1,2}|
  &\leq Cn^{-4H}\|g\|_{L^1(\R)}^2 \frac{1}{n^2}\sum_{\substack{2\vee(nt_1)\leq i,j\leq nt_2\\|i-j|\geq2}} \big(\frac{i\vee j}{n})^{-H}\big(\frac{i\wedge j\wedge|j-i|}{n}\big)^{-H(2\ell+5)}\nonumber\\
  &\leq Cn^{-2(H+\kappa)} \|g\|_{L^1(\R)}^2 \frac{1}{n^2}\sum_{\substack{2\vee(nt_1)\leq i,j\leq nt_2\\|i-j|\geq2}} \big(\frac{i\vee j}{n})^{-H}\big(\frac{i\wedge j\wedge|j-i|}{n}\big)^{-H(2\ell+3+2\kappa)}.
\end{align}
On the other hand, by first applying the integration by parts 
\begin{align*}
\int_{\R}g(\tilde{y})e^{\textbf{i}\frac{\tilde{\xi} }{n^{a}}\tilde{y}}d\tilde{y}
  &=\frac{n^{2a}}{\tilde{\xi}^2}\int_{\R}g^{\prime\prime}(\tilde{y})e^{\textbf{i}\frac{\tilde{\xi} }{n^{a}}\tilde{y}}d\tilde{y}
\end{align*}
in \eqref{eq:Atilde112def}, and then using \eqref{eq:intofexpztildezcase} and \eqref{eq:estimationip2p}, we get 
\begin{align*} 
|\tilde{A}_{1,1,2}|
  &\leq Cn^{2a-4H}\int_{\R^{4}}\sum_{\substack{2\vee(nt_1)\leq i,j\leq nt_2\\|i-j|\leq 1}}\int_{\frac{i-1}{n}}^{\frac{i}{n}}\int_{\frac{j-1}{n}}^{\frac{j}{n}}\int_{[0,1]^2}|\xi\tilde{\xi}|^{\ell+1}|g(y)g^{\prime\prime}(\tilde{y})|e^{-\frac{1}{2}\vec{\xi}^{*}\Sigma(z,\tilde{z})\vec{\xi}}d\vec{z}d\vec{s}d\vec{y}d\vec{\xi},
\end{align*}
which by equation \eqref{eq:Lemmaauxgoal2} in Lemma \ref{lem:techintegralbound}  the condition $a\leq H$, yields
\begin{align*} 
|\tilde{A}_{1,1,2}|
  &\leq Cn^{-2H+H(2\ell+3)-2}\|g\|_{W^{2,1}}^2\sum_{\substack{2\vee(nt_1)\leq i,j\leq nt_2\\|i-j|\leq 1}}\big(\frac{i\vee j}{n}\big)^{-H}.
\end{align*}
Therefore, by \eqref{eq:bounimaxjint} with $\kappa$ replaced by $\kappa+1$, we get
\begin{align*} 
|\tilde{A}_{1,1,2}|
  &\leq Cn^{-2H(1+\kappa)-2}\|g\|_{W^{2,1}}^2\sum_{\substack{2\vee(nt_1)\leq i,j\leq nt_2\\|i-j|\leq 1}} \big(\frac{i\vee j}{n})^{-H}\big(\frac{i\wedge j\wedge(1+|j-i|)}{n}\big)^{-H(2\ell+3+2\kappa)} .
\end{align*}
Combining this inequality with \eqref{eq:R11SkDR12decomp} and \eqref{eq:A112boundmid}, we obtain
\begin{multline}\label{eq:R11DR12finalboundq}
|\E[\langle  R_{t_2,1,1}^{(n,Sk)}-R_{t_1,1,1}^{(n,Sk)},D(R_{t_2,1,2}^{(n)}-R_{t_1,1,2}^{(n)})\rangle_{\Hg}]|\\
\begin{aligned}
  &\leq Cn^{-2H(1+\kappa)-2}\|g\|_{W^{2,1}}^2\sum_{ 2\vee(nt_1)\leq i,j\leq nt_2 } \big(\frac{i\vee j}{n})^{-H}\big(\frac{i\wedge j\wedge(1+|j-i|)}{n}\big)^{-H(2\ell+3+2\kappa)}\\
	&\leq Cn^{-2H(1+\kappa)}\|g\|_{W^{2,1}}^2 \int_{[\frac{\lfloor nt_1\rfloor}{n},\frac{\lfloor nt_2\rfloor}{n}]^2}(s\vee\tilde{s})^{-H}(s\wedge\tilde{s}\wedge|s-\tilde{s}|)^{-H(2\ell+3+2\kappa)}d\vec{s},
\end{aligned}
\end{multline}
which by \eqref{eq:integralbound} gives
\begin{multline}\label{ineq:R11,DR12finalbound}
|\E[\langle  R_{t_2,1,1}^{(n,Sk)}-R_{t_1,1,1}^{(n,Sk)},D(R_{t_2,1,2}^{(n)}-R_{t_1,1,2}^{(n)})\rangle_{\Hg}]|\\
\begin{aligned}
\leq Cn^{-2H(1+\kappa)}\|g\|_{W^{2,1}}^2 \times 
\left\{\begin{array}{ccc}
\delta^{-H}\rho_n(t_1,t_2)^{2-H(2\ell+2\kappa+1)}  &\text{ if }  & t_1,t_2\in[\delta,T]\\
\rho_n(t_1,t_2)^{2-H(2\ell+2\kappa+2)}             &\text{ if }  & t_1,t_2\in[0,T].
\end{array}\right.
\end{aligned}
\end{multline}
From \eqref{ineq:R11skboundq}, \eqref{eq:R11DR1finalboundq}  and \eqref{ineq:R11,DR12finalbound}, it follows that for all $t>0$, 
\begin{align}\label{eq:deltaR11finalboundq}
n^{2a}\| \delta(R_{t,1,1}^{(n,Sk)})\|_{L^2(\Omega)}^2
  &\leq Cn^{-2H\kappa}\|g\|_{2,1}^2 t^{1-H}.
\end{align}
for some constant $C>0$ independent of $t$ and $n$. In addition, by \eqref{ineq:R11skboundq},  \eqref{eq:R11vsDRfinalbound} and \eqref{ineq:R11,DR12finalbound},  
\begin{multline}\label{eq:SkortermfinalboundBill}
\|n^{a}\delta( R_{t_2,1,1}^{(n,Sk)}-R_{t_1,1,1}^{(n,Sk)})\|_{L^{2}(\Omega)}^2\\
\begin{aligned}
& \leq  C\|g\|_{W^{2,1}}^2\times 
\left\{\begin{array}{ccc}
\delta^{-H}\rho_n(t_1,t_2)^{2-H(2\ell+2\kappa+3)}  &\text{ if }  & t_1,t_2\in[\delta,T]\\
\rho_n(t_1,t_2)^{2-2H(\ell+\kappa+2)}             &\text{ if }  & t_1,t_2\in[0,T]
\end{array}\right.
\end{aligned}
\end{multline} 
To bound $\|R_{t_2,1,2}^{(n)}-R_{t_1,1,2}^{(n)}\|_{L^{2}(\Omega)}^2$, we notice by \eqref{eq:R12defq} that
\begin{align}\label{eq:decompRt12}
\|R_{t_2,1,2}^{(n)}-R_{t_1,1,2}^{(n)}\|_{L^{2}(\Omega)}^2
  &=A_{1,1,3}+\tilde{A}_{1,1,3}
\end{align}
where 
\begin{align*}
A_{1,1,3}
  &= \frac{1}{4\pi^2}\int_{\R^{4}}\sum_{\substack{2\vee(nt_1)\leq i,j\leq nt_2\\|i-j|\geq 2}}\int_{\frac{i-1}{n}}^{\frac{i}{n}}\int_{\frac{j-1}{n}}^{\frac{j}{n}}\int_{[0,1]^2}(\textbf{i}\xi)^{\ell+2}(\textbf{i}\tilde{\xi})^{\ell+2}\\
	&\times \langle\Indi{[\frac{i-1}{n},s]},\Indi{[0,\frac{i-1}{n}]}\rangle_{\Hg}\langle\Indi{[\frac{j-1}{n},\tilde{s}]},\Indi{[0,\frac{j-1}{n}]}\rangle_{\Hg}g(y)g(\tilde{y})\\
	&\times \E[e^{\textbf{i}\xi (X_{\frac{i-1}{n}}-\lettrequivabien-\frac{y}{n^{H}} +z\Delta_{i,s}X)-\frac{1}{2}(1-z^2)\sigma_{i,s}^2\xi^2
	+\textbf{i}\tilde{\xi} (X_{\frac{j-1}{n}}-\lettrequivabien-\frac{\tilde{y}}{n} +\tilde{z}\Delta_{j,\tilde{s}}X)-\frac{1}{2}(1-\tilde{z}^2)\sigma_{j,\tilde{s}}^2\xi^2
	}]d\vec{z}d\vec{s}d\vec{y}d\vec{\xi}
\end{align*}
and 
\begin{align}\label{eq:Atilde113def}
\tilde{A}_{1,1,3}
  &= \frac{1}{4\pi^2}\int_{\R^{4}}\sum_{\substack{2\vee(nt_1)\leq i,j\leq nt_2\\|i-j|\leq 1}}\int_{\frac{i-1}{n}}^{\frac{i}{n}}\int_{\frac{j-1}{n}}^{\frac{j}{n}}\int_{[0,1]^2}(\textbf{i}\xi)^{\ell+2}(\textbf{i}\tilde{\xi})^{\ell+2}\nonumber\\
	&\times \langle\Indi{[\frac{i-1}{n},s]},\Indi{[0,\frac{i-1}{n}]}\rangle_{\Hg}\langle\Indi{[\frac{j-1}{n},\tilde{s}]},\Indi{[0,\frac{j-1}{n}]}\rangle_{\Hg}g(y)g(\tilde{y})\nonumber\\
	&\times \E[e^{\textbf{i}\xi (X_{\frac{i-1}{n}}-\lettrequivabien-\frac{y}{n^{H}} +z\Delta_{i,s}X)-\frac{1}{2}(1-z^2)\sigma_{i,s}^2\xi^2
	+\textbf{i}\tilde{\xi} (X_{\frac{j-1}{n}}-\lettrequivabien-\frac{\tilde{y}}{n} +\tilde{z}\Delta_{j,\tilde{s}}X)-\frac{1}{2}(1-\tilde{z}^2)\sigma_{j,\tilde{s}}^2\xi^2
	}]d\vec{z}d\vec{s}d\vec{y}d\vec{\xi}.
\end{align}
By \eqref{eq:estimationip2p}, 
\begin{align*}
|A_{1,1,3}|
  &\leq Cn^{-4H}\int_{\R^{4}}\sum_{\substack{2\vee(nt_1)\leq i,j\leq nt_2\\|i-j|\geq 2}}\int_{\frac{i-1}{n}}^{\frac{i}{n}}\int_{\frac{j-1}{n}}^{\frac{j}{n}}\int_{[0,1]^2}|\xi\tilde{\xi}|^{\ell+2}|g(y)g(\tilde{y})|e^{-\frac{1}{2}\vec{\xi}^*\Sigma(z,\tilde{z})\vec{\xi}}d\vec{z}d\vec{s}d\vec{y}d\vec{\xi}.
\end{align*}
Thus, by equation \eqref{eq:Lemmaauxgoal} in Lemma \ref{lem:techintegralbound}, 
\begin{align*}
|A_{1,1,3}|
  &\leq Cn^{-4H }\|g\|_{L^1(\R)}^2 \frac{1}{n^2}\sum_{\substack{2\vee(nt_1)\leq i,j\leq nt_2\\|i-j|\geq 2}}\big(\frac{i\vee j}{n})^{-H}\big(\frac{i\wedge j\wedge|j-i|}{n})^{-H(2\ell+5)}\\
  &\leq C\|g\|_{L^1(\R)}^2 n^{-H(2+2\kappa)} \frac{1}{n^2}\sum_{\substack{2\vee(nt_1)\leq i,j\leq nt_2\\|i-j|\geq 2}}\big(\frac{i\vee j}{n}\big)^{-H}\big(\frac{i\wedge j\wedge |j-i|}{n}\big)^{-H(2\ell+3+2\kappa)}\nonumber\\
	&\leq Cn^{-H(2+2\kappa)}\|g\|_{L^1(\R)}^2 \int_{[\frac{\lfloor nt_1\rfloor}{n},\frac{\lfloor nt_2\rfloor}{n}]^2}(s\vee\tilde{s})^{-H}(s\wedge\tilde{s}\wedge|s-\tilde{s}|)^{-H(2\ell+3+2\kappa)}d\tilde{s},
\end{align*}
and consequently, by relation \eqref{eq:integralbound} and condition $a\leq H$,

\begin{align}\label{ineq:R12finalbound2pqA113}
n^{2a}|A_{1,1,3}|
  &\leq Cn^{-2H\kappa}\|g\|_{L^1(\R)}^2\times 
\left\{\begin{array}{ccc}
\delta^{-H}\rho_n(t_1,t_2)^{2-H(2\ell+2\kappa+1)}  &\text{ if }  & t_1,t_2\in[\delta,T]\\
\rho_n(t_1,t_2)^{2-H(2\ell+2\kappa+2)}             &\text{ if }  & t_1,t_2\in[0,T].
\end{array}\right.
\end{align}
On the other hand, by first using the integration by parts 
\begin{align*}
\int_{\R^2}g(y)g(\tilde{y})e^{-\textbf{i}(\frac{y\xi}{n^a}+\frac{\tilde{y}\tilde{\xi}}{n^a})}dyd\tilde{y}
  &=-\frac{n^{2a}}{\xi\tilde{\xi}}\int_{\R^2}g^{\prime}(y)g^{\prime}(\tilde{y})e^{-\textbf{i}(\frac{y\xi}{n^a}+\frac{\tilde{y}\tilde{\xi}}{n^a})}dyd\tilde{y}
\end{align*}
in \eqref{eq:Atilde113def} and then applying \eqref{eq:estimationip2p} and equation \eqref{eq:Lemmaauxgoal2} in  Lemma \ref{lem:techintegralbound} to the resulting expression, we obtain 
\begin{align*}
|\tilde{A}_{1,1,3}|
  &\leq Cn^{2a-4H -2 }\|g\|_{W^{1,1}}^2  \sum_{\substack{2\vee(nt_1)\leq i,j\leq nt_2\\|i-j|\leq 1}}\big(\frac{i\vee j}{n})^{-H}\big(\frac{1}{n})^{-H(2\ell+3)}.
\end{align*}
Therefore, by applying inequality \eqref{eq:bounimaxjint} with $\kappa$ replaced by $\kappa+1$ and condition $a\leq H$, we get
\begin{align*}
|\tilde{A}_{1,1,3}|
  &\leq Cn^{-2H(1+\kappa) -2 }\|g\|_{W^{1,1}}^2  \sum_{\substack{2\vee(nt_1)\leq i,j\leq nt_2\\|i-j|\leq 1}}\big(\frac{i\vee j}{n})^{-H}\big(\frac{i\wedge j\wedge(1+|i-j|)}{n})^{-H(2\ell+3)}\\
	&\leq Cn^{-2H(1+\kappa)}\|g\|_{W^{1,1}}^2\int_{[\frac{\lfloor nt_1\rfloor}{n},\frac{\lfloor nt_2\rfloor}{n}]^2}(s\vee\tilde{s})^{-H}(s\wedge\tilde{s}\wedge|s-\tilde{s}|)^{-H(2\ell+3+2\kappa)}d\vec{s},
\end{align*}
and consequently, by \eqref{eq:integralbound} and condition $a\leq H$,

\begin{align}\label{ineq:R12finalbound2pqAtil113}
n^{2a}|\tilde{A}_{1,1,3}|
  &\leq Cn^{-2H\kappa}\|g\|_{W^{1,1}}^2\times 
\left\{\begin{array}{ccc}
\delta^{-H}\rho_n(t_1,t_2)^{2-H(2\ell+2\kappa+1)}  &\text{ if }  & t_1,t_2\in[\delta,T]\\
\rho_n(t_1,t_2)^{2-H(2\ell+2\kappa+2)}             &\text{ if }  & t_1,t_2\in[0,T].
\end{array}\right.
\end{align}
Finally, by \eqref{eq:decompRt12}, \eqref{ineq:R12finalbound2pqA113} and \eqref{ineq:R12finalbound2pqAtil113}, we obtain

\begin{align}\label{ineq:R12finalbound2pq}
n^{2a}\|R_{t_2,1,2}^{(n)}-R_{t_1,1,2}^{(n)}\|_{L^{2}(\Omega)}^2
  &\leq Cn^{-2H\kappa}\|g\|_{W^{1,1}}^2\times 
\left\{\begin{array}{ccc}
\delta^{-H}\rho_n(t_1,t_2)^{2-H(2\ell+2\kappa+1)}  &\text{ if }  & t_1,t_2\in[\delta,T]\\
\rho_n(t_1,t_2)^{2-H(2\ell+2(\kappa+2))}             &\text{ if }  & t_1,t_2\in[0,T].
\end{array}\right.
\end{align}
Since $R_{t,1}^{(n)}=R_{t,1,2}^{(n)}+\delta(R_{t,1,1}^{(n,Sk)})$, from \eqref{eq:SkortermfinalboundBill}, \eqref{eq:deltaR11finalboundq} and \eqref{ineq:R12finalbound2pq}, we conclude that

\begin{multline}\label{eq:R1c2finalbound1}
\|n^{a}\delta( R_{t_2,1}^{(n)}-R_{t_1,1}^{(n,Sk)})\|_{L^{2}(\Omega)}^2\\
\begin{aligned}
& \leq  C\|g\|_{W^{2,1}}^2\times 
\left\{\begin{array}{ccc}
\delta^{-H}\rho_n(t_1,t_2)^{2-H(2\ell+2\kappa+3)}  &\text{ if }  & t_1,t_2\in[\delta,T]\\
\rho_n(t_1,t_2)^{2-2H(\ell+\kappa+2)}             &\text{ if }  & t_1,t_2\in[0,T]
\end{array}\right.
\end{aligned}
\end{multline} 
and
\begin{align}\label{eq:R1c2finalbound2}
\|n^{a}\delta(R_{t,1}^{(n)})\|_{L^2(\Omega)}^2
  &\leq Ct^{1-H}n^{-2H\kappa}\|g\|_{W^{2,1}}^2 .
\end{align}
\noindent This finishes the proof of  \eqref{eq:remindsfortight1222d} and \eqref{eq:remindsfortight12asdd} in the case $i=1.$\\

%=========================================================
\noindent To handle the case $i=2$, we reproduce the steps of the proof of \eqref{ineq:case1}, with $\kappa$ replaced by $1+\kappa$, in order to show that 
\begin{align}\label{ineq:case1v2}
n^{2a}\|R_{t_2,2}^{(n)}-R_{t_1,2}^{(n)}\|_{L^{2}(\Omega)}^2
 \leq C\|g\|_{L^1(\R)}^2n^{-2H\kappa}
\times \left\{\begin{array}{ccc}
\delta^{-H}\rho_n(t_1,t_2)^{2-H(2\ell+2\kappa+1)}  &\text{ if }  & t_1,t_2\in[\delta,T]\\
\rho_n(t_1,t_2)^{2-H(2\ell+2\kappa+2)}             &\text{ if }  & t_1,t_2\in[0,T].
\end{array}\right.
\end{align}
Relations \eqref{eq:remindsfortight1222d} and \eqref{eq:remindsfortight12asdd} follow from \eqref{ineq:case1v2}.\\

\noindent To handle the case $i=4$, we reproduce the proof of \eqref{eq:Rt3case1final}, with the following modifications:
\begin{enumerate}
\item[-] the index $i=3$ is replaced by $i=4$;
\item[-] the variable $\kappa$ is now replaced by $1+\kappa$;
\item[-] the terms $(e^{-\textbf{i}\xi\frac{y}{n^{H}}}-1 )$ and $(e^{-\textbf{i}\tilde{\xi}\frac{\tilde{y}}{n^{H}}}-1 )$ are replaced by $(e^{-\textbf{i}\xi\frac{y}{n^{H}}}-1+\textbf{i}\xi\frac{y}{n^{H}})$ and $(e^{-\textbf{i}\tilde{\xi}\frac{\tilde{y}}{n^{H}}}-1+\textbf{i}\tilde{\xi}\frac{\tilde{y}}{n^{H}})$ respectively;
\item[-] the inequality \eqref{eq:complexexpbound} is replaced by 
\begin{align*}
|e^{-\textbf{i}x}-1+\textbf{i}x|^{\kappa}
  \leq 2|x|^{1+\kappa}.
\end{align*}
\end{enumerate}
By doing these modifications, we can easily show that
\begin{align}\label{ineq:case1term4v2}
n^{2a}\|R_{t_2,4}^{(n)}-R_{t_1,4}^{(n)}\|_{L^{2}(\Omega)}^2
 \leq C\|g\|_{L^1(\R)}^2n^{-2H\kappa}
\left\{\begin{array}{ccc}
\delta^{-H}\rho_n(t_1,t_2)^{2-H(2\ell+2\kappa+1)}  &\text{ if }  & t_1,t_2\in[\delta,T]\\
\rho_n(t_1,t_2)^{2-H(2\ell+2\kappa+2)}             &\text{ if }  & t_1,t_2\in[0,T].
\end{array}\right.
\end{align}
Relations \eqref{eq:remindsfortight1222d} and \eqref{eq:remindsfortight12asdd} in the case $i=4$ follow from \eqref{ineq:case1term4v2}. This finishes the proof of Theorem \ref{thm:main2p}.

%=======================================================================================================
\section{Appendix}\label{ref:Appendix}
In this section we investigate the existence and regularity of the derivatives of the local time of a fBm, and present the proofs of the technical lemmas that were used in Sections \ref{sec:firstfluct} and \ref{sec:firstfluc2t2}.

For the rest of the section, $X$ is a fBm with Hurst parameter $H\in (0,1)$.

\begin{Lemma}\label{lem:techintegralboundasd}
Let  $\vec{s}=(s_1,\dots, s_{m}),\vec{u}=(u_1,\dots, u_{m})\in\R_{+}^{m}$, $\vec{z}=(z_1,\dots, z_m)\in\{0,1\}^{m}$ and $q_{1},\dots, q_{m}\in\N$ be such that $u_i\leq s_i$ for all $i=1,\dots,m$ and $s_{1}\leq \cdots \leq s_{m}$. Denote by $\Lambda$ the covariance matrix of $(N_{1},\dots, N_{m})$, with 
\begin{align*}
N_{i}
  :=X_{u_{i}}+z_{i}(X_{s_i}-X_{u_i})+\sqrt{1-z_i^2}(X_{s_i}^{(i)}-X_{u_i}^{(i)}),
\end{align*}
where $X^{(1)},\dots, X^{(m)}$ are independent copies of $X$. Then, if 
$\tau_{i}:=\min\{|s_{i}-s_{j}|\wedge|s_{i}-u_{j}|\ \processsymb\ 1\leq j\leq m, \ i\neq j\}$,  we have that
\begin{align}\label{eq:boundsstildesfixed}
\int_{\R^{m}}e^{-\frac{1}{2}\vec{\xi}^*\Lambda\vec{\xi}}\prod_{i=1}^{p}|\xi_{i}|^{q_{i}}d\vec{\xi}
  &\leq C s_{1}^{-H}\tau_1^{-Hq_1}\bigg(\prod_{j=2}^{p}(s_{j}-s_{j-1})^{-H} \tau_{j}^{-Hq_{i}}\bigg),
\end{align}
for some constant $C>0$ only depending on $H,\vec{q}$ and $m$.
\end{Lemma}
\begin{proof}
By the generalized H\"older inequality, we have that
\begin{align}\label{eq:CauchyShwarz}
\int_{\R^{m}}e^{-\frac{1}{2}\vec{\xi}^*\Lambda\vec{\xi}}\prod_{i=1}^{p}|\xi_{i}|^{q_{i}}d\vec{\xi}
  &\leq \prod_{i=1}^{p}\left(\int_{\R^{m}}e^{-\frac{1}{2}\vec{\xi}^*\Lambda\vec{\xi}}|\xi_{i}|^{pq_{i}}d\vec{\xi}\right)^{\frac{1}{p}}.
\end{align}
The terms of the product in the right hand side can be written in terms of conditional variances in the following way: define $\tilde{\Sigma}:=\Sigma^{-1}$ and denote by $\Phi_{\tilde{\Sigma}}$ the probability density of a centered Gaussian random vector with covariance $\tilde{\Sigma}$. Then, if $r\geq 0$,
\begin{align}
\int_{\R^{m}}e^{-\frac{1}{2}\vec{\xi}^*\Lambda\vec{\xi}}|\xi_i|^rd\vec{\xi}
  &=|\tilde{\Lambda}|^{\frac{1}{2}}\int_{\R^{m}}\Phi_{\tilde{\Sigma}}(\xi)|\xi_i|^rd\vec{\xi}
	=C_{r}|\Lambda|^{-\frac{1}{2}}\tilde{\Lambda}_{i,i}^{\frac{r}{2}}\label{eq:integralexpsig},
\end{align}
for some $C_{r}$ only depending on $r$. To bound $\tilde{\Lambda}_{i,i}$ we define $\Sigma$ as the covariance matrix 
\begin{align*}
\Sigma
  &=\text{Cov}[(N_{1},\dots, N_{i-1},N_{i+1},\dots, N_{m})].
\end{align*}
Then, by using \eqref{eq:detdecomp},  we express the determinant $|\Sigma|$ in the form 
\begin{align*}
|\Lambda|
  &=\bigg(\prod_{\substack{1\leq r\leq m\\r\neq m}}\text{Var}[N_{r}\ |\ N_{1},\dots, N_{r-1}]\bigg)\text{Var}[N_{i}\ |\ N_{j}\  \text{ with } \ j\neq i ]\\
	&=|\Sigma|\text{Var}[N_{i}\ |\ N_{j}\ \text{ with }\ j\neq i ].
\end{align*}
Therefore, if $adj(V)$ denotes the adjugate matrix of $V$, we have that
\begin{align}
\tilde{\Lambda}_{i,i}
  &=(\Lambda^{-1})_{i,i}=\frac{(adj(\Lambda))_{i,i}}{|\Lambda|}=\frac{|\Sigma|}{|\Lambda|}=\text{Var}[N_{i}\  |\ N_{j}\  \text{ with } \ j\neq i]^{-1}.\label{eq:integralexpsig2}
\end{align}
Combining \eqref{eq:integralexpsig} with \eqref{eq:integralexpsig2}, we obtain
\begin{align}\label{eq:integralexpsigpd}
\int_{ \R^{m}}e^{-\frac{1}{2}\vec{\xi}^*\Lambda\vec{\xi}}|\xi_i|^rd\vec{\xi}
  & =C_{r}|\Lambda|^{-\frac{1}{2}}\text{Var}[N_{i}\  |\ N_{j}\  \text{ with } \ j\neq i]^{-\frac{r}{2}}.
\end{align}
To bound   the conditional variance in the right-hand side we proceed as follows.\\

\noindent If $z_{i}=1$, then by the local non determinism property \eqref{ineq:localnd}, there exists a constant $c>0$ only depending on $H$, such that
\begin{align}\label{eq:Varnitaui1}
\text{Var}[N_{i}\  |\ N_{j}\  \text{ with } \ j\neq i]
  &\geq\text{Var}[X_{s_{i}}\  |\ X_{s_{j}},X_{u_j}\  \text{ with } \ j\neq i]
	\geq c\tau_{i}^{2H}.
\end{align}
On the other hand, if $z=0$, then by a further application local non determinism property \eqref{ineq:localnd},
\begin{align}\label{eq:boundcasezis0}
\text{Var}[N_{i}\  |\ N_{j}\  \text{ with } \ j\neq i]
  &\geq\text{Var}[X_{u_{i}}\  |\ X_{s_{j}},X_{u_j}\  \text{ with } \ j\neq i]+  |s_{i}-u_{i}|^{2H}\nonumber\\
	&\geq c\min_{j\neq i}((|u_{i}-s_{j}|\wedge|u_{i}-u_j|)^{2H}+|s_{i}-u_{i}|^{2H}).
\end{align}
Using \eqref{eq:boundcasezis0}, as well as the the fact that for every $x,y\in\R$, 
\begin{align}\label{ineq:Minkowskytype}
(|x|^a+|y|^a)^{1/a}
  &\geq\left\{\begin{array}{ccc}|x|+|y|&\text{ if } & a\leq 1\\\frac{1}{2}(|x|+|y|)&\text{ if } & a\geq 1,\end{array}\right.
\end{align}
we get
\begin{align}\label{eq:Varnitaui2}
\text{Var}[N_{i}\  |\ N_{j}\  \text{ with } \ j\neq i]
	&\geq c\bigg(\min_{j\neq i}(|u_{i}-s_{j}|+|s_{i}-u_{i}|)\wedge(|u_{i}-u_j|+|s_{i}-u_{i}|)\bigg)^{2H}\geq  c\tau_{i}^{2H}.
\end{align}
By \eqref{eq:integralexpsigpd}, \eqref{eq:Varnitaui1} and \eqref{eq:Varnitaui2}, we get the bound
\begin{align}\label{ineq:integralgoalmiddle}
\int_{\R^{m}}e^{-\frac{1}{2}\vec{\xi}^*\Lambda\vec{\xi}}|\xi_i|^rd\vec{\xi}
  &\leq C|\Lambda|^{-\frac{1}{2}}\prod_{i=1}^{m}\tau_{i}^{-Hq_{i}}.
\end{align}
It thus remains to bound the determinant $|\Lambda|$. To this end, we use \eqref{eq:varreduceinfo} to write
\begin{align}\label{eq:Varboundfortermsofdetprev}
|\Lambda|
  &=\text{Var}[N_1]\prod_{i=2}^{p}\text{Var}[N_{i}\ |\ N_{1},\dots, N_{i-1}]\nonumber\\
	&\geq\prod_{i=1}^{p}\text{Var}[X_{u_{i}}+z_{i}(X_{s_i}-X_{u_i})+\sqrt{1-z_i^2}(X_{s_i}^{(i)}-X_{u_i}^{(i)})\ |\ X_{u_j},X_{s_{j}}\ \text{ with }\ i\leq j-1]\nonumber\\	
	&=\prod_{i=1}^{p}\bigg(\text{Var}[X_{u_{i}}+z_{i}(X_{s_i}-X_{u_i})\ |\ X_{u_j},X_{s_{j}}\ \text{ with }\ i\leq j-1]+(1-z_i^2)|s_{i}-u_{i}|^{2H}\bigg),
\end{align}
where for convenience, we have used the notation $u_0=s_0=0$. Notice that if $z_{i}=1$, then by the local-non determinism property \eqref{ineq:localnd} and the condition $s_i\leq s_{i+1}$,
\begin{align}\label{eq:Varboundfortermsofdet}
\text{Var}[X_{u_{i}}+z_{i}(X_{s_i}-X_{u_i})\ |\ X_{u_j},X_{s_{j}}\ \text{ with }\ i\leq j-1]
  &\geq c |s_{i}-s_{j}|^{2H}.
\end{align}
On the other hand, if $z_{i}=0$, then by the local-non determinism property \eqref{ineq:localnd} and the inequality \eqref{ineq:Minkowskytype}, 
\begin{multline*}
\text{Var}[X_{u_{i}}+z_{i}(X_{s_i}-X_{u_i})\ |\ X_{u_j},X_{s_{j}}\ \text{ with }\ i\leq j-1]\\
\begin{aligned}
  &\geq c \min_{1\leq j\leq i-1}((|u_{i}-s_{j}|\wedge|u_{i}-u_{j}|)^{2H}+|s_i-u_i|^{2H})\\
	&\geq c^{\prime} \min_{1\leq j\leq i-1}\big((|u_{i}-s_{j}|+|s_i-u_i|)\wedge(|u_{i}-u_{j}|+|s_i-u_i|)\big)^{2H}\\
	&\geq c^{\prime} \min_{1\leq j\leq i-1}\big(|s_i-s_j|\wedge(|s_i-u_j|)\big)^{2H}.
\end{aligned}
\end{multline*}
Thus, using the condition $u_j\leq s_j\leq s_{j+1}$, we deduce that \eqref{eq:Varboundfortermsofdet} holds as well in the case $z_i=0$.
Relation \eqref{eq:boundsstildesfixed} follows from \eqref{ineq:integralgoalmiddle}, \eqref{eq:Varboundfortermsofdetprev} and \eqref{eq:Varboundfortermsofdet}.
\end{proof}

As a consequence of the previous lemma, we have the following result:

\begin{Lemma}\label{lem:techintegralbound}
Let $i,j\in\N$, $s,\tilde{s}\geq0$ and $z,\tilde{z}\in[0,1]$ be such that $2\leq i, j\leq nt$,  $\frac{i-1}{n}\leq s\leq \frac{i}{n}$ and 
$\frac{j-1}{n}\leq \tilde{s}\leq\frac{j}{n}$. Denote by $\Sigma$  the covariance matrix of the Gaussian vector $(N_1,N_2)$, defined by
\begin{align*} 
N_1
  &:=X_{\frac{i-1}{n}}+z(X_{s}-X_{\frac{i-1}{n}})+(1-z^2)^{\frac{1}{2}}(X_{s}^{\prime}-X_{\frac{i-1}{n}}^{\prime})\nonumber\\
N_2
  &:=X_{\frac{j-1}{n}}+\tilde{z}(X_{\tilde{s}}-X_{\frac{j-1}{n}})+(1-\tilde{z}^2)^{\frac{1}{2}}(X_{\tilde{s}}^{\prime}-X_{\frac{j-1}{n}}^{\prime}),
\end{align*}
where $X^{\prime},X^{\prime\prime}$ are independent copies of $X$. Then, for all $q,\tilde{q}\geq0$, there exists a constant $C>0$ only depending on $q,\tilde{q}$, $t$ and the Hurst parameter  $H$, such that
\begin{enumerate}
\item If  $|j-i|\geq 2$, 
\begin{align}\label{eq:Lemmaauxgoal}
\int_{\R^{2}}|\xi|^q|\tilde{\xi}|^{\tilde{q}}e^{-\frac{1}{2}\vec{\xi}^*\Sigma\vec{\xi}}d\vec{\xi}
  &\leq Cn^{H(q+\tilde{q}+2)}(i\vee j)^{-H}(i\wedge j\wedge|j-i|)^{H(q+\tilde{q}+1)},
\end{align}
where $\vec{\xi}:=(\xi,\tilde{\xi})$.

\item  If  $|j-i|\leq 1$ and $H(q+\tilde{q}+1)<1$, 
\begin{align}\label{eq:Lemmaauxgoal2}
\int_{\frac{i-1}{n}}^{\frac{i}{n}}\int_{\frac{j-1}{n}}^{\frac{j}{n}}\int_{\R^{2}}|\xi|^q|\tilde{\xi}|^{\tilde{q}}e^{-\frac{1}{2}\vec{\xi}^*\Sigma\vec{\xi}}d\vec{\xi}dsd\tilde{s}
  &\leq Cn^{H(q+\tilde{q}+2)-2}(i\vee j)^{-H}.
\end{align}
\end{enumerate}

\end{Lemma}
\begin{proof}
We first prove \eqref{eq:Lemmaauxgoal} in the case $|i-j|\geq 2$. Suppose without loss of generality that $\frac{i}{n}\leq s\leq \frac{j-1}{n}\leq \tilde{s}$. By applying Lemma \ref{lem:techintegralboundasd} in the case $m=2$, $u_1=\frac{i-1}{n}$, $u_{2}=\frac{j-1}{n}$ and $s_1=s$, $s_{2}=\tilde{s}$, we have that 
\begin{align}\label{lem:techintegralboundasdapp}
\int_{\R^{2}}|\xi|^q|\tilde{\xi}|^{\tilde{q}}e^{-\frac{1}{2}\vec{\xi}^*\Sigma\vec{\xi}}d\vec{\xi}
  &\leq C (s(\tilde{s}-s))^{-H}(|s\wedge(\frac{j-1}{n}-s))^{-Hq}|\tilde{s}-\frac{i}{n}|^{-H\tilde{q}}\nonumber\\
	&\leq C  \tilde{s}^{-H}(s\wedge(\tilde{s}-s))^{-H}(|s\wedge(\frac{j-1}{n}-s))^{-Hq}|\tilde{s}-\frac{i}{n}|^{-H\tilde{q}},
\end{align}
where in the second inequality we have used the fact that for every $0\leq s\leq \tilde{s}$, 
\begin{align}\label{ineq:inequsstilde}
\tilde{s}(s\wedge(\tilde{s}-s))
  &\leq 2(s\vee (\tilde{s}-s))(s\wedge(\tilde{s}-s))
	=2s(\tilde{s}-s).
\end{align}
Notice that if $|i-j|,i,j\geq 2$, then $|i-1|\geq\frac{i}{2}$ and $j-i-1\geq \frac{1}{2}(j-i-1)$, which by \eqref{lem:techintegralboundasdapp} implies that 
\begin{align*}
\int_{\R^{2}}|\xi|^q|\tilde{\xi}|^{\tilde{q}}e^{-\frac{1}{2}\vec{\xi}^*\Sigma\vec{\xi}}d\vec{\xi}
  &\leq C (\frac{i}{n})^{-H}(\frac{j-i}{n})^{-H}(|\frac{i-1}{n}\wedge(\frac{j-1}{n}-\frac{i}{n}))^{-Hq}|\frac{j}{n}-\frac{i}{n}|^{-H\tilde{q}}\\
  &\leq C (\frac{i}{n})^{-H}(\frac{j-i}{n})^{-H}(\frac{i}{n}\wedge(\frac{j-i}{n}))^{-Hq}(\frac{j-i}{n})^{-H\tilde{q}}\\
	&= C (\frac{i\vee(j-i)}{n})^{-H}(\frac{i\wedge(j-i)}{n})^{-H}(\frac{i}{n}\wedge(\frac{j-i}{n}))^{-Hq}(\frac{j-i}{n})^{-H\tilde{q}}.
\end{align*}
Inequality \eqref{eq:Lemmaauxgoal} under the condition $|i-j|,i,j\geq 2$ follows by combining the previous inequality with the fact that $i\leq j\leq 2i\vee(j-i)$.\\

\noindent To prove \eqref{eq:Lemmaauxgoal2} under the condition $|i-j|\leq  1$,  we will assume without loss of generality that $i\leq j$, and define
\begin{align*}
I
  &:=\int_{[\frac{i-1}{n},\frac{j}{n}]^2}\int_{\R^{2}}|\xi|^q|\tilde{\xi}|^{\tilde{q}}e^{-\frac{1}{2}\vec{\xi}^*\Sigma\vec{\xi}}d\vec{\xi}d\vec{\eta}d\vec{s}.
\end{align*}
Notice that $I$ is an upper bound for the integral in the left-hand side of \eqref{eq:Lemmaauxgoal2}. Consequently, in order to prove \eqref{eq:Lemmaauxgoal2} it suffices show that
\begin{align}\label{eq:Lemmaauxgoal2sd}
I  &\leq Cn^{H(q+\tilde{q}+2)-2}(i\vee j)^{-H}.
\end{align}
Using  Lemma \ref{lem:techintegralboundasd}, we can check that 
\begin{align*}
I
  &\leq C\int_{[\frac{i-1}{n},\frac{j}{n}]^2}\Indi{\{s_1\leq s_2\}}(\tilde{s}(\tilde{s}-s))^{-H}\nonumber\\
	&\times \big((\tilde{s}-s)\wedge (s-\frac{i-1}{n})\wedge (\frac{i}{n}-s)\wedge(\tilde{s}-\frac{j-1}{n})\wedge (\frac{j}{n}-\tilde{s})\big)^{-H(q+\tilde{q})}d\vec{s}.
\end{align*}
Hence, by \eqref{ineq:inequsstilde},

\begin{align*}
I
  &\leq C\big(\frac{i}{n}\big)^{-H}\int_{[\frac{i-1}{n},\frac{j}{n}]^2}\Indi{\{s_1\leq s_2\}}\nonumber\\
	&\times \big((\tilde{s}-s)\wedge (s-\frac{i-1}{n})\wedge (\frac{i}{n}-s)\wedge(\tilde{s}-\frac{j-1}{n})\wedge (\frac{j}{n}-\tilde{s})\big)^{-H(q+\tilde{q}+1)}d\vec{s}\\
	&\leq C\big(\frac{i}{n}\big)^{-H}\int_{[\frac{i-1}{n},\frac{j}{n}]^2}\Indi{\{s_1\leq s_2\}}\bigg((\tilde{s}-s)^{-H(q+\tilde{q}+1)}+ (s-\frac{i-1}{n})^{-H(q+\tilde{q}+1)}\nonumber\\
	&+ (\frac{i}{n}-s)^{-H(q+\tilde{q}+1)}+(\tilde{s}-\frac{j-1}{n})^{-H(q+\tilde{q}+1)}+ (\frac{j}{n}-\tilde{s})^{-H(q+\tilde{q}+1)}\big)\bigg)d\vec{s}.
\end{align*}
Relation \eqref{eq:Lemmaauxgoal2sd} follows by integrating the right-hand side of the previous inequality.
\end{proof}

The next statement is one of the \blue{main technical contributions of the paper}.

\begin{Lemma}\label{Lem:existLocalderiv}
Let $\ell\in\Z$ be such that $0< H<\frac{1}{2\ell+1}$. Then, for every $t\geq0$ and $\lettrequivabien\in\R$, as $\varepsilon\to 0$ the random variables
\begin{align}\label{eq:Ltnellapprox}
L_{t,\varepsilon}^{(\ell)}(\lettrequivabien):=\int_{0}^{t}\phi_{\varepsilon}^{(\ell)}(X_{s}-\lettrequivabien)ds,
\end{align}
converge in $L^{2}(\Omega)$ to a limit $L^{(\ell)}_t(\lettrequivabien)$,  as $\varepsilon\rightarrow0$ . The limit random variable $L_{t}^{(\ell)}(\lettrequivabien)$ can be written in Fourier form as
\begin{align}\label{eq:Fourierrepmain}
L_{t}^{(\ell)}(\lettrequivabien)
  &=\int_{\R}\int_{0}^{t}(\textbf{i}\xi)^{\ell}e^{\textbf{i}\xi(X_{s}-\lettrequivabien)}dsd\xi:=\lim_{N\rightarrow\infty}\int_{-N}^{N}\int_{0}^{t}(\textbf{i}\xi)^{\ell}e^{\textbf{i}\xi(X_{s}-\lettrequivabien)}dsd\xi,
\end{align}
where the limit in the right hand side is understood in the sense of $L^{2}(\Omega).$ {\color{black}Moreover, if $\ell\geq 1$, we have that 
\begin{align}\label{eq:LTderivativeapprox}
L_{t}^{(\ell)}(\lettrequivabien)
  &=\lim_{h\rightarrow 0}\frac{1}{h}(L_{t}^{(\ell-1)}(\lettrequivabien+h)-L_{t}^{(\ell-1)}(\lettrequivabien)),
\end{align}
where the limit is understood in the $L^{2}(\Omega)$-sense.} In addition, for fixed $\lettrequivabien$ and for all $\gamma<1-H(\ell+1)$ and $p>1$, the process $\{L_{t}^{\ell}(\lettrequivabien)\ \processsymb\ t\geq0\}$ obtained as the pointwise limit in $L^{2}(\Omega)$ of $\{L_{t,n}^{\ell}\ \processsymb\ t\geq0\}$, has a modification with H\"older continuous trajectories (in the variable $t$) of order $\gamma$ and 
\begin{align}\label{eq:momentsofincLt}
\|L_{v}^{(\ell)}(\lettrequivabien)-L_{u}^{(\ell)}(\lettrequivabien)\|_{L^{2p}(\Omega)}^{2p}
  &\leq C|u-v|^{2p(1-H(\ell+1))},
\end{align}
for all $u<v$, where $C$ is a constant independent of $u$ and $v$.
\end{Lemma}

\begin{Remark} {\rm The arguments used in the proofs of Lemma \ref{Lem:existLocalderiv} and Lemma \ref{lem:nonexistenceLell} also show that, for every fixed $t>0$ and $\lettrequivabien \in \R$, the variance of $L_{t}^{(\ell)}(\lettrequivabien)$ is strictly positive.
}
\end{Remark}

\begin{proof}[Proof of Lemma \ref{Lem:existLocalderiv}]
%Notice that for every $m,n\geq 1$, we have that 
%\begin{align*}
%\|L_{t,n}^{(\ell)}(\lettrequivabien)-L_{t,m}^{(\ell)}(\lettrequivabien)	\|_{L^{2}(\Omega)}^2
%   &=\E[(L_{t,n}^{(\ell)}(\lettrequivabien))^2]+\E[(L_{t,m}^{(\ell)}(\lettrequivabien))^2]-2\E[L_{t,n}^{(\ell)}(\lettrequivabien)L_{t,m}^{(\ell)}(\lettrequivabien)],
%\end{align*}
It suffices to show that, for every pair of positive sequences $\{\varepsilon_n, \eta_n\}$ such that $\varepsilon_n, \eta_n\to 0$, one has that, as $n\to\infty$, the sequence
\begin{align*}
\E[L_{t,\varepsilon_n}^{(\ell)}(\lettrequivabien)L_{t,\eta_n}^{(\ell)}(\lettrequivabien)]
\end{align*}
converges to a finite limit, independent of the choice of $\{\varepsilon_n, \eta_n\}$. To this end, we first write $\phi_{\varepsilon}^{(\ell)}$, for $\varepsilon>0$ as 
\begin{align}\label{eq:phifourierrep}
\phi_{\varepsilon}^{(\ell)}(x)
  &=\frac{1}{2\pi}\int_{\R}(\textbf{i}\xi)^{\ell}e^{\textbf{i}\xi  x-\frac{\varepsilon}{2}\xi^2}d\xi.
\end{align}
For $\vec{s}=(s,\tilde{s})\in\R_{+}^2$, denote by $\Sigma(\vec{s})$ the covariance matrix of $(X_{s},X_{\tilde{s}})$ and by $A(n )$ the diagonal matrix with $A_{1,1}=\varepsilon_n$ and $A_{2,2}=\eta_n$. From \eqref{eq:phifourierrep}, it follows that
\begin{align}\label{eq:LtnLtmfirstrel}
\E[L_{t,\varepsilon_n}^{(\ell)}(\lettrequivabien)L_{t,\eta_n}^{(\ell)}(\lettrequivabien)]
  &=\frac{(-1)^{\ell}}{4\pi^2}\int_{[0,t]^2}\int_{\R^2}( \xi\tilde{\xi})^{\ell}e^{-\textbf{i}\xi\lettrequivabien-\textbf{i}\tilde{\xi}\lettrequivabien}e^{-\frac{1}{2}\vec{\xi}^*(\Sigma(\vec{s})+A(n))\vec{\xi}}d\vec{\xi}d\vec{s},
\end{align}
where $\vec{\xi}:=(\xi,\tilde{\xi})$, $\vec{s}:=(s,\tilde{s})$. Notice that the integrand in \eqref{eq:LtnLtmfirstrel} converges as $n\rightarrow\infty$ and satisfies
\begin{align*}
|( \xi\tilde{\xi})^{\ell}e^{-\textbf{i}\xi\lettrequivabien-\textbf{i}\tilde{\xi}\lettrequivabien}e^{-\frac{1}{2}\vec{\xi}^*(\Sigma(\vec{s})+A(n))\vec{\xi}}|
  &\leq |\xi\tilde{\xi}|^{\ell}e^{-\frac{1}{2}\vec{\xi}^*\Sigma(\vec{s})\vec{\xi}}.
\end{align*}
Moreover, by applying  inequality \eqref{eq:boundsstildesfixed} in Lemma \ref{lem:techintegralbound}, we have that
\begin{align}\label{ineq:integralofabsgaussker}
\int_{[0,t]^2}\int_{\R^2}|\xi\tilde{\xi}|^{\ell}e^{-\frac{1}{2}\vec{\xi}^*\Sigma(\vec{s})\vec{\xi}}d\vec{\xi}d\vec{s}
  &\leq C\int_{[0,t]^2}(s\vee \tilde{s})^{-H}(s\wedge\tilde{s}\wedge|s-\tilde{s}|)^{-H(2\ell+1)}d\vec{s}.
\end{align} 
The right hand side is of the previous inequality is finite due to the condition $H(2\ell+1)<1$, and thus, by the dominated convergence theorem, we conclude that $\E[L_{t,\varepsilon_n}^{(\ell)}(\lettrequivabien)L_{t,\eta_n}^{(\ell)}(\lettrequivabien)]$ 
converges to a limit independent of the choice of $\{\varepsilon_n, \eta_n\}$. \\

To show \eqref{eq:Fourierrepmain}, we define $\varepsilon_{M}:=M^{-2}$, for $M>0$. By applying the Fourier inversion formula to the Gaussian kernel $\phi_{\varepsilon_{M}}$, we can easily check that
\begin{align*}
\int_{0}^{t}\phi_{\varepsilon_{M}}^{(\ell)}(X_{s})ds
  &=\frac{1}{2\pi}\int_{\R}\int_{0}^{t}(\textbf{i}\xi)^{\ell}e^{\textbf{i}\xi(X_{s}-\lettrequivabien)-\frac{1}{2}\varepsilon_{M}\xi^2}dsd\xi,
\end{align*}
and thus,
\begin{align}\label{eq:espectralapproxlocaltime}
\frac{1}{2\pi}\int_{[-M,M]}\int_{0}^{t}(\textbf{i}\xi)^{\ell}e^{\textbf{i}\xi(X_{s}-\lettrequivabien)}dsd\xi
  &=T_{1}^{(n)}+T_{2}^{(n)}+\int_{0}^{t}\phi_{\varepsilon_{M}}^{(\ell)}(X_{s})ds,
\end{align}
where 
\begin{align*}
T_{1}^{(M)}
  &:=\frac{1}{2\pi}\int_{[-M,M]}\int_{0}^{t}(\textbf{i}\xi)^{\ell}e^{\textbf{i}\xi(X_{s}-\lettrequivabien)}(1-e^{-\frac{1}{2} \varepsilon_{M}\xi^{2}})dsd\xi\\
T_{2}^{(M)}
	&:=\frac{1}{2\pi}\int_{\R}\int_{0}^{t}(\textbf{i}\xi)^{\ell}e^{\textbf{i}\xi(X_{s}-\lettrequivabien)-\frac{1}{2}\sigma_{M}^2\xi^2}\Indi{[-M,M]^c}(\xi)dsd\xi.
\end{align*}
Since the third term in the right hand side of \eqref{eq:espectralapproxlocaltime} converges to $L_{t}^{(\ell)}(z)$  as $M\rightarrow\infty$ , it suffices to show that $T_{1}^{(M)}$ and $T_{2}^{(M)}$ converge to zero in $L^{2}(\Omega)$. Proceeding as in the proof of \eqref{eq:LtnLtmfirstrel}, we can easily show that 
\begin{align*}
\|T_{1}^{(M)}\|_{L^{2}(\Omega)}^2
  &=\frac{(-1)^{\ell}}{4\pi^2}\int_{[0,t]^2}\int_{[-M,M]^2}( \xi\tilde{\xi})^{\ell}e^{-\textbf{i}\xi\lettrequivabien-\textbf{i}\tilde{\xi}\lettrequivabien}e^{-\frac{1}{2}\vec{\xi}^*(\Sigma(\vec{s})+\sigma_{M}^2)\vec{\xi}}(1-e^{-\frac{1}{2} \varepsilon_{M}\xi^{2}})
	(1-e^{-\frac{1}{2} \varepsilon_{M}\tilde{\xi}^{2}})d\vec{\xi}d\vec{s}.
\end{align*}
Thus, using the fact that $\varepsilon_M|\xi|^2\leq \frac{1}{M}$ for all $\xi\in[-M,M]$, we  obtain
\begin{align*}
\|T_{1}^{(M)}\|_{L^{2}(\Omega)}^2
  &\leq \frac{1}{M^2}\int_{[0,t]^2}\int_{\R^2}|\xi\tilde{\xi}|^{\ell}e^{-\frac{1}{2}\vec{\xi}^*\Sigma(\vec{s})\vec{\xi}}d\vec{\xi}d\vec{s},
\end{align*}
which by \eqref{ineq:integralofabsgaussker}, leads to $\lim_{M\rightarrow\infty}\|T_{1}^{(M)}\|_{L^{2}(\Omega)}^2=0$. By following similar arguments, we can show that 
\begin{align*}
\|T_{2}^{(M)}\|_{L^{2}(\Omega)}^2
  &\leq \int_{[0,t]^2}\int_{\R^2}|\xi\tilde{\xi}|^{\ell}e^{-\frac{1}{2}\vec{\xi}^*\Sigma(\vec{s})\vec{\xi}}\Indi{[-M,M]}(\xi)\Indi{[-M,M]}(\tilde{\xi})d\vec{\xi}d\vec{s}.
\end{align*}
The integrand of the right-hand side is bounded by $e^{-\frac{1}{2}\vec{\xi}^*\Sigma(\vec{s})\vec{\xi}}$, which by \eqref{ineq:integralofabsgaussker} is integrable over $\vec{s}\in[0,t]^2$ and $\vec{\xi}\in\R^2$. Thus, by the dominated convergence theorem, $\lim_{M\rightarrow\infty}\|T_{2}^{(M)}\|_{L^{2}(\Omega)}^2=0$, as required. This finishes the proof of \eqref{eq:Fourierrepmain}.\\

\noindent Next we prove \eqref{eq:LTderivativeapprox} in the case $\ell\geq 1$. Using the representation \eqref{eq:Fourierrepmain}, we have that 
\begin{align*}
L_{t}^{(\ell-1)}(\lettrequivabien+h)-L_{t}^{(\ell-1)}(\lettrequivabien)-hL_{t}^{(\ell)}(\lettrequivabien)
  &=\int_{\R}\int_{0}^{t}(\textbf{i}\xi)^{\ell-1}e^{\textbf{i}\xi (X_{s}-\lambda)}(e^{-\textbf{i}\xi h}-1-h\textbf{i}\xi)dsd\xi.
\end{align*}
Proceeding as before, we can bound this $L^{2}(\Omega)$-norm as follows
\begin{align}\label{Ltdiffapproxv1}
\|L_{t}^{(\ell-1)}(\lettrequivabien+h)-L_{t}^{(\ell-1)}(\lettrequivabien)-hL_{t}^{(\ell)}(\lettrequivabien)\|_{L^{2}(\Omega)}^2
  &\leq \int_{[0,t]^2}\int_{\R^2}|\xi\tilde{\xi}|^{\ell-1}e^{-\frac{1}{2}\vec{\xi}^*\Sigma(\vec{s})\vec{\xi}}\\
	&\times |(e^{-\textbf{i}\xi h}-1-h\textbf{i}\xi)(e^{-\textbf{i}\tilde{\xi} h}-1-h\textbf{i}\tilde{\xi})|d\vec{\xi}d\vec{s}\nonumber.
\end{align}
Let $0<\varepsilon<1$ be such that $H(2\ell+1+2\varepsilon)<1$. Using the inequality 
\begin{align*}
|e^{-\textbf{i}\eta}-1-\textbf{i}\eta|
  &\leq 3|\eta|^{1+\varepsilon},
\end{align*}
which is valid for all $\eta\in\R$, we conclude from \eqref{Ltdiffapproxv1} that 
\begin{align}\label{eq:boundsstildesfixedextra}
\|L_{t}^{(\ell-1)}(\lettrequivabien+h)-L_{t}^{(\ell-1)}(\lettrequivabien)-hL_{t}^{(\ell)}(\lettrequivabien)\|_{L^{2}(\Omega)}^2
  &\leq 9h^{2+2\varepsilon}\int_{[0,t]^2}\int_{\R^2}|\xi\tilde{\xi}|^{\ell+\varepsilon}e^{-\frac{1}{2}\vec{\xi}^*\Sigma(\vec{s})\vec{\xi}}d\vec{\xi}d\vec{s}.
\end{align}
By inequality \eqref{eq:boundsstildesfixed} in Lemma \ref{lem:techintegralbound}, there exists a constant $C>0$, such that 
\begin{align*}
\int_{\R^2}|\xi\tilde{\xi}|^{\ell+\varepsilon}e^{-\frac{1}{2}\vec{\xi}^*\Sigma(\vec{s})\vec{\xi}}d\vec{\xi}d\vec{s}
  &\leq C\int_{[0,t]^2}(s\vee \tilde{s})^{-H}(s\wedge\tilde{s}\wedge|s-\tilde{s}|)^{-H(2\ell+2\varepsilon+1)}d\vec{s},
\end{align*}
for some constant $C>0$. Thus, by the condition $H(2\ell+1+2\varepsilon)<1$, we have that the right hand side of \eqref{eq:boundsstildesfixedextra} is finite, which in turn implies that
\begin{align*}
\lim_{h\rightarrow0}\|\frac{1}{h}(L_{t}^{(\ell-1)}(\lettrequivabien+h)-L_{t}^{(\ell-1)}(\lettrequivabien)-hL_{t}^{(\ell)}(\lettrequivabien))\|_{L^{2}(\Omega)}^2\rightarrow0,
\end{align*}
as required. This finishes the proof of \eqref{eq:LTderivativeapprox}.\\

\noindent Next we prove \eqref{eq:momentsofincLt}. To this end, we select a positive sequence $\{ \varepsilon_{n}\}_{n\geq 1}$ converging to zero and, for $t>0$, we write $L_{t,n}^{(\ell)}(\lettrequivabien) := L_{s,\varepsilon_n}^{(\ell)}(\lettrequivabien)$ for simplicity.  We use the convergence $\lim_nL_{s,n}^{(\ell)}(\lettrequivabien)=L_s^{(\ell)}(\lettrequivabien)$, valid for all $s>0$, to write 
\begin{align}\label{ineq:Luvsitsapprox}
\|L_{u}^{(\ell)}(\lettrequivabien)-L_{v}^{(\ell)}(\lettrequivabien)\|_{L^{2p}(\Omega)}^{2p}
   &\leq \limsup_n\|L_{u,n}^{(\ell)}(\lettrequivabien)-L_{v,n}^{(\ell)}(\lettrequivabien)\|_{L^{2p}(\Omega)}^{2p}.
\end{align}
Assume without loss of generality that $u\geq v$. The $L^{2}(\Omega)$-norm in the right hand side of the previous inequality can be estimated by using the identity \eqref{eq:phifourierrep}, which allows us to write
\begin{align}\label{eq:Lellunincrements}
L_{u,n}^{(\ell)}(\lettrequivabien)-L_{v,n}^{(\ell)}(\lettrequivabien)
  &=\frac{\textbf{i}^{\ell}}{2\pi}\int_{[v,u]}\int_{\R}\xi^{\ell}e^{\textbf{i}\xi(X_s-\lettrequivabien)-\frac{1}{2n}\xi^2}d\xi ds.
\end{align}
This leads to 
\begin{align*}
\|L_{u,n}^{(\ell)}(\lettrequivabien)-L_{v,n}^{(\ell)}(\lettrequivabien)\|_{L^{2p}(\Omega)}^{2p}
  &= \frac{(-1)^{\ell p}}{4\pi^2}\int_{[v,u]^{2p}}\int_{\R^{2p}}\E[\prod_{i=1}^{2p}\xi_i^{\ell}e^{\textbf{i}\xi_i (X_{s_i}-\lettrequivabien)-\frac{1}{2n} \xi_i^2}]d\vec{\xi} d\vec{s},
\end{align*}
where $\vec{\xi}=(\xi_1,\dots, \xi_{2p})$ and $\vec{s}=(s_1,\dots, s_{2p})$. From here we can easily obtain 
\begin{align*}
\|L_{u,n}^{(\ell)}(\lettrequivabien)-L_{v,n}^{(\ell)}(\lettrequivabien)\|_{L^{2p}(\Omega)}^{2p}
  &\leq \int_{[v,u]^{2p}}\int_{\R^2} e^{-\frac{1}{2}\vec{\xi}^{*}\Sigma(\vec{s})\vec{\xi}}\prod_{i=1}^{2p}|\xi_i|^{\ell}d\vec{\xi} d\vec{s},
\end{align*}
where $\Sigma(\vec{s})$ denotes the covariance matrix of $(X_{s_1},\dots, X_{s_{2p}})$. Therefore, by applying Lemma \ref{lem:techintegralboundasd} with $m=2p$, $u_{i}=s_i$ and $z_{i}=1$, we get
\begin{align*}
\|L_{u,n}^{(\ell)}(\lettrequivabien)-L_{v,n}^{(\ell)}(\lettrequivabien)\|_{L^{2p}(\Omega)}^{2p}
  &\leq C\int_{[v,u]^{2p}}(s_{1}(s_{2p}-s_{2p-1}))^{-H}((s_{1}\wedge(s_{2}-s_1))(s_{2p}-s_{2p-1}))^{-H\ell}\\
	&\times\prod_{i=2}^{2p-1}((s_{i}-s_{i-1})^{-H}((s_{i}-s_{i-1})\wedge(s_{i+1}-s_{i}))^{-H\ell})d\vec{s}.
\end{align*} 
Thus, by changing the coordinates $(s_{1},\dots, s_{2p})$ by $(\tau_{1}:=s_{1},\tau_2:=s_{2}-s_{1},\dots, \tau_{2p}:=s_{2p}-s_{2p-1})$, we obtain 
\begin{align*}
\|L_{u,n}^{(\ell)}(\lettrequivabien)-L_{v,n}^{(\ell)}(\lettrequivabien)\|_{L^{2p}(\Omega)}^{2p}
  &\leq C\int_{v}^{u}\int_{[0,v-u]^{2p-1}}\bigg(\prod_{i=1}^{2p}\tau_{i}^{-H}\bigg)\bigg(\tau_{2p}^{-H\ell}
	\prod_{i=1}^{2p-1}(\tau_i\wedge\tau_{i+1})^{-H\ell}\bigg)d\tau_{2p-1}\cdots d\tau_{1}.
\end{align*}
By using the bound $\tau_{1}^{-H}\leq (v-u)^{-H}\Indi{\{\tau_{1}\geq v-u\}}+\tau_{1}^{-H}\Indi{\{\tau_{1}\leq v-u \}}$ in the integral above, we obtain 
\begin{multline}
\|L_{u,n}^{(\ell)}(\lettrequivabien)-L_{v,n}^{(\ell)}(\lettrequivabien)\|_{L^{2p}(\Omega)}^{2p}\\
\begin{aligned}
  &\leq C(v-u)^{-H}\int_{v}^{u}\int_{[0,v-u]^{2p-1}}\bigg(\prod_{i=2}^{2p}\tau_{i}^{-H}\bigg)\bigg(\tau_{2p}^{-H\ell}
	\prod_{i=1}^{2p-1}(\tau_i\wedge\tau_{i+1})^{-H\ell}\bigg)d\tau_{2p-1}\cdots d\tau_{1}\\
	&+C\int_{[0,v-u]^{2p}}\bigg(\prod_{i=1}^{2p}\tau_{i}^{-H}\bigg)\bigg(\tau_{2p}^{-H\ell}
	\prod_{i=1}^{2p-1}(\tau_i\wedge\tau_{i+1})^{-H\ell}\bigg)d\tau_{2p-1}\cdots d\tau_{1},
\end{aligned}
\end{multline}
which in turn leads to 
\begin{multline*}
\|L_{u,n}^{(\ell)}(\lettrequivabien)-L_{v,n}^{(\ell)}(\lettrequivabien)\|_{L^{2p}(\Omega)}^{2p}\\
\begin{aligned}
	&\leq C(v-u)^{-H}\int_{v}^{u}\int_{[0,v-u]^{2p-1}}\bigg(\prod_{i=2}^{2p}\tau_{i}^{-H}\bigg)\bigg(\tau_{2p}^{-H\ell}
	\prod_{i=1}^{2p-1}(\tau_i^{-H\ell}+\tau_{i+1}^{-H\ell})\bigg)d\tau_{2p-1}\cdots d\tau_{1}\\
	&+C\int_{[0,v-u]^{2p}}\bigg(\prod_{i=1}^{2p}\tau_{i}^{-H}\bigg)\bigg(\tau_{2p}^{-H\ell}
	\prod_{i=1}^{2p-1}(\tau_i^{-H\ell}+\tau_{i+1}^{-H\ell})\bigg)d\tau_{2p-1}\cdots d\tau_{1}.
\end{aligned}
\end{multline*}
By expanding the product above, we can write
\begin{multline}\label{ineq:Llincrementsbound}
\|L_{u,n}^{(\ell)}(\lettrequivabien)-L_{v,n}^{(\ell)}(\lettrequivabien)\|_{L^{2p}(\Omega)}^{2p}\\
\begin{aligned}
	&\leq C(v-u)^{-H}\int_{v}^{u}\int_{[0,v-u]^{2p-1}}\sum_{\substack{\alpha_{2},\dots, \alpha_{2p}\in\{0,1,2\}\\\sum_{i=1}^{2p}\alpha_i=2p}}\tau_{1}^{-H\ell\alpha_1}\prod_{i=2}^{2p}\tau_{i}^{-H(1+\alpha_i\ell)} d\tau_{2p-1}\cdots d\tau_{1}\\
	&+C \int_{[0,v-u]^{2p}}\sum_{\substack{\alpha_{1},\dots, \alpha_{2p}\in\{0,1,2\}\\\sum_{i=1}^{2p}\alpha_i=2p}}\prod_{i=1}^{2p}\tau_{i}^{-H(1+\alpha_i\ell)} d\tau_{2p-1}\cdots d\tau_{1}.
\end{aligned}
\end{multline}
Due to the condition $H(2\ell+1)<1$, the integrals in the right-hand side of \eqref{ineq:Llincrementsbound} are finite and can be computed explicitly, leading to
\begin{align}\label{ineq:Luapproxbound}
\|L_{u,n}^{(\ell)}(\lettrequivabien)-L_{v,n}^{(\ell)}(\lettrequivabien)\|_{L^{2p}(\Omega)}^{2p}
	&\leq C (v-u)^{2p(1-H(\ell+1))}.
\end{align}
Relation \eqref{eq:momentsofincLt} then follows from \eqref{ineq:Luvsitsapprox} and \eqref{ineq:Luapproxbound}.\\

\noindent Finally, we prove that for all $\gamma<1-H(2\ell+1)$, the process $L_{t,n}^{(\ell)}(\lettrequivabien)$ has a H\"older continuous modification with exponent $\gamma$. By the condition $H<\frac{1}{2\ell+1}$, we have that for all $p>\frac{1}{2(1-H(\ell+1))}$, the exponent in the right hand side of \eqref{eq:momentsofincLt} is strictly bigger than 1. The H\"older continuity of $X$ then follows from the Kolmogorov continuity criterion.
\end{proof}

The next statement is the announced counterpart to Lemma \ref{Lem:existLocalderiv}, showing that the conditions on $H$ appearing therein  are  sharp.

\begin{Lemma}\label{lem:nonexistenceLell}
Let $\ell\geq 1$ be an even integer satisfying $H\geq \frac{1}{2\ell+1}$. Then,% there exists $\ell_0\in\{0,\dots, \ell\}$, such that 
\begin{align}\label{eq:nonexistenceLell}
\lim_{n\rightarrow\infty}\|L_{t,n}^{(\ell)}(0)\|_{L^2(\Omega)}=\infty.
\end{align}
\end{Lemma}
\begin{proof}
%We proceed by contradiction. Suppose that $L_{t,n}^{(k)}$ converges in $L^{2}(\Omega)$ for all $k=0,\dots, \ell$. 
%
%%Let $c_{0},\dots, c_{\ell}$ be the coefficients of Hermite polynomial of order $\ell$ rescaled by $-\textbf{i}$, given by 
%$$H_{\ell}(x)
  %=\sum_{k=0}^{\ell}c_k(-\textbf{i})^{k}x^{k}.$$
%Then, if $\|L_{t,n}^{(p)}(0)\|_{L^2(\Omega)}<\infty$ for all $p=0,\dots, \ell$, we have that 
%\begin{align}
%\sup_{n\geq 1} \|\sum_{k=0}^{\ell} c_{k}L_{t,n}^{(k)}(0)\|_{L^2(\Omega)}<\infty,
%\end{align}
Using \eqref{eq:Lellunincrements}, we can easily show that
%\begin{align*}
%\sum_{k=0}^{\ell} c_{k}L_{t,n}^{(k)}(0)
  %&=\frac{1}{2\pi}\int_{0}^{t}\int_{\R}H_{\ell}(\xi)e^{\textbf{i}\xi X_s-\frac{1}{2n}\xi^2}d\xi ds.
%\end{align*}
%From here it easily follows that  if $\Lambda^{(n)}(s,\tilde{s})$ denotes the covariance matrix of $(X_{s},X_{\tilde{s}})$, 
\begin{align}\label{eq:powerlemmaapp}
\| L_{t,n}^{(\ell)}(0)\|_{L^{2}(\Omega)}^2
  &=\frac{(-1)^{\ell}}{4\pi^2}\int_{[0,t]^2}\int_{\R^2}(\xi\tilde{\xi})^{\ell}e^{-\frac{1}{2}\vec{\xi}^*\Lambda^{(n)}(s,\tilde{s})\vec{\xi}}d\vec{\xi} d\vec{s},
\end{align}
where $\Lambda^{(n)}(s,\tilde{s})$ denotes the covariance matrix of $(X_{s},X_{\tilde{s}})$. In order to compute the integral over $\vec{\xi}$ appearing in the right-hand side, we proceed as follows. First we decompose the power function $x^{\ell}$ in the form
$$x^{\ell}
  =\sum_{q=0}^{ \lfloor \ell/2\rfloor }c_{q,\ell}H_{\ell-2q}(x),$$
where $c_{q,\ell}:=\frac{\ell!}{q!(\ell-q)!}$ and $H_{q}$ denotes the $q$-th Hermite polynomial. Therefore, if we define $\tilde{\Lambda}^{(n)}(s,\tilde{s})$ as the inverse of $\Lambda^{(n)}(s,\tilde{s})$ and $\Phi_{\tilde{\Lambda}^{(n)}(s,\tilde{s})}$ as the centered Gaussian kernel with covariance $\tilde{\Lambda}^{(n)}(s,\tilde{s})$,  we have that 
\begin{multline*}
\int_{[0,t]^2}\int_{\R^2}(\xi\tilde{\xi})^{\ell}e^{-\frac{1}{2}\vec{\xi}^*\Lambda^{(n)}(s,\tilde{s})\vec{\xi}}d\vec{\xi}\\
\begin{aligned} 
  &=|\tilde{\Lambda}^{(n)}(s,\tilde{s})|^{\frac{1}{2}}\int_{[0,t]^2}\int_{\R^2}(\tilde{\Lambda}_{1,1}^{(n)}(s,\tilde{s})\tilde{\Lambda}_{2,2}^{(n)}(s,\tilde{s}))^{\frac{\ell}{2}}\\
	&\times \sum_{q,\tilde{q}=0}^{ \lfloor \ell/2\rfloor }c_{q,\ell}c_{\tilde{q},\ell}H_{\ell-2q}\bigg(\frac{\xi}{\sqrt{\tilde{\Lambda}_{1,1}^{(n)}(s,\tilde{s})}}\bigg)H_{\ell-2\tilde{q}}\bigg(\frac{\tilde{\xi}}{\sqrt{\tilde{\Lambda}_{1,1}^{(n)}(s,\tilde{s})}}\bigg) \Phi_{\tilde{\Lambda}^{(n)}(s,\tilde{s})}(\vec{\xi})d\vec{\xi}d\vec{s},
\end{aligned}
\end{multline*}
which by \eqref{eq:Hqorthog}, implies that 
\begin{multline*}
\int_{[0,t]^2}\int_{\R^2}(\xi\tilde{\xi})^{\ell}e^{-\frac{1}{2}\vec{\xi}^*\Lambda^{(n)}(s,\tilde{s})\vec{\xi}}d\vec{\xi}\\
  =|\tilde{\Lambda}^{(n)}(s,\tilde{s})|^{\frac{1}{2}}\int_{[0,t]^2} (\tilde{\Lambda}_{1,1}^{(n)}(s,\tilde{s})\tilde{\Lambda}_{2,2}^{(n)}(s,\tilde{s}))^{\frac{\ell}{2}} \sum_{q=0}^{ \lfloor \ell/2\rfloor }c_{q,\ell}\bigg(\frac{\tilde{\Lambda}_{1,2}^{(n)}(s,\tilde{s})}{\sqrt{\tilde{\Lambda}_{1,1}^{(n)}(s,\tilde{s}) \tilde{\Lambda}_{2,2}^{(n)}(s,\tilde{s})}}\bigg)^{\ell-2q}d\vec{s}.
\end{multline*}
Therefore, using the fact that
\begin{align*}
\tilde{\Lambda}_{2,2}^{(n)}(s,\tilde{s})
  &=\frac{1}{|\Lambda^{(n)}(s,\tilde{s})|}\left(\begin{array}{cc}\tilde{s}^{2H} +\frac{1}{n}& -R(s,\tilde{s})\\-R(s,\tilde{s})& s^{2H}+\frac{1}{n} \end{array}\right),
\end{align*}
we obtain  
\begin{multline*}
(-1)^{\ell}\int_{[0,t]^2}\int_{\R^2}(\xi\tilde{\xi})^{\ell}e^{-\frac{1}{2}\vec{\xi}^*\Lambda^{(n)}(s,\tilde{s})\vec{\xi}}d\vec{\xi}\\
  =\int_{[0,t]^2}|\Lambda^{(n)}(s,\tilde{s})|^{-\frac{1}{2}-\ell}((s^{H}+\frac{1}{n})(\tilde{s}^{H}+\frac{1}{n}))^{\ell}\sum_{q=0}^{ \lfloor \ell/2\rfloor }c_{q,\ell}\bigg(\frac{R(s,\tilde{s})}{(s^{H}+\frac{1}{n})(\tilde{s}^{H}+\frac{1}{n})}\bigg)^{\ell-2q}d\vec{s}.
\end{multline*}
Using the previous identity as well as the fact that $R(s,\tilde{s})\geq0$ and $c_{0,\ell}=1$, we conclude that 
\begin{multline*}
(-1)^{\ell}\int_{[0,t]^2}\int_{\R^2}(\xi\tilde{\xi})^{\ell}e^{-\frac{1}{2}\vec{\xi}^*\Lambda^{(n)}(s,\tilde{s})\vec{\xi}}d\vec{\xi}\\
  \geq \int_{[0,t]^2}|\Lambda^{(n)}(s,\tilde{s})|^{-\frac{1}{2}-\ell}((s^{H}+\frac{1}{n})(\tilde{s}^{H}+\frac{1}{n}))^{\ell} \bigg(\frac{R(s,\tilde{s})}{(s^{H}+\frac{1}{n})(\tilde{s}^{H}+\frac{1}{n})}\bigg)^{\ell}d\vec{s},
\end{multline*}
which by \eqref{eq:powerlemmaapp}, leads to 

%can rewrite \eqref{eq:Hermitemix} as
\begin{align}\label{eq:suofhermitelocal}
\| L_{t,n}^{(\ell)}(\lettrequivabien)\|_{L^{2}(\Omega)}^2
  &\geq\frac{1}{4\pi^2}\int_{[0,t]^2}|\Lambda^{(n)}(s,\tilde{s})|^{-\frac{1}{2}}\bigg(\frac{R(s,\tilde{s})}{|\Lambda^{(n)}(s,\tilde{s})|}\bigg)^{\ell}d\vec{s}.
\end{align}
Define $\Sigma(s,\tilde{s})$ as the covariance matrix of $(X_{s},X_{\tilde{s}})$. Combining \eqref{eq:suofhermitelocal} with Fatou's lemma, and the fact that $R(s,\tilde{s})\geq0$ for all $s,\tilde{s}\geq0$, we deduce that
\begin{align}\label{eq:suofhermitelocal2}
0\leq \int_{[0,t]^2}|\Sigma(s,\tilde{s})|^{-\frac{1}{2}}\bigg(\frac{R(s,\tilde{s})}{|\Sigma(s,\tilde{s})|}\bigg)^{\ell}d\vec{s}
  &\leq\limsup_n\int_{[0,t]^2}|\Lambda^{(n)}(s,\tilde{s})|^{-\frac{1}{2}}\bigg(\frac{R(s,\tilde{s})}{|\Lambda^{(n)}(s,\tilde{s})|}\bigg)^{\ell}d\vec{s}\nonumber\\
	&\leq 4\pi^2\limsup_{n}\| L_{t,n}^{(\ell)}(0)\|_{L^{2}(\Omega)}^2.
\end{align}
Thus, to prove \eqref{eq:nonexistenceLell}, it suffices to find a divergent lower bound for the integral in the left-hand side. To this end, define  
$$\mathcal{C} := \{(s,\tilde{s})\in[0,t]^2\ \processsymb\ \frac{t}{2}\leq s\leq \tilde{s}\leq s+\frac{3t}{2}\},$$
and notice that by \eqref{eq:varreduceinfo}, for every $(s,\tilde{s})\in\mathcal{C}$, we have that $R(s,\tilde{s})\geq \frac{t^{2H}}{4}$ and 
\begin{align*}
|\Sigma(s,\tilde{s})|
  &=\text{Var}[X_{s}]\text{Var}[X_{\tilde{s}}\ |\ X_{s}]
	=\text{Var}[X_{s}]\text{Var}[X_{\tilde{s}}-X_{s}\ |\ X_{s}]
	\leq t^{2H}(\tilde{s}-s)^{2H}.
\end{align*}
Thus, by \eqref{eq:suofhermitelocal2},
\begin{align}\label{eq:nonexistenceLellprev}
\int_{\mathcal{C}}t^{-H}4^{-\ell}(\tilde{s}-s)^{-H(2\ell+1)}d\vec{s}
 &\leq 4\pi^2\limsup_{n}\|L_{t,n}^{(\ell)}(0)\|_{L^{2}(\Omega)}^2.
\end{align}  
Notice that there exists a small ball contained in $\mathcal{C}$ with center in the diagonal of $[0,t]^2$, which forces the integral in the left-hand side to be divergent due to the condition $H\geq \frac{1}{2\ell+1}$. Relation \eqref{eq:nonexistenceLell} then follows from  \eqref{eq:nonexistenceLellprev}.
	
\end{proof}

\end{document}